\documentclass[12pt]{amsart}
\usepackage[centertags]{amsmath}

\usepackage{amsfonts}
\usepackage{amssymb}
\usepackage{amsthm}
\usepackage{amsmath}
\usepackage{amscd}
\usepackage{mathrsfs}
\usepackage{enumerate}
\usepackage{ulem}
\usepackage{fontenc}
\usepackage[all]{xy}
\usepackage{array}
\usepackage[english]{babel}
\usepackage{graphicx}

\def\CC {{\mathbb C}}     
\def\GG {{\mathbb G}}     
\def\NN {{\mathbb N}}     
\def\OO {{\mathbb O}}     
\def\PP {{\mathbb P}}     
\def\SS {{\mathbb S}}     
\def\ZZ {{\mathbb Z}}     

\def\eps {\varepsilon}

\def\hw  {\hookrightarrow}

\def\mc {\mathcal}

\def\ol  {\overline}

\def\tst {\Longleftrightarrow}
\def\ul  {\underline}

\newcommand{\huw}{\mathrel{\rotatebox[origin=c]{90}{$\hookrightarrow$}}}
\newcommand{\hdw}{\mathrel{\rotatebox[origin=c]{-90}{$\hookrightarrow$}}}

\newtheorem{theorem}{Theorem}[section]
\newtheorem{lemma}[theorem]{Lemma}

\newtheorem{prop}[theorem]{Proposition}

\newtheorem{coro}[theorem]{Corollary}
\newtheorem{rem}{Remark}[section]
\newtheorem{definition}{Definition}[section]

\newtheorem{example}{Example}[section]

\makeatletter
\def\fps@figure{htbp}
\makeatother
\makeindex

\usepackage{pict2e}

\makeatletter
\DeclareRobustCommand{\loplus}{\mathbin{\mathpalette\dog@lsemi{+}}}
\DeclareRobustCommand{\lotimes}{\mathbin{\mathpalette\dog@lsemi{\times}}}
\DeclareRobustCommand{\roplus}{\mathbin{\mathpalette\dog@rsemi{+}}}
\DeclareRobustCommand{\rotimes}{\mathbin{\mathpalette\dog@rsemi{\times}}}

\newcommand{\dog@rsemi}[2]{\dog@semi{#1}{#2}{-90,90}}
\newcommand{\dog@lsemi}[2]{\dog@semi{#1}{#2}{270,90}}
\newcommand{\dog@semi}[3]{%
  \begingroup
  \sbox\z@{$\m@th#1#2$}%
  \setlength{\unitlength}{\dimexpr\ht\z@+\dp\z@\relax}%
  \makebox[\wd\z@]{\raisebox{-\dp\z@}{%
    \begin{picture}(1,1)
    \linethickness{\variable@rule{#1}}
    \roundcap
    \put(0.5,0.5){\makebox(0,0){\raisebox{\dp\z@}{$\m@th#1#2$}}}
    \put(0.5,0.5){\arc[#3]{0.5}}
    \end{picture}%
  }}%
  \endgroup
}
\newcommand{\variable@rule}[1]{%
  \fontdimen8  
  \ifx#1\displaystyle\textfont3\else
    \ifx#1\textstyle\textfont3\else
      \ifx#1\scriptstyle\scriptfont3\else
        \scriptscriptfont3\relax
  \fi\fi\fi
}
\makeatother

\begin{document}

\pagestyle{plain}

\title{Linear embeddings of grassmannians and ind-grassmannians}
\author{Ivan Penkov and Valdemar Tsanov}

\maketitle


\begin{abstract}
By a grassmannian we understand a usual complex grassmannian or possibly an orthogonal or symplectic grassmannian. We classify, with few exceptions, linear embeddings of grassmannians into larger grassmannians, where the linearity requirement is the condition that the embedding induces an isomorphism on Picard groups. This classification implies that most linear embeddings of grassmannians are equivariant.

A linear ind-grassmannian is the direct limit of a chain of linear embeddings of grassmannians. We conclude the paper by classifying linear embeddings of linear ind-grassmannians.
\end{abstract}

{\small

}

\section*{Introduction}

Projective lines on grassmannians, and more general linearly embedded projective spaces in grassmannians, are classical objects in projective geometry. On the other hand, the Pl\"ucker embedding realizes a grassmannian as a projective variety. What these constructions have in common is that the sheaf $\mc O(1)$ on the ambient variety restricts to $\mc O(1)$ on the subvariety. It makes sense to consider embeddings of arbitrary grassmannians with this property, and following \cite{Penkov-Tikhomirov-Lin-ind-Grass} we call an embedding of complex grassmannians $X\stackrel{\varphi}{\hookrightarrow} Y$ {\it linear} if $\mc O_X(1) \cong \varphi^*\mc O_Y(1)$. Here we allow $X$ or $Y$, possibly both $X$ and $Y$, to be isotropic grassmannians. In \cite{Penkov-Tikhomirov-Lin-ind-Grass} the linear embeddings of grassmannians of the same type, i.e., when both $X$ and $Y$ are ordinary grassmannians, or when both $X$ and $Y$ are orthogonal or symplectic grassmannians, have been classified with some exceptions. 

A main application of this classification has been the classification of linear ind-grassmannians \cite{Penkov-Tikhomirov-Lin-ind-Grass}. The latter are defined as direct limits of usual finite-dimensional grassmannians of the same type under certain linear embeddings called standard extensions. Every linear ind-grassmannians is a homogeneous ind-space for one the ind-groups $SL(\infty)=\lim\limits_{\to} SL(n)$, $SO(\infty)=\lim\limits_{\to} SO(n)$, $Sp(\infty)=\lim\limits_{\to} Sp(n)$. A famous particular case of a linear ind-grassmannian is the Sato grassmannian.

In this paper we classify linear embeddings $X\stackrel{\varphi}{\hw}Y$ of grassmannians and ind-grassmannians, where $Y$ is not a spinor grassmannian or, respectively, a spinor ind-grassmannian. It is essential that $X$ and $Y$ may have different types.

Here is a brief description of the content. In section \ref{Sec Basic Notions} we present some main definitions concerning linear embeddings of grassmannians, and also recall the families of maximal linearly embedded projective spaces in grassmannians. Next, in section \ref{Sec Lin Emb Grass} we show that using the results of \cite{Penkov-Tikhomirov-Lin-ind-Grass} one can classify linear embeddings of grassmannians of different types, for instance embeddings of ordinary grassmannians into orthogonal grassmannians, embeddings of symplectic grassmannians into ordinary grassmannians, etc. There are three types of orthogonal grassmannians which behave in a special way and whose consideration is postponed until section \ref{Sec Special Emb}. There we construct some special embeddings which enable us to complete the classification except in the case where the target grassmannian is a spinor grassmannian. section \ref{Sec Special Emb} is concluded by a list of maximal embeddings of grassmannians.

Section \ref{Sec Equivar} is devoted to a study of equivariance properties of the linear embeddings. We show that most embeddings are actually equivariant. Non-equivariant embeddings occur only into isotropic grassmannians, and their images are contained in projective spaces, quadrics, or grassmannians of isotropic planes. Finally, in section \ref{Sec Ind Embed} we use the classification from sections \ref{Sec Lin Emb Grass} and \ref{Sec Special Emb} to classify the linear embeddings of linear ind-grassmannians, except when the target is a spinor ind-grassmannian.

Our arguments make essential use of properties of families of linearly embedded projective spaces in grassmannians. Such families have been studied in detail by Landsberg and Manivel in \cite{Lands-Mani-2003-ProjGeo} from a Lie-theoretic point of view, also in the context of general flag varieties. It is conceivable that a combination of the two approaches could have interesting applications in the study of embeddings of flag varieties.\\

We dedicate this paper to the memory of our late friend Joseph A. Wolf, with sadness that we can no longer share the joy of mathematics with him.\\

\noindent{\bf Acknowledgement:} Both authors have been supported in part by DFG grant PE 980/9-1. V.T. has also been supported by the Bulgarian Ministry of Education and Science, Scientific Programme ``Enhancing the Research Capacity in Mathematical Sciences (PIKOM)'', No. DO1-67/05.05.2022.

\section{Basic definitions and preliminaries}\label{Sec Basic Notions}

The base field is the field of complex numbers $\CC$. The notation ${}^*$ indicates dual vector space or pullback of vector bundles along morphisms of algebraic varieties, depending on the context. Symmetric and exterior powers of a vector space $V$ are denoted respectively by $S^nV$ and $\Lambda^nV$.

\subsection{Grassmannians}

We denote by $V_n$ an $n$-dimensional complex vector space, by $\PP^n$ a complex projective space of dimension $n$, and by $\PP(V_n)\cong\PP^{n-1}$ the projective space of $1$-dimensional subspaces of $V_n$. The grassmannian $G(m,V_n)$ of $m$-dimensional subspaces of $V_n$, for $1\leq m\leq n-1$, is an algebraic variety of dimension $m(n-m)$. The Pl\"ucker embedding
\begin{gather*}
\begin{array}{rcl}
G(m,V_{n})=\{ U\subset V_n : \dim U= m\} & \stackrel{\pi}{\hookrightarrow} & \PP(\Lambda^m V_n)\;, \\
U & \mapsto & \Lambda^m U
\end{array}
\end{gather*}
realizes $G(m,V_n)$ as a projective variety. It is well known that the Picard group of $G(m,V_n)$ is isomorphic to $\ZZ$ and is generated by the class of $\mc O_{G(m,V_{n})}(1)$, where $\mc O_{G(m,V_{n})}(1)$ is the line bundle with fibre $\Lambda^m U^*$ over the point $U\in G(m,V_n)$. We have $\mc O_{G(m,V_{n})}(1)=\pi^*\mc O_{\PP(\Lambda^m V_n)}(1)$.

For a fixed symmetric or skew-symmetric nondegenerate bilinear form on $V_n$, denoted respectively by $\kappa$ or $\omega$, the grassmannian of isotropic subspaces of dimension $m\leq n/2$ in $V_n$ is, respectively,
\begin{gather*}
\begin{array}{lcl}
GO(m,V_{n}):=\{ U\subset V_n : \dim U= m, \kappa_{\vert U\times U}=0\} & \stackrel{\tau}{\hookrightarrow} & G(m,V_n) \;,\\
GS(m,V_{n}):=\{ U\subset V_n : \dim U= m, \omega_{\vert U\times U}=0\} & \stackrel{\tau}{\hookrightarrow} & G(m,V_n) \;.
\end{array}
\end{gather*}
In either case we call $\tau$ the {\bf tautological embedding}. 

Note that all above grassmannians are connected, except for the variety of maximal isotropic subspaces of an even-dimensional orthogonal space. The latter variety has two connected components which are isomorphic, and in what follows $GO(m,V_{2m})$ stands for any one of them.

The Picard group of $GS(m,V_n)$ is freely generated by the class of $\mc O_{GS(m,V_{n})}(1):=\tau^*\mc O_{G(m,V_{n})}(1)$. The analogous statement holds for $GO(m,V_n)$ with three exceptions: $n=2m,2m+1$ or $2m+2$. In fact, the first two exceptions ``coincide'', as the corresponding varieties are isomorphic. More precisely, the {\bf spinor grassmannians} $GO(m,V_{2m})$ and $GO(m-1,V_{2m-1})$ are isomorphic. The Picard group here, say for $n=2m$, is freely generated by the class of a square root of $\tau^*\mc O_{G(m,V_{2m})}(1)$ and we denote such a square root by $\mc O_{GO(m,V_{2m})}(1)$. The third exception $GO(m,V_{2m+2})$ is notably different, as its Picard group is isomorphic to $\ZZ^2$; in what follows we exclude this variety from our discussion. Henceforth, by a {\bf grassmannian} we mean one of the varieties $G(m,V_n)$, $GO(m,V_n)$, $GS(m,V_n)$, except $GO(m,V_{2m+2})$.

There are a few isomorphisms between grassmannians, well known and classified by Onishchik \cite{Onishchik}. The infinite series of such isomorphisms are two: the odd-dimensional projective space $\PP^{2m-1}$ is isomorphic to both $G(1,V_{2m})$ and $GS(1,V_{2m})$; and the spinor grassmannians $GO(m,V_{2m})$ and $GO(m-1,V_{2m-1})$ are isomorphic. In low dimensions, there is a finite number of coincidences, i.e., isomorphisms: $\PP^1\cong G(1,V_2)\cong GS(1,V_2)\cong GO(2,V_3)\cong GO(2,V_4)$, $G(2,V_4)\cong GO(1,V_6)$, $GS(2,V_4)\cong GO(1,V_5)$, $\PP^3\cong G(1,V_4) \cong GS(1,V_4) \cong GO(2,V_5) \cong GO(3,V_6)$, $GO(1,V_8)\cong GO(4,V_8)\cong GO(3,V_7)$.

The grassmannian $GO(1,V_{n})$ is a quadric of dimension $n-2$ in $\PP(V_n)$, and we shall also use the notation $Q^{n-2}$ for this variety.

For a grassmannian $X$ we set
$$
V_X:=H^0(X,\mc O_X(1))^*
$$
and denote by $\pi_X$ the natural embedding $\pi_X:X\hookrightarrow\PP(V_X)$.

\subsection{Linear embeddings}

Let $X,Y$ be grassmannians. We start by recalling the definition of a linear embedding $X \stackrel{\varphi}{\hookrightarrow} Y$.

\begin{definition}
An embedding $X \stackrel{\varphi}{\hookrightarrow} Y$ is {\bf linear} if ${\mc O}_X(1)\cong\varphi^*\mc O_Y(1)$.
\end{definition}

Clearly, a composition of two embeddings of grassmannians is linear if and only if both embeddings are linear. A linear embedding $X\stackrel{\varphi}{\hookrightarrow} Y$ is said to {\bf factor through a projective space} if it can be written as a composition of embeddings
$$
\varphi: \; X \; \stackrel{\pi_X}{\hookrightarrow} \; \PP(V_X) \; \stackrel{\psi}{\hookrightarrow} \; Y \;.
$$
Here $\psi$ is necessarily linear. Also note that any linear embedding of grassmannians $X \stackrel{\varphi}{\hookrightarrow} Y$ induces a linear embedding of projective spaces
\begin{gather}\label{For hat phi}
\PP(V_X) \stackrel{\hat\varphi}{\hookrightarrow} \PP(V_Y)\;.
\end{gather}

\begin{example}$\;$
\begin{enumerate}
\item[$\bullet$] The Pl\"ucker embedding $G(m,V_n)\hookrightarrow \PP(\Lambda^m V_n)$ is linear. More generally, the embedding $\pi_X:X\hookrightarrow \PP(V_X)$ is linear.
\item[$\bullet$] The {\bf tautological embedding} $GO(m,V_n)\stackrel{\tau}{\hookrightarrow} G(m,V_n)$ is linear if and only if $n\geq 5$ and $m<\frac{n}{2}-1$.
\item[$\bullet$] The {\bf tautological embedding} $GS(m,V_n)\stackrel{\tau}{\hookrightarrow} G(m,V_n)$ is linear.
\item[$\bullet$] The Veronese embedding ${\rm Ver}_q:\PP(V_n)\hookrightarrow\PP(S^qV_n)$, $[v]\mapsto[v^q]$ is not linear for $q\geq 2$.
\end{enumerate}
\end{example}

A {\bf minimal projective embedding} of a grassmannian $X$ is a projective embedding $\pi:X\hookrightarrow \PP^k$, such that there does not exist an embedding of $X$ into $\PP^l$ for $l<k$. The embedding $\pi_X$ is a minimal projective embedding. Every minimal projective embedding of $X$ is linear and has the form $X\stackrel{\pi_X}{\hookrightarrow} \PP(V_X)\stackrel{\alpha}{\cong} \PP(V)$ for a suitable isomorphism $\alpha$.

A {\bf projective space on} $Y$ is a linearly embedded $\PP^k\subset Y$ for some $k\geq 1$. These are exactly the subvarieties of $Y$ sent by $\pi_Y$ to projective subspaces of $\PP(V_Y)$. An embedding of grassmannians $\varphi:X\to Y$ is linear if and only if it sends any projective space on $X$ to a projective space on $Y$.

A {\bf quadric on} $Y$ is a linearly embedded quadric $GO(1,V_k)\subset Y$ for some $k$. A {\bf standard quadric} on an orthogonal grassmannian $Y=GO(m,V_n)$ is a quadric $Q\subset Y$ such that $\tau_Y(Q)=\tau_Y(Y)\cap \PP^k$ for some projective space $\PP^k$ on $G(m,V_n)$, where $\tau_Y$ is the tautological embedding of $Y$.

Next we recall the definitions of various types of linear embeddings.

\begin{definition}\label{Def Standard Extension} {\rm (\cite{Penkov-Tikhomirov-Lin-ind-Grass})}
An embedding $G(m,V_n) \stackrel{\sigma}{\hookrightarrow} G(l,V_s)$ is said to be a {\bf strict standard extension} if $\sigma$ is given by
\begin{gather}\label{For sigma strict standard}
\sigma(U) = U\oplus V''
\end{gather}
for some isomorphism $V_s\cong V_{n}\oplus V'$ and a fixed subspace $V''\subset V'$ of dimension $l-m$.

An embedding $G(m,V) \stackrel{\varphi}{\hookrightarrow} G(l,W)$ is a {\bf standard extension} if it fits into a commutative diagram
\begin{gather*}
\begin{array}{ccc}
G(m,V) & \stackrel{\varphi}{\hookrightarrow} & G(l,W) \\
i_1 \downarrow \qquad & & \quad \downarrow i_2 \\
G(m',V_n) & \stackrel{\sigma}{\hookrightarrow} & G(l',V_s) \;,
\end{array}
\end{gather*}
where $i_1,i_2$ are isomorphisms and $\sigma$ is a strict standard extension.

An embedding $GO(m,V_n) \stackrel{\sigma}{\hookrightarrow} GO(l,V_s)$ with $n-2m\ne 2, s-2l\ne 2$ is a {\bf standard extension} if $\sigma$ is given by formula (\ref{For sigma strict standard}) for some orthogonal isomorphism $V_s\cong V_n\oplus V'$ and a fixed isotropic subspace $V''\subset V'$, where in addition we assume $l=\lfloor\frac{s}{2}\rfloor$ whenever $m=\lfloor\frac{n}{2}\rfloor$.

An embedding $GS(m,V_n) \stackrel{\sigma}{\hookrightarrow} GS(l,V_s)$ is a {\bf standard extension} if it is given by formula (\ref{For sigma strict standard}) for some symplectic isomorphism $V_s\cong V_n\oplus V'$ and for a fixed isotropic subspace $V''\subset V'$.
\end{definition}

It is easy to see that a standard quadric on $GO(l,V_s)$ is the image of a standard extension $\sigma:GO(1,V_n)\hookrightarrow GO(l,V_s)$. By analogy, we call the image of a standard extension $\sigma:GS(1,V_n)\hookrightarrow GS(l,V_s)$ a {\bf standard symplectic projective space} on $GS(l,V_s)$.

\begin{definition}\label{Def Isotropic Extension} {\rm (\cite{Penkov-Tikhomirov-Lin-ind-Grass})}
An embedding $G(l,V_n) \stackrel{\iota}{\hookrightarrow} GO(l,V_s)$ is an {\bf isotropic extension} if $l< n\leq \lfloor\frac{s}{2}\rfloor$, and there exists an isotropic subspace $W\subset V_s$ and an isomorphism $f:V_n\cong W$ such that
$$
\iota(U) = f(U) \subset W \;\; for\;\; U\in G(l,V_n)\;.
$$
An isotropic extension is {\bf minimal} if $W$ is a maximal isotropic subspace.

Isotropic extensions and minimal isotropic extensions $G(l,V_n) \stackrel{\iota}{\hookrightarrow} GS(l,V_s)$ are defined analogously.
\end{definition}

\begin{lemma}\label{Lemma Linearity of standard and isotropic ext}
Standard and isotropic extensions are linear embeddings.
\end{lemma}

\begin{proof}
The statement is straightforward, but note that the condition on dimensions in the case of orthogonal grassmannians is essential for linearity.
\end{proof}

\begin{definition}\label{Def Combi Standard and Isotrop} {\rm (\cite{Penkov-Tikhomirov-Lin-ind-Grass})}
A {\bf combination of standard and isotropic extensions} is a sequence of embeddings
$$
GO(m,V_n) \stackrel{\tau}{\hookrightarrow} G(m,V_n) \stackrel{\sigma}{\hookrightarrow} G(m',V_{n'}) \stackrel{\iota}{\hookrightarrow} GO(m',V_{n''})
$$
or
$$
GS(m,V_n) \stackrel{\tau}{\hookrightarrow} G(m,V_n) \stackrel{\sigma}{\hookrightarrow} G(m',V_{n'}) \stackrel{\iota}{\hookrightarrow} GS(m',V_{n''}) \;,
$$
where $\tau$ is a tautological embedding, $\sigma$ is a standard extension, and $\iota$ is an isotropic extension.

A {\bf mixed combination of standard and isotropic extensions} is a sequence of embeddings
$$
GO(m,V_n) \stackrel{\tau}{\hookrightarrow} G(m,V_n) \stackrel{\sigma}{\hookrightarrow} G(m',V_{n'}) \stackrel{\iota}{\hookrightarrow} GS(m',V_{n''})
$$
or
$$
GS(m,V_n) \stackrel{\tau}{\hookrightarrow} G(m,V_n) \stackrel{\sigma}{\hookrightarrow} G(m',V_{n'}) \stackrel{\iota}{\hookrightarrow} GO(m',V_{n''}) \;,
$$
where $\tau$ is a tautological embedding, $\sigma$ is a standard extension, and $\iota$ is an isotropic extension.
\end{definition}

\begin{lemma}\label{Lemma Connect by lines}
Let $X$ be a grassmannian and $x,y\in X$ be two points. Then there exist finitely many projective lines $L_1,...,L_k$ on $X$ such that $x,y\in L_1\cup...\cup L_k$.
\end{lemma}

\begin{proof} Let $X:=G(m,V)$ and let $U,U'$ be two points of $X$. Then $U$ and $U'$ can be connected by a sequence $U=U_1, U_2,...,U_s=U'$, such that $\dim(U_i\cap U_{i+1})=m-1$ for $i=1,...,s-1$. Since $U_i$ and $U_{i+1}$ lie on a projective line on $X$, the statement follows. The cases of orthogonal and symplectic grassmannians are similar.
\end{proof}

\subsection{Maximal projective spaces on grassmannians}\label{Sec Lin Spaces}

As a prerequisite, we need to recall the descriptions and some intersection properties of maximal projective spaces on grassmannians. An alternative description of the families of (maximal) projective spaces on grassmannians, and more general flag varieties, can be found in \cite{Lands-Mani-2003-ProjGeo}.

\subsubsection{Ordinary grassmannians} On a projective space there is a single maximal projective space - the space itself. In the grassmannian $G(m,V_n)$ with $1<m<n-1$ there are two connected families of maximal projective spaces. A space from the first family is determined by an $(m+1)$-dimensional subspace $U_{m+1}\subset V_n$, and has the form
$$
\PP^{m}_{U_{m+1}} :=\;\{ U\in G(m,V_n): U\subset U_{m+1} \} \cong \PP^{m}.
$$
A space from the second family is determined by an $(m-1)$-dimensional subspace $U_{m-1}\subset V_n$, and has the form
$$
\PP^{n-m}_{U_{m-1}} :=\;\{ U\in G(m,V_n): U_{m-1}\subset U \} \cong \PP^{n-m}.
$$
The intersection of any two spaces from the same family is empty or equals a point. The intersection of two spaces from different families is empty or is a projective line.

\subsubsection{Symplectic grassmannians}\label{Sec max proj GS} On the grassmannian $GS(m,V_{2n})$ with $1<m< n$ there are two connected families of maximal projective spaces. A space from the first family is determined by an $(m+1)$-dimensional isotropic subspace $U_{m+1}\subset V_{2n}$, and has the form
$$
\PP^{m}_{U_{m+1}} :=\;\{ U\in GS(m,V_{2n}): U\subset U_{m+1} \} \cong \PP^{m}.
$$
A space from the second family is determined by an $(m-1)$-dimensional isotropic subspace $U_{m-1}\subset V_{2n}$, and has the form
$$
\PP^{2(n-m)-1}_{U_{m-1}} :=\;\{ U\in GS(m,V_{2n}): U_{m-1}\subset U \} \cong \PP^{2(n-m)-1}.
$$
The spaces $\PP^{2(n-m)-1}_{U_{m-1}}$ are exactly the maximal standard symplectic projective spaces on $GS(m,V_{2n})$.

The intersection of two maximal projective spaces on $GS(m,V_{2n})$ belonging to the same family is empty or equals a point. The intersection of two space of different families is empty or is a projective line.

Every maximal projective space on the grassmannian $GS(m,V_{2m})$ is a standard symplectic projective line determined by an $(m-1)$-dimensional subspace $U_{m-1}\subset V_{2m}$, and has the form 
$$
\PP^{1}_{U_{m-1}} :=\;\{ U\in GS(m,V_{2n}): U_{m-1}\subset U \} \cong \PP^1.
$$

\subsubsection{Quadrics}\label{Sec max proj in quadrics} All maximal projective spaces on a quadric $Q^{n-2}=GO(1,V_n)\subset \PP(V_n)$ have dimension $\lfloor n/2\rfloor-1$ and are exactly the projectivizations of maximal isotropic subspaces of $V_n$.

For odd $n=2r+1$, the maximal projective spaces on $Q^{2r-1}$ are parametrized by the spinor grassmannian $GO(r,V_{2r+1})$. The intersection of two distinct maximal projective spaces on $Q^{2r-1}$ can be empty or be a projective space of any dimension between $0$ and $r-1$.

For even $n=2r$, there are two connected families of maximal projective spaces on $Q^{n-2}$, each parametrized by the spinor grassmannian $GO(r,V_{2r})$. Two different spaces from the same family intersect in a copy of $\PP^{r-1-2k}$ for some $k\geq 1$, and an empty intersection occurs if and only if $r$ is even. Two spaces from different families intersect in a copy of $\PP^{r-2k}$ for some $k\geq 1$, and an empty intersection occurs if and only if $r$ is odd.

\subsubsection{Generic orthogonal grassmannians} On the grassmannian $GO(m,V_{n})$ with $1<m< r-1$, $r:=\lfloor n/2\rfloor$, there are two types of maximal projective spaces. The first type constitutes a single connected family, each of whose members is determined by an $(m+1)$-dimensional isotropic subspace $U_{m+1}\subset V_{n}$ and has the form
$$
\PP^{m}_{U_{m+1}} :=\;\{ U\in GO(m,V_{n}): U\subset U_{m+1} \} \cong \PP^{m}.
$$
The intersection of two distinct projective spaces of this type is either empty or is a point.

Maximal projective spaces on $GO(m,V_n)$ of the second type are maximal projective spaces on maximal standard quadrics on $GO(m,V_n)$. These maximal projective spaces form one or two connected families depending on the parity of $n$. A maximal standard quadric on $GO(m,V_n)$ is determined by an $(m-1)$-dimensional isotropic subspace $U_{m-1}\subset V_{2n}$, and has the form
$$
Q^{n-2m}_{U_{m-1}} :=\;\{ U\in GO(m,V_n): U_{m-1}\subset U \} \cong Q^{n-2m}.
$$
The intersection of two distinct maximal standard quadrics on $GO(m,V_n)$ is either empty or is a point. The intersection of maximal projective spaces on $GO(m,V_n)$ contained in one maximal standard quadric is as described in subsection \ref{Sec max proj in quadrics}.

The intersection of a maximal projective space $\PP^{m}_{U_{m+1}}$ and a maximal standard quadric $Q^{n-2m-2}_{U_{m-1}}$ is empty or equals a point.

\begin{coro}
Let $GO(m,V_n)$ be identified with the image of its tautological embedding in $G(m,V_n)$, assuming $m<\frac{n}{2}-1$. If $\PP$ is a maximal projective space on $G(m,V_n)$ then there are four possibilities for the intersection $\PP\cap GO(m,V_n)$: it is empty; it is a point; it is a maximal projective space on $GO(m,V_n)$; it is a maximal standard quadric on $GO(m,V_n)$.
\end{coro}

\subsubsection{The grassmannians $GO(n-1,V_{2n+1})$}\label{Sec MaxProj in GOcodim1odd} The maximal projective spaces on $X$ form a single connected family parametrized by the spinor grassmannian of maximal isotropic subspaces of $V_{2n+1}$:
$$
\PP^{n-1}_{U_n}:=\{U\in X: U\subset U_n\} \;,\quad U_n\in GO(n,V_{2n+1})\;.
$$
The intersection of two distinct maximal projective spaces on $X$ is a point or is empty, i.e., for $U_n\ne U'_n$ we have
$$
\PP^{n-1}_{X,U_n}\cap \PP^{n-1}_{X,U'_n} = \begin{cases} \{ U_n\cap U'_n\} & \;{\rm if}\; \dim U_n\cap U'_n=n-1\;,\\ \emptyset & \;{\rm if}\; \dim U_n\cap U'_n<n-1\;. \end{cases}
$$
The maximal standard quadrics on $X$ are parametrized by $GO(n-2,V_{2n+1})$, and have the form
$$
Q^3_{U_{n-2}}:=\{U\in X: U_{n-2}\subset U\} \;.
$$
The intersection of two distinct maximal standard quadrics on $X$ is empty or is a point. The intersection of $Q^3_{U_{n-2}}$ and $\PP^{n-1}_{U_n}$ is empty if $U_{n-2}\not\subset U_n$, and equals the projective line 
$$
\PP^1_{U_{n-2}\subset U_{n}}:=\{U\in X: U_{n-2}\subset U\subset U_n\}
$$
if $U_{n-2}\subset U_n$. Every projective line on $X$ is the intersection of a unique maximal standard quadric and a unique maximal projective space.

\subsubsection{Spinor grassmannians}\label{Sec MaxLin on Spinor}

Let $X:=GO(m,V_{2m})$ and $\bar X:=\ol{GO}(m,V_{2m})$ denote the two connected components of the variety of maximal isotropic subspaces of a fixed vector space $V_{2m}$ endowed with a non-degenerate symmetric bilinear form. Thus $\bar X\cong X$. We assume $m\geq 5$.

We note that two distinct points $V,W\in X$ lie on a projective line on $X$ if and only if $\dim V\cap W=m-2$. Thus any projective line on $X$ is determined uniquely by an $(m-2)$-dimensional isotropic subspace, and has the form
\begin{gather}\label{For Lines on Spinor}
\PP^1_{U_{m-2}}:=\{W\in GO(m,V_{2m}): U_{m-2}\subset W\}\quad{\rm for}\quad U_{m-2}\in GO(m-2,V_{2m})\;.
\end{gather}
There are two families of maximal projective spaces on $GO(m,V_{2m})$: one family of $\PP^{m-1}$-s and one family of $\PP^3$-s, parametrized respectively by $\bar X$ and $GO(m-3,V_{2m})$. A space from the first family has the form
$$
\PP^{m-1}_{\bar U_{m}}:= \{W\in GO(m,V_{2m}) : \dim W\cap \bar U_{m}= m-1 \}  \cong \PP^{m-1} \;,\quad \bar U_{m}\in \ol{GO}(m,V_{2m})\;.
$$
A space from the second family has the form
$$
\PP^3_{U_{m-3}} := \{W\in GO(m,V_{2m}) : U_{m-3} \subset W \} \cong \PP^3\;,\quad U_{m-3}\in GO(m-3,V_{2m})\;.
$$

The intersection of two distinct spaces from the family of maximal $\PP^{m-1}$-s is empty or is a projective line, i.e., for a pair of distinct subspaces $\bar U_{m},\bar U'_{m}\in \bar X$ we have
$$
\PP^{m-1}_{\bar U_{m}} \cap \PP^{m-1}_{\bar U'_{m}}  = \begin{cases} \emptyset & \;{\rm if}\; \dim \bar U_{m}\cap \bar U'_{m} <m-2\;, \\ \PP^1_{\bar U_{m}\cap \bar U'_{m}} & \;{\rm if}\; \dim \bar U_{m}\cap \bar U'_{m} = m-2 \;. \end{cases}
$$
There are three possibilities for the intersection of two distinct spaces from the family of maximal $\PP^3$-s: it is empty, it is a point, or it is a projective line. Indeed, $\PP^{3}_{U_{m-3}}$ and $\PP^{3}_{U'_{m-3}}$ intersect if and only if the sum $U:=U_{m-3}+U'_{m-3}$ is contained in $U_m$ for some $U_m\in X$. If $U_{m-3}\subset U_m$ for some $U_m\in X$ and $U_{m-3}\ne U'_{m-3}$, then the intersection $U\cap U_m$ has codimension $0,1$ or $2$ in $U_m$. Hence
$$
\PP^{3}_{U_{m-3}} \cap \PP^{3}_{U'_{m-3}}  =\;\{ W\in GO(m,V_{2m}): W\supset U \} =\;\begin{cases} \emptyset & \;{\rm if}\; U\not\subset W\; \forall\; W\in X\;,\\
\{U_m\} & \;{\rm if}\; \dim U\geq m-1\;,\; U\subset U_m\in X\;,\\  
\PP^1_{U} & \;{\rm if}\; \dim U = m-2 \;.
\end{cases}
$$

There are three possibilities for the intersection of $\PP^{m-1}_{\bar U_{m}}$ and $\PP^{3}_{U_{m-3}}$: it is empty, a point, or a projective plane. More precisely,
$$
\PP^{m-1}_{\bar U_{m}}\cap\PP^{3}_{U_{m-3}} = \begin{cases} \emptyset & \;{\rm if}\; \dim (U_{m-3}+  \bar U_{m})>m+1 \;, \\ \{V_m\} & \;{\rm if}\;\; \begin{array}{|l} \dim (U_{m-3}+ \bar U_{m})= m+1\;,\\ U_{m-3}+ \bar U_{m}=V_m+\bar U_{m}\;\textrm{for some}\; V_m\in X\;,\end{array} \\ \PP^2_{U_{m-3}\subset \bar U_{m}}& \;{\rm if}\; U_{m-3}\subset \bar U_{m}\;,\end{cases}
$$
where
$$
\PP^2_{U_{m-3}\subset \bar U_{m}} := \{W\in GO(m,V_{2m}) : U_{m-3}\subsetneq  W\cap \bar U_{m} \} \cong \PP^2 \;.
$$
For every $\PP^2$ on $GO(m,V_{2m})$ containing a point $V_m$, there exists a unique flag $U_{m-3}\subset \bar U_{m}$ such that $\dim V_m\cap \bar U_m=m-1$ and $\PP^2=\PP^2_{U_{m-3}\subset \bar U_{m}}$.

\section{Classifications of linear embeddings of grassmannians}\label{Sec Lin Emb Grass}

\subsection{Linear embeddings into quadrics}

Let $X\subset \PP(V_X)$ be a grassmannian identified with its image under the minimal projective embedding $\pi_X$. It is well known that the ideal $I(X)$ in the homogeneous coordinate ring $\CC[V_X]$ is generated by its degree-two component $I_2(X)$, which is in turn spanned by (in most cases degenerate) quadratic polynomials, see e.g. \cite[Theorem 16.2.2.6]{Landsberg-2012-book}. For $X=G(m,V_n)$ these quadratic polynomials are the original Pl\"ucker relations. The ideal of $X$ is of course trivial if and only if $X=\PP(V_X)$.

\begin{lemma}\label{Lemma embed in quadric}
Let $X\subset\PP(V_X)$ be a projectively embedded grassmannian. For every nonzero $p\in I_2(X)$ there exists a linear embedding $X\stackrel{\kappa_p}{\hookrightarrow} Q=GO(1,V_n)$ with $n=2\dim V_X-{\rm rank}\,p$, which does not factor through a projective space or through a quadric of smaller dimension. Furthermore, if $Q\subset\PP(V_n)$ is tautologically embedded and $q\in I_2(Q)$ is a generator of the ideal of $Q$, then $\kappa_p$ extends to a linear embedding $\tilde\kappa_p:\PP(V_X)\hookrightarrow \PP(V_n)$ and $\tilde\kappa_p^*q=ap$ for some nonzero scalar $a$.

Conversely, every linear embedding $X\stackrel{\varphi}{\hookrightarrow} Q^{s-2}=GO(1,V_s)$, which does not factor through a projective space, factors as a composition $\varphi: X\stackrel{\kappa_p}{\hookrightarrow} Q \stackrel{\sigma}{\hookrightarrow} Q^{s-2}$ for some nonzero $p\in I_2(X)$ and some standard extension of quadrics $\sigma$.
\end{lemma}

\begin{proof}
Let $p\in I_2(X)$ be a nonzero element and $r$ be its rank. Then $r\geq2$ since $X$ is not contained in a hyperplane on $\PP(V_X)$. The space $V_X$ admits a decomposition $V_X=V_r\oplus U$ such that the restriction of $p$ to $V_r$ is non-degenerate and $U$ is an isotropic space for $p$. Set $V_n:=V_X\oplus U^*=V_r\oplus U\oplus U^*$. Pick a non-degenerate quadratic polynomial $q\in \CC[V_n]_2$ whose restriction to $V_X$ equals $p$ and which vanishes on $U^*$. The quadric $Q\subset\PP(V_n)$ defined by the vanishing of $q$ contains $X$. Furthermore, $X$ is not contained in a projective space on $Q$ because $\PP(V_X)$ is the minimal projective subspace of $\PP(V)$ containing $X$ and $q$ does not vanish on $\PP(V_X)$. To show that $X$ is not contained in a smaller quadric $\tilde Q\subset Q$, note that $\PP(V_X)\subset \PP(V_{\tilde Q})\subset\PP(V_n)$ whenever $X\subset \tilde Q\subset Q$. On the other hand, $V_X$ contains the orthogonal space $W^\perp$ to a maximal $q$-isotropic subspace $W\subset V_n$, and $W^\perp$ is not contained in a proper subspace of $V_n$ on which the restriction of $q$ is non-degenerate. Hence, $X\subset \tilde Q\subset Q$ implies $\tilde Q=Q$. This proves the first part of the lemma.

For the converse statement, consider a linear embedding $X\stackrel{\varphi}{\hookrightarrow} Q^{s-2}$ which does not factor through a projective space. Let $Q^{n-2}\subset Q^{s-2}$ be a minimal quadric containing $\varphi(X)$ and embedded in $Q^{s-2}$ by a standard extension. Let $\PP(V_X)\stackrel{\hat\varphi}{\hookrightarrow}\PP(V_n)\subset \PP(V_s)$ be the corresponding linear embeddings of projective spaces. Then $V_n$ is a minimal non-degenerate subspace of $V_s$ containing $V_X$. Denote by $q\in I_2(Q^{n-2})$ a generator of the ideal of $Q^{n-2}$ in $\CC[V_n]$ and let $p:=\hat\varphi^*q$. The element $p\in I_2(X)$ is nonzero because $\varphi$ does not factor through a projective space, the equality $n=2\dim V_X-{\rm rank}\,p$ holds, and the embedding of $X$ into $Q^{n-2}$ is the embedding $\kappa_p$. This completes the proof.
\end{proof}

The above lemma and its proof have the following two immediate corollaries.

\begin{coro}
Let $X\stackrel{\varphi}{\hookrightarrow} Y$ be a linear embedding of grassmannians and $\PP(V_X)\stackrel{\hat\varphi}{\hookrightarrow} \PP(V_Y)$ be the induced linear embedding of projective spaces. Let $p_1\in I_2(X)$, $p_2\in I_2(Y)$ be nonzero elements and $\kappa^X_{p_1}$, $\kappa^Y_{p_2}$ be the respective embeddings of $X$ and $Y$ into quadrics. Then the following statements are equivalent:
\begin{enumerate}
\item[{\rm (i)}] The embedding $\kappa^X_{p_1}$ factors as $\kappa^X_{p_1}=\kappa^Y_{p_2}\circ\varphi$.
\item[{\rm (ii)}] The equalities $\hat\varphi^*p_2=ap_1$ (for some nonzero scalar $a$) and $\dim V_X-{\rm rank}\,p_1 = 2\dim V_Y - {\rm rank}\,p_2$ hold.
\end{enumerate}
\end{coro}

\begin{coro}
For any grassmannian $X$ let $r_X:=\max\{{\rm rank}\,p:p\in I_2(X)\}$. Then the minimal number $n$ for which a linear embedding $X\hookrightarrow Q^{n-2}$ exists is equal to $2\dim V_X- r_X$. The minimal projective embedding $X\hookrightarrow \PP(V_X)$ factors through a quadric if and only if $r_X=\dim V_X$.
\end{coro}

\subsection{Linear embeddings of grassmannians: generic case}

We now proceed with the main steps of our classification of linear embeddings between grassmannians. First we recall the following theorem.

\begin{theorem}\label{Theo Embed Same Type} {\rm (\cite[Theorem 3.1]{Penkov-Tikhomirov-Lin-ind-Grass}, Linear embeddings between grassmannians of the same type)}

\begin{enumerate}
\item[\rm (i)] Every linear embedding $G(m,V_n) \hookrightarrow G(l,V_s)$ is either a standard extension or factors through a projective space.
\item[\rm (ii.1)] For a linear embedding $\varphi:GO(m,V_n) \hookrightarrow GO(l,V_s)$, where either $l\leq\frac{s}{2}-2$ and $m<\frac{n}{2}-2$, or both $n,s$ are odd and $0<\lfloor\frac{s}{2}\rfloor-l\leq \lfloor\frac{n}{2}\rfloor-m\leq2$, there are four options: $\varphi$ is a standard extension; $\varphi$ is a combination of standard and isotropic extensions; $\varphi$ factors through a projective space; $\varphi$ factors through a standard quadric but not through a projective space.
\item[\rm (ii.2)] For a linear embedding $\varphi:GO(\lfloor\frac{n}{2}\rfloor,V_{n}) \hookrightarrow GO(\lfloor\frac{s}{2}\rfloor,V_{s})$ there are two options: $\varphi$ is a standard extension; $\varphi$ factors through a projective space.
\item[\rm (iii)] For a linear embedding $\varphi:GS(m,V_n) \hookrightarrow GS(l,V_s)$ there are three options: $\varphi$ is a standard extension; $\varphi$ factors through a projective space; $\varphi$ is a combination of standard and isotropic extensions.
\end{enumerate}
\end{theorem}

\begin{proof}
The statement given here covers a few cases which are omitted in Theorem 3.1 in \cite{Penkov-Tikhomirov-Lin-ind-Grass}, namely in part (ii) of this theorem the hypothesis is $l<\lfloor\frac{s}{2}\rfloor-2$ and $m<\lfloor\frac{n}{2}\rfloor-2$ instead of our hypothesis $l\leq\frac{s}{2}-2$ and $m<\frac{n}{2}-2$. The proof of Theorem 3.1 in \cite{Penkov-Tikhomirov-Lin-ind-Grass} is however valid in the additional cases as well.
\end{proof}

Our main result in this subsection is the following addition to Theorem \ref{Theo Embed Same Type}.

\begin{theorem}\label{Theo Embed Mix Type} {\rm (Linear embeddings between grassmannians of different types)}

\begin{enumerate}
\item[\rm (i)] A linear embedding $G(m,V_n)\hookrightarrow GO(l,V_s)$ with $s\ne2l,2l+1,2l+2$, which does not factor through a projective space or a standard quadric, is a composition
$$
G(m,V_n) \stackrel{\sigma}{\hookrightarrow} G(l,V_{n'}) \stackrel{\iota}{\hookrightarrow} GO(l,V_s) \;,
$$
where $\sigma$ is a standard extension and $\iota$ is an isotropic extension.
\item[\rm (ii)] A linear embedding $G(m,V_n)\hookrightarrow GS(l,V_{s})$ is a composition
$$
G(m,V_n) \stackrel{\sigma}{\hookrightarrow} G(l,V_{n'}) \stackrel{\iota}{\hookrightarrow} GS(l,V_{s}) \;,
$$
where $\sigma$ is a standard extension and $\iota$ is an isotropic extension, or factors through a projective space.
\item[\rm (iii)] A linear embedding $GO(m,V_n)\hookrightarrow G(l,V_s)$ with $n\geq 7$ and $m<\lfloor\frac{n}{2}\rfloor-2$, which does not factor through a projective space, is a composition
$$
GO(m,V_n) \stackrel{\tau}{\hookrightarrow} G(m,V_n) \stackrel{\sigma}{\hookrightarrow} G(l,V_s) \;,
$$
where $\tau$ is a tautological embedding and $\sigma$ is a standard extension.
\item[\rm (iv)] A linear embedding $GS(m,V_{n})\hookrightarrow G(l,V_s)$, which does not factor through a projective space, is a composition
$$
GS(m,V_{n}) \stackrel{\tau}{\hookrightarrow} G(m,V_{n}) \stackrel{\sigma}{\hookrightarrow} G(l,V_s) \;,
$$
where $\tau$ is a tautological embedding and $\sigma$ is a standard extension.

\item[\rm (v)] A linear embedding $GO(m,V_n)\hookrightarrow GS(k,V_{q})$, where $n\geq 7$ and $m<\lfloor\frac{n}{2}\rfloor-2$, is a mixed combination of standard and isotropic extensions, or factors through a projective space. The same holds for a linear embedding $GS(k,V_{q})\hookrightarrow GO(m,V_n)$, where $7\leq n\ne 2m,2m+1,2m+2$, with the additional possibility for it to factor through a standard quadric.
\end{enumerate}
\end{theorem}

\begin{proof} For part (i) we may assume that $m\ne 1,n-1$ since the statement is already known for embeddings of projective spaces. We consider the composition
$$
G(m,V_n) \stackrel{\varphi}{\hookrightarrow} GO(l,V_s) \stackrel{\tau}{\hookrightarrow} G(l,V_s) \;,
$$
where $\varphi$ is a given linear embedding and $\tau$ is the tautological embedding. By Theorem \ref{Theo Embed Same Type},(i) the embedding $\tau\circ\varphi$ is either a standard extension or factors through a projective space on $G(l,V_s)$. Let us first assume that $\tau\circ\varphi$ is a strict standard extension. Then the image of $\tau\circ\varphi$ consists of all $l$-dimensional subspaces of $V_s=V_n\oplus V$ of the form $S_m\oplus W$ for some fixed $(l-m)$-dimensional subspace $W\subset V$ and $S_m\in G(m,V_n)$. By hypothesis the image of $\tau\circ\varphi$ lies in $GO(l,V_s)$, so each subspace $S_m\oplus W$ is isotropic in $V_s$. Therefore the entire space $V_n\oplus W$ is isotropic in $V_s$, and hence $\varphi$ is the composition
$$
G(m,V_n) \stackrel{\varphi'}{\hookrightarrow}  G(l,V_n\oplus W)\stackrel{\iota}{\hookrightarrow} GO(l,V_s)
$$
where $\varphi'(S_m)=S_m\oplus W$ and $\iota(S_m\oplus W)=S_m\oplus W \in GO(l,V_s)$. The claim follows in this case.

We now suppose that $\tau\circ\varphi$ is a non-strict standard extension. Then the composition
$$
G(n-m,V_n^*)\stackrel{\delta}{\to} G(m,V_n) \stackrel{\varphi}{\hookrightarrow} GO(l,V_s) \stackrel{\tau}{\hookrightarrow} G(l,V_s)
$$
of $\tau\circ\varphi$ with the duality isomorphism $\delta$ is a strict standard extension. Our claim is already proven for strict standard extensions. Thus $\varphi\circ\delta=\iota\circ\sigma'$ where $\sigma':G(n-m,V_n^*)\to G(m',V_{n'})$ is a standard extension and $\iota:G(m',V_{n'})\to GO(l,V_s)$ is an isotropic extension. Then $\sigma:=\sigma'\circ\delta^{-1}:G(m,V_n)\to G(m',V_{n'})$ is a standard extension and $\varphi=\iota\circ\sigma$ as required.

It remains to consider the case where $\tau\circ\varphi$ factors through a projective space, i.e., $\tau\circ\varphi=\lambda\circ\psi$ where $G(m,V_n)\stackrel{\psi}{\hookrightarrow} \PP^{k}\stackrel{\lambda}{\hookrightarrow} G(l,V_s)$. Then we have $\varphi(G(m,V_n))\subset \lambda(\PP^k)\cap \tau(GO(l,V_s))$. We can assume that $\lambda(\PP^k)$ is a maximal projective space on $G(l,V_s)$, in which case there are two possibilities: either $k=l$ and $\lambda(\PP^k)=\{U_l\in G(l,V_s):U_l\subset W\}=:\PP_W^{l}$ for a fixed subspace $W\subset V_s$ of dimension $l+1$, or $k=s-l$ and $\lambda(\PP^k)=\{U_l\in G(l,V_s):U\subset U_l\}=:\PP_U^{s-l}$ for a fixed subspace $U\subset V_s$ of dimension $l-1$. 

We claim that in the former case $W$ is necessarily isotropic. This follows from the observation that $W$ contains all $l$-dimensional spaces $U_l$ in the image of the embedding $\tau\circ\varphi$, hence at least two different isotropic $l$-dimensional subspaces. Hence $\lambda(\PP^k)\subset \tau(GO(l,V_s))$ and $\varphi$ factors through a projective space, which contradicts our assumption.

In the case where $\lambda(\PP^k)=\PP_U^{s-l}$ with $\dim U=l-1$, we observe that $\lambda(\PP^k)\cap\tau(GO(l,V_s))=:Q_U^{s-2l-2}$ is a maximal standard quadric on $GO(l,V_s)$, hence $\varphi$ factors through this quadric. Since the dimension of $G(m,V_n)$ is at least $4$ we have $s-2l\geq 6$, and hence the quadric $Q_U^{s-2l-2}$ is a grassmannian with Picard group isomorphic to $\ZZ$ and generated by the restriction of $\mc O_{\PP_U^{s-l}}(1)$. By assumption $\varphi$ does not factor through a projective space on $GO(l,V_s)$, so the embedding of $X$ in $Q_U^{s-2l-2}$ is one of the embeddings described in Lemma \ref{Lemma embed in quadric}.

Part (ii) is analogous to part (i). Here $\lambda(\PP^k)\cap \tau(GS(l,V_s))$ is a projective space and the possibility of factoring through a quadric does not occur.

Part (iii). Let $GO(m,V_n) \stackrel{\varphi}{\hookrightarrow} G(l,V_s)$ be a linear embedding. We assume that $l\leq s/2$, otherwise the argument below can be applied to the composition of $\varphi$ with the duality isomorphism $G(l,V_s)\cong G(s-l,V_s^*)$. Now the inequality $n\geq 7$ implies $l< s-2$. By hypothesis we also have $m<\lfloor\frac{n}{2}\rfloor-2$, and hence Theorem \ref{Theo Embed Same Type} can be applied to the composition
$$
GO(m,V_n) \stackrel{\varphi}{\hookrightarrow} G(l,V_s) \stackrel{\iota_0}{\hookrightarrow} GO(l,V_{2s}) \;,
$$
where $\iota_0$ is a minimal isotropic extension where $V_s$ is identified with a maximal isotropic subspace of $V_{2s}$. Since the embedding $\iota_0\circ\varphi$ factors through the grassmannian $G(l,V_s)$, $\iota_0\circ\varphi$ cannot be a standard extension, and by Theorem \ref{Theo Embed Same Type} there are three options: $\iota_0\circ\varphi$ factors through a projective space, or factors through a quadric, or is a combination of standard and isotropic extensions of the form
$$
GO(m,V_n) \stackrel{\tau}{\hookrightarrow} G(m,V_n) \stackrel{\sigma}{\hookrightarrow} G(l,V_r) \stackrel{\iota}{\hookrightarrow} GO(l,V_{2s}) \;.
$$
If $\iota_0\circ\varphi$ factors through a projective space or a standard quadric $Z\subset GO(l,V_{2s})$, then the intersection $Z\cap \iota_0(G(l,V_s))$ is a projective space through which $\varphi$ factors.

If the third option holds and $\iota_0\circ \varphi=\iota\circ\sigma\circ \tau$ then, from the definitions of standard and isotropic extensions, we obtain injective linear maps $V_n\hookrightarrow V_r \hookrightarrow V_{2s}$ with $V_r \hookrightarrow V_{2s}$ isotropic. By altering the map $V_r \hookrightarrow V_{2s}$ without changing the image of $V_n$ in $V_{2s}$, if necessary, we can assume that the image of $V_r$ is contained in the maximal isotropic subspace $V_s$ of $V_{2s}$. This yields an embedding $\iota(G(l,V_r))\stackrel{\sigma_1}{\hookrightarrow} \iota_0(G(l,V_s))$ which must be a standard extension due to Theorem \ref{Theo Embed Same Type},(i), as under our current hypothesis $\varphi$ does not factor through a projective space. Then the composition $\sigma_2:=\sigma_1\circ\sigma:G(m,V_n)\hookrightarrow G(l,V_s)$ is a standard extension satisfying $\varphi=\sigma_2\circ\tau$ as required.

Part (iv) is analogous to part (iii).

Part (v). First we consider a linear embedding $GO(m,V_n) \stackrel{\varphi}{\hookrightarrow} GS(k,V_s)$. Let us form the composition
$$
GO(m,V_n) \stackrel{\varphi}{\hookrightarrow} GS(k,V_s) \stackrel{\tau}{\hookrightarrow} G(k,V_s) \;,
$$
where $\tau$ is the tautological embedding. By (iii), $\tau\circ\varphi$ either factors through a projective space or can be written as $\sigma\circ\tau'$, where
$$
GO(m,V_n) \stackrel{\tau'}{\hookrightarrow} G(m,V_n) \stackrel{\sigma}{\hookrightarrow} G(k,V_s)
$$
with $\tau'$ tautological and $\sigma$ a standard extension. In the latter case, the image of $\sigma\circ\tau'$ lies in $GS(k,V_s)$ if $V_s\oplus W$ is isotropic for the symplectic form chosen on $V_s$, and we can apply the same argument as in (i) to prove the claim. If $\tau\circ\varphi$ factors through a projective space $\lambda(\PP^k)\subset G(k,V_s)$, the image of $\tau\circ\varphi$ is contained in the intersection $\lambda(\PP^k)\cap \tau(GS(k,V_s))$, which is in turn a projective space on $GS(k,V_s)$. Thus $\varphi$ factors through a projective space.

The case of a linear embedding $GS(k,V_{q})\stackrel{\varphi}{\hookrightarrow} GO(m,V_n)$ is analogous, except for the situation where $\tau\circ\varphi:GS(k,V_{q})\to G(m,V_n)$ factors through a maximal projective space $\lambda(\PP^k)\subset G(m,V_n)$. Here, as in part (i), the intersection $\lambda(\PP^k)\cap \tau(GO(m,V_n))$ is either a projective space or a standard quadric on $GO(m,V_n)$. In the former case $\varphi$ factors through a projective space. In the latter case, if $Q\subset GO(m,V_n)$ is a standard quadric containing the image $\varphi(X)$, then the resulting embedding $\varphi_1:X\hookrightarrow Q$ either factors through a projective space (on $Q$ and hence on $GO(m,V_n)$), or is one of the embeddings described in Lemma \ref{Lemma embed in quadric}.
\end{proof}

\begin{coro}\label{Coro in Lagrange only Lagrange}
If $X$ is a grassmannian and $X\stackrel{\varphi}{\hookrightarrow}GS(m,V_{2m})$ is a linear embedding then $X$ is isomorphic to $GS(l,V_{2l})$ with $1\leq l\leq m$, and the embedding $\varphi$ is a standard extension.
\end{coro}

\begin{proof}
The statement follows from the observation that the grassmannians $GS(m,V_{2m})$ with $m\geq 1$ are characterized among all grassmannnians by the property that the maximal projective spaces on them are projective lines, and from Theorem \ref{Theo Embed Same Type}.
\end{proof}

\begin{coro}\label{Coro Embed in GOcodim1odd}
Let $X$ be a grassmannian. If $X\stackrel{\varphi}{\hookrightarrow} GO(n-1,V_{2n+1})$ is a linear embedding which does not factor through a projective space, then $\varphi$ is a standard extension $X\cong GO(k-1,V_{2k+1})\hookrightarrow GO(n-1,V_{2n+1})$ for some $k\leq n$.
\end{coro}

\begin{proof}
Set $Y=GO(n-1,V_{2n+1})$. As recalled in subsection \ref{Sec MaxProj in GOcodim1odd}, the maximal projective spaces on $Y$ form a single family parametrized by the spinor grassmannian $\PP^{n-1}_{X,U_n}$, $U_n\in GO(n,V_{2n+1})$. The intersection of two distinct maximal projective spaces on $Y$ is either empty or a point. The only grassmannians with the latter property are $GO(k-1,V_{2k+1})$ and $GS(k,V_{2k})$ for $k\geq 1$. This implies that if $X$ is a grassmannian not isomorphic to $GO(k-1,V_{2k+1})$ or $GS(k,V_{2k})$, then any linear embedding $\varphi:X\hw Y$ factors through a projective space. Note that the maximal projective spaces on $Y$ are exactly the images of isotropic extensions into $Y$, since any isotropic extension into $Y$ has the form $\iota:G(n-1,U_n)\hw Y$ for some maximal isotropic subspace $U_n\subset V_{2n+1}$, and the grassmannian $G(n-1,U_n)$ is a projective space. Therefore every embedding into $Y$ which factors through an isotropic extension into $Y$ factors through a projective space. Now Theorem \ref{Theo Embed Mix Type},(v) implies that any linear embedding $GS(k,V_{2k})\hw Y$ factors through a projective space. The case where $X\cong GO(k-1,V_{2k+1})$ is handled in Theorem \ref{Theo Embed Same Type}, (ii), and we observe that, unless the embedding $\varphi$ is a standard extension, it factors through an isotropic extension into $Y$, and hence through a projective space.
\end{proof}

\begin{coro}\label{Coro Embed in GOcodim2even}
Let $X$ be a non-spinor grassmannian and $X\stackrel{\varphi}{\hw} GO(n-2,V_{2n})$ be a linear embedding which does not factor through an isotropic extension $G(n-2,U_n)\hw GO(n-2,V_{2n})$. Then $X$ is isomorphic to $GO(k-2,V_{2k})$ or $GO(k-2,V_{2k-1})$ for some $k\leq n$, and $\varphi$ is a standard extension.
\end{coro}

\begin{proof}
Let $Y:=GO(n-2,V_{2n})$. We observe that any projective space on $Y$ is contained in the image of some isotropic extension $G(n-2,U_n)\hw Y$. Thus the hypothesis implies that $\varphi$ does not factor through a projective space. A maximal standard quadric on $Y$ is 4-dimensional, hence if $\varphi$ factors through a standard quadric then $X$ is isomorphic to $\PP^1$, $\PP^2$, $Q^3$, or $Q^4$. The first two of these cases are excluded, while in the latter two $\varphi$ is a standard extension. Further, $X$ is not an ordinary or symplectic grassmannian, because by Theorem \ref{Theo Embed Mix Type} every linear embedding of such grassmannians into $Y$ factors through an isotropic extension or through a standard quadric. Hence $X$ is an orthogonal grassmannian.

The case $X=GO(m,V_r)$ with $m< \frac{r}{2}-2$ is impossible. This follows from Theorem \ref{Theo Embed Same Type}. Indeed, since there are no standard extensions $GO(m,V_r)\hw GO(n-2,V_{2n})$ with $m< \frac{r}{2}-2$, the only remaining option for $\varphi$ is to be a combination of standard and isotropic extensions. A contradiction with our assumption.

The case $X=GO(k-2,V_{2k})$ is considered in \cite[Proposition 3.15]{Penkov-Tikhomirov-Lin-ind-Grass}. The options for $\varphi$ can be reduced to the following two: $\varphi$ is a standard extension, or $\varphi$ factors through an isotropic extension. The proof of \cite[Proposition 3.15]{Penkov-Tikhomirov-Lin-ind-Grass} extends without substantial alterations to the case of $X=GO(k-2,V_{2k-1})$. Indeed, the key step in that proof is to observe that the restriction of $\varphi$ to a maximal standard quadric $Q^3\subset X$ is a standard extension or factors through a projective space. In the latter case $\varphi$ factors through an isotropic extension, and in the former case $\varphi$ is a standard extension.
\end{proof}

\begin{coro}\label{Coro Embed Quadrics}
Let $n\geq 7$ and $Y$ be a non-spinor grassmannian. Every linear embedding $GO(1,V_n)\stackrel{\varphi}{\hookrightarrow} Y$ is a standard extension
$$
GO(1,V_n) \hookrightarrow GO(l,V_s)
$$
for $s-2l\geq n$, or factors through a projective space.
\end{coro}

\begin{proof}
The statement follows by examination of the cases occurring in Theorems \ref{Theo Embed Same Type}, \ref{Theo Embed Mix Type}, and Corollaries \ref{Coro Embed in GOcodim1odd}, \ref{Coro Embed in GOcodim2even}.
\end{proof}

\begin{rem}
Due to the isomorphism $GO(1,V_5)\cong GS(2,V_4)$, the three-dimensional quadric admits standard extensions to both orthogonal and symplectic grassmannians. These exhaust its linear embeddings which do not factor through a projective space. The four-dimensional quadric $GO(1,V_6)$ is isomorphic to $G(2,V_4)$, and hence, besides the standard extensions to orthogonal grassmannians, it admits standard extensions to ordinary grassmannians, as well as isotropic extensions to orthogonal and symplectic grassmannians.
\end{rem}

\subsection{Pullbacks of tautological bundles}\label{Sec Pullbacks}

If $X=G(m,V_n)$, we denote by $\mc S$ (or $\mc S_{X}$ if necessary) the tautological bundle of rank $m$ on $X$. By definition $\mc S$ is a subbundle of $V_n\otimes \mc O_X$. Also, we let $\mc S^\perp$ denote the bundle $((V_n\otimes \mc O_X)/\mc S)^*$ on $X$, which is the tautological bundle on $G(n-m,V_n^*)$. Here $(\cdot)^*$ stands for dual bundle. Note that $(\mc S^\perp)^\perp=\mc S$ holds. Below we refer to both $\mc S$ and $\mc S^\perp$ as {\bf tautological bundles} on $X$.

Recall that the grassmannian $G(m,V_n)$ represents the functor
$$
\GG(m,V_n) : \;Algebraic\; varieties\; \to\; Sets
$$
which sends an algebraic variety $Z$ to the set $\GG(m,V_n)(Z)$ of rank-$m$ subbundles of $V_n\otimes \mc O_Z$. In other words, we have a bijection ${\rm Hom}(Z,G(m,V_n))\cong\GG(m,V_n)$ via which $\varphi\in {\rm Hom}(Z,G(m,V_n))$ is identified with the subbundle $\varphi^* \mc S\subset V_n\otimes \mc O_Z$. As a consequence, the subbundle $\varphi^* \mc S\subset V_n\otimes \mc O_Z$ determines the morphism $\varphi$. A similar statement holds for the subbundle $\varphi^*(\mc S^\perp)\subset V_n^*\otimes \mc O_Z$.

If $X=GO(m,V_n)$ (respectively, $X=GS(m,V_n)$), then $\mc S$ stands for the tautological bundle of rank $m$ on $X$. In this situation the bundle $\mc S^\perp$ is defined as the subbundle of $V_n\otimes\mc O_X$ orthogonal to $\mc S$ and we have $\mc S\subset \mc S^\perp\subset V_n\otimes \mc O_X$. We reserve the name tautological bundle for $\mc S$ but not for $\mc S^\perp$. The grassmannian $X$ represents the functor $\GG\OO(m,V_n)$ (respectively, $\GG\SS(m,V_n)$) sending an algebraic variety $Z$ to the set of rank-$m$ isotropic subbundles of $V_n\otimes \mc O_Z$. Any such subbundle determines a morphism $\varphi\in {\rm Hom}(Z,X)$.

In the following statement we describe the pullbacks $\varphi^*\mc S_Y$ of tautological bundles on $Y$ for linear embeddings $X\stackrel{\varphi}{\hw} Y$ of grassmannians as in Theorems \ref{Theo Embed Same Type} and \ref{Theo Embed Mix Type}.

\begin{prop}\label{Prop Subbund to lin emb}
Let $X\stackrel{\varphi}{\hw} Y$ be a linear embedding, where $X$ and $Y$ are non-spinor grassmannians and in addition $X$ is not isomorphic to $GO(m,V_{2m+3})$ or $GO(m,V_{2m+4})$ for any $m\geq 3$.
\begin{enumerate}
\item[{\rm (i)}] If the embedding $\varphi$ does not factor through a projective space or a standard quadric, and $\mc S_Y$ is a tautological bundle on $Y$, then the pullback $\varphi^*\mc S_Y$ is isomorphic to a direct sum of a trivial bundle with $\mc S_X$ or $\mc S_X^\perp$, where $\mc S_X$ is a tautological bundle on $X$.
\item[{\rm (ii)}] If $Y$ is an ordinary grassmannian and $\varphi$ factors through a projective space, then one of pullbacks $\varphi^*\mc S_Y$ or $\varphi^* \mc S_Y^\perp$ is isomorphic to the direct sum of $\mc O_X(-1)$ with a trivial bundle. The same holds for the pullback $\varphi^*\mc S_Y$ if $Y$ is an orthogonal or symplectic grassmannian, under the assumption that $\varphi$ factors respectively through a standard quadric or through a standard symplectic projective space.
\item[{\rm (iii)}] If $Y$ is an orthogonal or symplectic grassmannian and $\varphi$ factors through a projective space which is not contained in a standard quadric or, respectively, in a standard symplectic projective space, then $\varphi^*\mc S_Y$ is isomorphic to the direct sum of $(V_X\otimes \mc O_X)/\mc O_X(-1)$ with a trivial bundle.
\end{enumerate}
\end{prop}

\begin{proof}
Case-by-case verification using the explicit form of the embeddings given in Theorems \ref{Theo Embed Same Type} and \ref{Theo Embed Mix Type}.
\end{proof}

Let us illustrate how the subbundle $\varphi^*\mc S_Y\subset V_s\otimes \mc O_X$ recovers the linear embedding $X\stackrel{\varphi}{\hw} Y$. For instance, assume that $X\stackrel{\varphi}{\hw} Y$ is a combination of standard and isotropic extensions
$$
X=GO(m,V_n) \stackrel{\tau_X}{\hw} G(m,V_n) \stackrel{\sigma}{\hw} G(l,V_r) \stackrel{\iota}{\hw} GO(l,V_s) = Y\;,
$$
$\sigma$ being a non-strict standard extension given by $\sigma(U)=U^\perp\oplus W$, where $U^\perp\subset V_n^*$, $W\subset V_r$ is a fixed subspace together with a monomorphism $V_n^*\oplus W \hw V_r$. The isotropic extension $\iota$ turns $V_n^*\oplus W$ into an isotropic subspace of the orthogonal space $V_s$. 

We have $\varphi^* \mc S_Y \cong \mc S^\perp_X \oplus (W\otimes \mc O_X)$. To read back $\varphi$ from the monomorphism 
\begin{gather}\label{For phi zwezda ot S}
\varphi^* \mc S_Y \cong \mc S^\perp_X \oplus (W\otimes \mc O_X) \hw V_s\otimes \mc O_X\;,
\end{gather}
observe that simply 
\begin{gather}\label{For phiU SU}
\varphi(U)=(\varphi^*\mc S_Y)_U
\end{gather}
where $U\in X$ and $(\varphi^*\mc S_Y)_U$ denotes the geometric fibre of the bundle $\varphi^*\mc S_Y$ the point $U$. It is the morphism (\ref{For phi zwezda ot S}), not just the bundle $\varphi^*\mc S_Y$, which ensures that the map (\ref{For phiU SU}) coincides with $\varphi$. In particular, note that $\iota(V_n^*\oplus W)$ is the union of the images in $V_s$ of all geometric fibres of $\varphi^*\mc S_Y$, and the space $W$ is the intersection of all such images.

If $\varphi$ instead equals a mixed combination of standard and isotropic extensions of the form
$$
X=GO(m,V_n) \stackrel{\tau_X}{\hw} G(m,V_n) \stackrel{\sigma}{\hw} G(l,V_r) \stackrel{\iota}{\hw} GS(l,V_s) = Y\;,
$$
then the same holds with the only change that the form on $V_s$ is now symplectic.

\section{Special and maximal linear embeddings}\label{Sec Special Emb}

In this subsection we complete the classification of linear embeddings of grassmannians $X\stackrel{\varphi}{\hw} Y$ for $Y\not\cong GO(m,V_{2m})$. As a corollary we classify maximal linear embeddings whose targets are not projective spaces, quadrics, or spinor grassmannians.

\subsection{Spinor grassmannians}\label{Sec SpinorGrass}

We first define certain linear embeddings of grassmannians into a spinor grassmannian. 

Let $V_{2m}$, $m\geq 5$, be a fixed vector space endowed with a non-degenerate symmetric bilinear form. Let $X:=GO(m,V_{2m})$ and $\bar X\cong X$ be the two connected components of the variety of maximal isotropic subspaces of $V_{2m}$. Fix $U_m\in X$ and set $V_m:= U_m$. For every splitting $V_{2m}=U_m\oplus U_m^*$ of $V_{2m}$ into a sum of maximal isotropic subspaces, and any $1\leq k\leq (m-1)/2$, we define the embedding
\begin{gather}\label{For Special Embed theta mr even}
\theta^{2k}_m: G(m-2k,V_m) \hookrightarrow GO(m,V_{2m}) \;,\; U\mapsto U\oplus (U^\perp\cap U_m^*)\;.
\end{gather}

Next, fix $\bar U_m\in\bar X$ and set $V_m:= \bar U_m$. For every splitting $V_{2m}=\bar U_m\oplus \bar U_m^*$ and any $1\leq k\leq (m-1)/2$, we define the embedding
\begin{gather}\label{For Special Embed theta mr odd}
\theta^{2k-1}_m: G(m-2k+1, V_m) \hookrightarrow GO(m,V_{2m}) \;,\; U\mapsto U\oplus (U^\perp\cap \bar U_m^*)\;.
\end{gather}

The embeddings $\theta_m^{r}$, $1\leq r\leq m-1$, are linear because they send projective lines on $G(m-r,V_m)$ to projective lines on $GO(m,V_{2m})$, which is verified in a straightforward manner.

With help of the embedding $\theta_m^2$ we are now able to classify linear embeddings of spinor grassmannians into non-spinor grassmannians. Note that the linear embeddings between spinor grassmannians are classified in Theorem \ref{Theo Embed Same Type},(ii).

\begin{prop}\label{Prop embed spinor through proj}
Let $Y=G(l,V_s)$ or $Y=GS(l,V_s)$. Every linear embedding $GO(m,V_{2m}) \stackrel{\varphi}{\hookrightarrow} Y$ factors through a projective space. Every linear embedding $GO(m,V_{2m}) \stackrel{\varphi}{\hookrightarrow} GO(l,V_s)$, where $s\ne 2l,2l+1,2l+2$, factors through a projective space or a standard quadric on $GO(l,V_s)$.
\end{prop}

\begin{proof}
Assume first that $Y=G(l,V_s)$. Fix $V_m\in GO(m,V_{2m})$ and let $\theta:=\theta_m^2$. We consider the composition
$$
G(m-2,V_m) \stackrel{\theta}{\hookrightarrow} GO(m,V_{2m}) \stackrel{\varphi}{\hookrightarrow} G(l,V_s) \;.
$$
By Theorem \ref{Theo Embed Same Type}, this composition is either a standard extension or factors through a projective space on $G(l,V_s)$.

We claim that in fact $\varphi\circ\theta$ factors through a projective space. To see this, fix an arbitrary point $y\in G(m-2,V_m)$ and its images $\theta(y)\in GO(m,V_{2m})$ and $\varphi\circ \theta(y)\in G(l,V_s)$. Recall from subsection \ref{Sec Lin Spaces} that in each of the grassmannians $G(m-2,V_m)$, $GO(m,V_{2m})$, and $G(l,V_s)$ there are two families of maximal projective spaces passing through the respective points $y, \theta(y), \varphi\circ \theta(y)$. In $GO(m,V_{2m})$ one family consists of $\PP^{m-1}$-s, the other of $\PP^3$-s, and every space of the first family intersects every space of the second family a copy of $\PP^2$. On the other hand, in both grassmannians $G(m-2,V_m)$ and $G(l,V_s)$ any two maximal projective spaces from different families intersect in a copy of $\PP^1$. This implies that $\varphi$ cannot separate the two families of maximal projective spaces on $GO(m,V_{2m})$ passing through $\theta(y)$, in the sense of sending two spaces from different families on $GO(m,V_{2m})$ to two respective spaces on $G(l,V_s)$ also belonging to different families. Since any standard extension $G(m-2,V_m)\hookrightarrow G(l,V_s)$ separates the two families of maximal projective spaces on $G(m-2,V_m)$, we conclude that the composition $\varphi\circ\theta$ is not a standard extension. Hence its image $\varphi\circ\theta(G(m-2,V_m))$ lies in some maximal projective space $\PP_0^N$ in $G(l,V_s)$ containing the point $\varphi\circ\theta(y)$. 

The space $\PP_0^N$ intersects any other maximal projective space from its own family in $G(l,V_s)$ only at the point $\varphi\circ \theta(y)$. Any maximal $\PP^3$ passing through $\varphi\circ \theta(y)$ in $GO(m,V_{2m})$ intersects $\PP_0^N$ along a copy of $\PP^1$, so all such $\PP^3$-s must be contained in $\PP_0^N$. The same argument shows that $\PP_0^N$ contains all maximal $\PP^{m-1}$-s on $GO(m,V_{2m})$ passing through $\theta(y)$. Hence all projective lines on $GO(m,V_{2m})$ passing through $\theta(y)$ are contained in $\PP_0^N$. By Lemma \ref{Lemma Connect by lines} any two points in $GO(m,V_{2m})$ are connected via a finite chain of projective lines, and the fact that the point $y$ is arbitrary enables us to complete the argument for the case $Y= G(l,V_s)$.

The cases where $Y$ is isomorphic to $GS(l,V_s)$ or $GO(l,V_s)$ are deduced from the case of $G(l,V_s)$ by an argument analogous to the one used in the proof of Theorem \ref{Theo Embed Mix Type},(v). Namely, we compose $\varphi$ with the tautological embedding $\tau_Y:Y\hookrightarrow G(l,V_s)$ and use the fact that the embedding $\tau_Y\circ \varphi:GO(m,V_{2m})\hookrightarrow G(l,V_s)$ factors through a maximal projective space $\PP$ on $G(l,V_s)$. If $Y= GS(l,V_s)$ then the intersection $\PP\cap\tau_Y(Y)$ is a projective space on $Y$ containing the image of $GO(m,V_{2m})$, so the statement holds in this case. If $Y= GO(l,V_s)$ then $\PP\cap\tau_Y(Y)$ is either a projective space or a maximal standard quadric $Q$ on $Y$. Hence, if $GO(m,V_{2m}) \stackrel{\varphi}{\hookrightarrow} Y$ does not factor through a projective space then $\varphi$ factors through a standard quadric on $Y$, and the embedding of $GO(m,V_{2m})$ into this quadric is one of the embeddings from Lemma \ref{Lemma embed in quadric}. This completes the proof.
\end{proof}

\subsection{The grassmannians $GO(n-2,V_{2n})$}\label{Sec GOnminus22n G2Spin}

We now classify linear embeddings of $GO(n-2,V_{2n})$ into non-spinor grassmannians. Set $X:=GO(n-2,V_{2n})$ throughout this subsection. First we construct a special linear embedding of $X$ into $G(2,V_{2^{n-1}})$.

Let $S:=GO(n,V_{2n})$. The vector space $V_S:=H(S,\mc O_S(1))^*$ has dimension $2^{n-1}$, and
$$
\pi_S:S\to \PP(V_{S})
$$
is a minimal projective embedding of $S$. Set $Y:=G(2,V_S)\cong G(2,V_{2^{n-1}})$ and recall that this grassmannian can be interpreted as the variety of projective lines in $\PP(V_S)$. Furthermore, from subsection \ref{Sec MaxLin on Spinor} we know that $X$ is the variety of projective lines on $S$, and formula (\ref{For Lines on Spinor}) allows us to parametrize the projective lines on $S$ as
$$
\PP^1_{S,U} :=\{V\in S: U\subset V\} \subset S \quad {\rm for} \quad U\in GO(n-2,V_{2n})= X \;.
$$

\begin{prop}\label{Prop special emb delta n}
The map 
\begin{gather}\label{For special emb delta n}
\delta_n:GO(n-2,V_{2n})\to G(2,V_S)\;,\; U\mapsto \pi_S(\PP^{1}_{S,U})
\end{gather}
is a linear embedding and its image is the variety of projective lines on $S$.
\end{prop}

\begin{proof}
It is clear that $\delta_n$ is an injection, and it is routine to check that $\delta_n$ is an embedding of algebraic varieties. The fact that the image of $\delta_n$ is exactly the variety of projective lines on $S$ follows directly from the construction. To prove that $\delta_n$ is linear we will show that it sends every projective line on $X$ to a projective line on $Y$. A projective line on $X$ has the form
$$
\PP^1_{X,U_{n-3}\subset U_{n-1}}=\{U\in X: U_{n-3}\subset U\subset U_{n-1}\}
$$
for some fixed isotropic subspaces $U_{n-3}\subset U_{n-1}\subset V_{2n}$ of the indicated dimensions. The points on $\PP^1_{X,U_{n-3}\subset U_{n-1}}$ correspond precisely to those projective lines in $S$ which contain the unique maximal isotropic subspace $U_n\subset V_{2n}$ satisfying $U_{n-1}\subset U_n\in S$.
The union of these projective lines on $S$ is the following projective plane on $S$:
$$
\PP^2_{S,U_{n-3}\subset \bar U_{n}}= \{V\in GO(n,V_{2n}): U_{n-3}\subsetneq (V\cap U_{n-1})\} = \bigcup\limits_{U\in \PP^1_{X,U_{n-3}\subset U_{n-1}}} \PP^1_{S,U}\;,
$$
where $\bar U_{n}\subset V_{2n}$ is the unique maximal isotropic subspace such that $U_n\cap\bar U_n=U_{n-1}$. (Note that $\bar U_n\notin S$.) It follows that the image of $\PP^1_{X,U_{n-3}\subset U_{n-1}}$ in $Y$ under $\delta_n$ is given by the variety of projective lines on the projective plane $\PP^2_{S,U_{n-3}\subset \bar U_{n}}$ containing the point $U_n$. Since this variety is a projective line on $Y$ the proof is complete.
\end{proof}

Recall that the spinor grassmannian $S$ used to define the embedding $\delta_n$ is a connected component of the variety of maximal isotropic subspaces of $V_{2n}$. If one uses the other connected component $\bar S$ of the same variety, one obtains an embedding $GO(n-2,V_{2n})\hw G(2,V_{\bar S})$ which is a composition of an isomorphism $G(2,V_{S})\stackrel{\sim}\to G(2,V_{\bar S})$ with the embedding $\delta_n$.

Let $\mc S_2:=\mc S_{G(2,V_S)}$ be the tautological bundle on $G(2,V_S)$. The bundle $\delta_n^*\mc S_2$ is the $SO(V_{2n})$-linearized vector bundle of rank 2 whose fibre at $U\in X$ is the 2-dimensional space $\pi_S(\PP^1_{S,U})$.

The embedding $\delta_n$ may factor through the tautological embedding of some orthogonal or symplectic grassmannian in $G(2,V_S)$. This occurs exactly when $V_S$ admits a non-degenerate, respectively symmetric or skew-symmetric, bilinear form for which all two dimensional subspaces of $V_S$ corresponding to projective lines on $S$ are isotropic. In fact, such a bilinear form always exists by Proposition \ref{Prop special emb delta 4k2 S 4k O}. If $\delta_n$ factors through $GO(2,V_S)$, we denote by $\delta_n^O$ the embedding such that
\begin{gather}\label{For delta O}
\delta_n : GO(n-2,V_{2n}) \stackrel{\delta_n^O}{\hw} GO(2,V_S) \stackrel{\tau}{\hw} G(2,V_S) \;.
\end{gather}
If $\delta_n$ factors through $GS(2,V_S)$, we denote by $\delta_n^S$ the embedding such that
\begin{gather}\label{For delta S}
\delta_n : GO(n-2,V_{2n}) \stackrel{\delta_n^S}{\hw} GS(2,V_S) \stackrel{\tau}{\hw} G(2,V_S) \;.
\end{gather} 

The following theorem is our main result in this subsection.

\begin{theorem}\label{Theo Embed GOcodim2}
The following statements hold for every $n\geq 4$.
\begin{enumerate}
\item[{\rm (i)}] Every linear embedding $GO(n-2,V_{2n})\hookrightarrow G(l,V_r)$ factors through a projective space, or factors through the tautological embedding and a standard extension $\sigma$ as
$$
GO(n-2,V_{2n}) \stackrel{\tau_X}{\hookrightarrow} G(n-2,V_{2n}) \stackrel{\sigma}{\hookrightarrow} G(l,V_r) \;,
$$
or factors through the embedding $\delta_n$ and a standard extension $\sigma'$ as
$$
GO(n-2,V_{2n}) \stackrel{\delta_n}{\hookrightarrow} G(2,V_{2^{n-1}}) \stackrel{\sigma'}{\hookrightarrow} G(l,V_r) \;.
$$
\item[{\rm (ii)}] Every linear embedding $GO(n-2,V_{2n})\stackrel{\varphi}{\hookrightarrow} GO(l,V_r)$, $l<\lfloor\frac{r-1}{2}\rfloor$, is a standard extension, or factors through a standard quadric, or factors as
$$
GO(n-2,V_{2n}) \stackrel{\psi}{\hookrightarrow} G(l,V_{\lfloor r/2\rfloor}) \stackrel{\iota}{\hookrightarrow} GO(l,V_r)
$$
where $\iota$ is an isotropic extension and $\psi$ is one of the embeddings of part (i), or factors as
$$
GO(n-2,V_{2n}) \stackrel{\delta_{n}^O}{\hookrightarrow} GO(2,V_{2^{n-1}}) \stackrel{\sigma}{\hookrightarrow} GO(l,V_r)
$$
for a standard extension $\sigma$.
\item[{\rm (iii)}] Every linear embedding $GO(n-2,V_{2n})\stackrel{\varphi}{\hookrightarrow} GS(l,V_{2r})$ factors as
$$
GO(n-2,V_{2n}) \stackrel{\psi}{\hookrightarrow} G(l,V_{r}) \stackrel{\iota}{\hookrightarrow} GS(l,V_{2r})
$$
where $\iota$ is an isotropic extension and $\psi$ is one of the embeddings of part (i), or factors as
$$
GO(n-2,V_{2n}) \stackrel{\delta_{n}^S}{\hookrightarrow} GS(2,V_{2^{n-1}}) \stackrel{\sigma}{\hookrightarrow} GS(l,V_{2r}).
$$
for a standard extension $\sigma$.
\end{enumerate}
\end{theorem}

\begin{proof}
We assume first $n\geq 5$. For part (i), let $\varphi:X\hookrightarrow Y=G(l,V_r)$ be a linear embedding which does not factor through a projective space. We borrow the strategy of \cite{Penkov-Tikhomirov-Lin-ind-Grass} and consider the images of maximal standard quadrics on $X$ under $\varphi$. Such a maximal quadric is $4$-dimensional, has the form $Q:=Q^4_{X,U_{n-3}}$ for some $U_{n-3}\in GO(n-3,V_{2n})$, and is isomorphic to the ordinary grassmannian $G(2,V_4)$. There are two options for the restriction $\varphi_{\vert_Q}$:

Case 1) $\varphi_{\vert_{Q}}$ factors through a projective space.

Case 2) $\varphi_{\vert_{Q}}$ is a standard extension.

In fact, exactly one of these two options holds simultaneously for all maximal standard quadrics on $X$. This is due to the irreducibility of the family of maximal standard quadrics in $X$ and the existence of two irreducible families of maximal projective spaces on $Y$; a detailed argument is given in the proof of \cite[Proposition 3.15]{Penkov-Tikhomirov-Lin-ind-Grass}.

In case 1) the situation is analogous to the one considered in Theorem \ref{Theo Embed Same Type}, (iii), and we deduce that $\varphi$ factors through the tautological embedding $\tau_X$ as $\varphi=\sigma\circ \tau_X$, where $\sigma$ is a standard extension of ordinary grassmannians.

We suppose now that case 2) holds, so that $\varphi_{\vert_{Q}}$ is a standard extension for every maximal standard quadric $Q$ on $X$. We will prove that $\varphi$ factors through $\delta_n:X\hookrightarrow G(2,V_{2^{n-1}})$. For this we need an auxiliary construction.

Any projective line $L:=\PP^1_{X,U_{n-3}\subset U_{n-1}}$ on $X$ is equal to the intersection of three uniquely determined maximal projective spaces on $X$, namely,
$$
\PP^1_{X,U_{n-3}\subset U_{n-1}} = \PP^{n-2}_{X, U_{n-1}} \cap \PP^2_{X,U_{n-3}\subset U_n} \cap \PP^2_{X,U_{n-3}\subset \bar U_n} 
$$
where $U_n\in S$ and $\bar U_n\in \bar S$ form the unique pair of maximal isotropic subspaces of $V_{2n}$ containing $U_{n-1}$. Moreover, the line $L$ is obtained as the intersection of any two of the above three maximal projective spaces on $X$. The spaces $\PP^2_{X,U_{n-3}\subset U_n}$ and $\PP^2_{X,U_{n-3}\subset \bar U_n}$ are maximal projective spaces on the standard quadric $Q^4_{X,U_{n-3}}\subset X$. For any $U\in S\cup \bar S$ we denote by $Z_U\subset X$ the image of the isotropic extension $G(n-2,U)\hookrightarrow GO(n-2,V_{2n})=X$. Then $\PP^{n-2}_{X, U_{n-1}}$ and $\PP^2_{X,U_{n-3}\subset U_n}$ are maximal projective spaces on $Z_{U_n}$, while $\PP^{n-2}_{X, U_{n-1}}$ and $\PP^2_{X,U_{n-3}\subset U_n}$ are such on $Z_{\bar U_n}$. We have $\PP^{n-2}_{X,U_{n-1}}=Z_{U_n}\cap Z_{\bar U_n}$.

The image of $L$ under $\varphi$ has the form 
$$
\varphi(L)=\PP^1_{Y,W_{l-1}\subset W_{l+1}} = \PP^{r-l}_{Y, W_{l-1}} \cap \PP^l_{Y,W_{l+1}}
$$
for a unique flag of subspaces $W_{l-1}\subset W_{l+1}\subset V_r$. The image under $\varphi$ of each of three spaces $\PP^{n-2}_{X, U_{n-1}}$, $\PP^2_{X,U_{n-3}\subset U_n}$ and $\PP^2_{X,U_{n-3}\subset \bar U_n}$ is contained in exactly one of the subspaces $\PP^{r-l}_{Y,W_{l-1}}$ and $\PP^{l}_{Y,W_{l+1}}$. Thus at least two of the three spaces are mapped to the same maximal projective space on $Y$. Consequently, the image under $\varphi$ of at least one of the three grassmannians $Q^4_{X,U_{n-3}}$, $Z_{U_n}$, $Z_{\bar U_n}$ is contained in a projective space on $Y$. If all three grassmannians are mapped to the same maximal projective space on $Y$ then $\varphi$ factors through a projective space on $Y$. This case is excluded by our hypothesis, so exactly one of $Q^4_{X,U_{n-3}}$, $Z_{U_n}$, $Z_{\bar U_n}$ is mapped by $\varphi$ into a projective space on $Y$. The possibility that $\varphi$ maps $Q^4_{X,U_{n-3}}$ to a projective space is also excluded, because is was considered as case 1). Hence the restriction of $\varphi$ to one of the grassmannians $Z_{U_n}$ and $Z_{\bar U_n}$ is a standard extension, while the restriction to the other grassmannian factors through a projective space. Up to change of notation, we may assume that $\varphi_{\vert_{Z_{\bar U_n}}}$ is a standard extension.

To summarize, it remains to consider the case where $\varphi_{\vert_{Q^4_{X,U_{n-3}}}}$ and $\varphi_{\vert_{Z_{\bar U_n}}}$ are standard extensions while $\varphi_{\vert_{Z_{U_n}}}$ factors through a projective space. By the irreducibility of the parameter spaces, this property holds for every $U_{n-3}\in GO(n-3,V_{2n})$, every $U_{n}\in S$ and every $\bar U_n\in \bar S$. In this setting we will construct a linear embedding $\ul{\pi}$ of the spinor grassmannian $S$ into the projective space $\PP(V_r^*)$, which will help us see that $\varphi$ factors through $\delta_n$.

For every $U_n\in S$ the image $\varphi(Z_{U_n})$ is contained in a maximal projective space on $Y$ of the form $\PP^l_{Y,W_{l+1}}$ or $\PP^{r-l}_{Y,W_{l-1}}$. These two cases are analogous and interchangeable by the use of the isomorphism $G(l,V_r)\cong G(r-l,V_r^*)$. We assume $\varphi(Z_{U_n})\subset \PP^l_{Y,W_{l+1}}$, which yields a morphism
$$
\pi:S\to G(l+1,V_r)\;,\; U_n\mapsto W_{l+1} \;.
$$
On the other hand, for every $\bar U_n\in \bar S$ the restriction of $\varphi$ to $Z_{\bar U_n}= G(n-2,\bar U_n)$ is a standard extension which is strict because of the above choice of $\PP^l_{Y,W_{l+1}}$ instead of $\PP^{r-l}_{Y,W_{l-1}}$. Thus there is an inclusion $\nu_{\bar U_n}:\bar U_n\hw V_r$ and a subspace $W_{\bar U_n}\subset V_r$ of dimension $l+2-n$ such that
$$
\varphi(U)=W_{\bar U_n}\oplus \nu_{\bar U_n}(U)
$$
for $U\in Z_{\bar U_n}$. We obtain a morphism
$$
\rho : \bar S \to G(l+2,V_r)\;,\quad \bar U_n\to W_{\bar U_n}\oplus \nu_{\bar U_n}(\bar U_n) \;.
$$

For $U_n\in S$ and $\bar U_n\in\bar S$ the intersection $Z_{U_n}\cap Z_{\bar U_n}$ in $X$ is empty or is a maximal projective space on $X$ of the form $\PP^{n-2}_{X,U_{n-1}}$ whenever $U_{n-1}=U_n\cap\bar U_n$ has dimension $n-1$. Therefore, there is a third morphism
$$
\pi': G(n-1,V_{2n}) \to G(l+1,V_r) \;,\quad U_{n-1}\to \pi(U_n)=W_{\bar U_n}\oplus \nu_{\bar U_n}(U_{n-1}) \;,
$$
where $U_n\in S$ and $\bar U_n\in\bar S$ is the unique pair such that $U_n\cap \bar U_n=U_{n-1}$. We note that $\pi$ and $\pi'$ have the same image, and $\pi(U_n)=\pi'(U_{n-1})$ if and only if $U_{n-1}\subset U_n$.

We claim that $\rho$ is a constant morphism and $\pi$ is an linear embedding. 

To prove that the morphism $\rho$ is constant it suffices to show that it is constant on any projective line on $\bar S$. Fix $\bar U_n,\bar U'_n\in S$ such that $\bar U_n\cap \bar U'_n=:U_{n-2}\in X$, so that $\bar U_n,\bar U'_n$ are two distinct points on the projective line $\PP^{1}_{\bar S,U_{n-2}}$. The element $U_{n-2}$ also determines a projective line on $S$, and we fix $U_n,U'_{n}\in \PP^1_{S,U_{n-2}}$ such that $U_n\cap U'_n=U_{n-2}$. Then each of the intersections 
$$
U^{(1)}_{n-1}:=U_n\cap \bar U_n \;,\quad U^{(2)}_{n-1}:=U'_n\cap \bar U_n \;,\quad U^{(3)}_{n-1}:=U_n\cap \bar U'_n \;,\quad U^{(4)}_{n-1}:=U'_n\cap \bar U_n
$$
has dimension $n-1$ and contains $U_{n-2}$. Hence we have
$$
\varphi(U_{n-2})\subset  \pi(U_n)\cap \pi(U'_n) \;,\quad  \pi(U_n)+\pi(U'_n)\subset \rho(\bar U_n)\cap \rho(\bar U'_n) \;.
$$
Each of $\pi(U_n)$ and $\pi(U_n)$ has codimension $1$ in each of $\rho(\bar U_n)$ and $\rho(\bar U'_n)$, therefore
$$
\rho(\bar U_n)= \rho(\bar U'_n) \;\tst\; \pi(U_n)\ne \pi(U'_n) \;\tst\; \pi(U_n)\cap \pi(U'_n)=\varphi(U_{n-2}) \;.
$$
In other words, the restriction of $\rho$ to $\PP^1_{\bar S,U_{n-2}}$ is constant if and only if the restriction of $\pi$ to $\PP^1_{S,U_{n-2}}$ is a linear embedding.

Since $U^{(1)}_{n-1}$ and $U^{(2)}_{n-1}$ are distinct hyperplanes in $\bar U_{n}$ they define distinct maximal projective spaces on $Z_{\bar U_n}$. It follows that $\pi'(U^{(1)}_{n-1}) \ne  \pi'(U^{(2)}_{n-1})$ because any standard extension of ordinary grassmannians maps distinct maximal projective spaces on its domain to distinct maximal projective spaces on the target grassmannian. Hence
$$
\pi(U_n)=\pi'(U^{(1)}_{n-1}) \ne  \pi'(U^{(2)}_{n-1}) = \pi(U'_n) \;.
$$
We can conclude that the restriction of $\pi$ to every projective line on $S$ is a linear embedding, and $\rho$ is a constant morphism.

Since the restriction of $\pi$ to any projective line is linear, the map $\pi$ is itself linear. To prove that $\pi$ is a linear embedding it remains to show that it is injective. Let $V_{l+2}:=\rho(\bar S)\in G(l+2,V_r)$ be the image of $\rho$. For every $\bar U_n\in \bar S$ we have $\varphi(Z_{\bar U_n})\subset \sigma_1(G(l,V_{l+2}))\subset Y$, where $\sigma_1:G(l,V_{l+2})\hw G(l,V_r)$ is the canonical standard extension. It follows that the embedding $\varphi$ factors through the standard extension $\sigma_1$, i.e., there exists a linear embedding $\psi:X\to Z:=G(l,V_{l+2})$ such that $\varphi=\sigma_1\circ \psi$. This allows us to reduce to the case where $V_r=V_{l+2}$ and $\varphi=\psi$.

Suppose $\pi$ is not injective. In other words, $\pi(U_n)=\pi(U''_n)=:W_{l+1}$ for some $U_n,U''_n\in S$ which do not belong to a projective space on $S$. Then $\dim (U_n\cap U''_n)=n-2k$ for some $k>1$. Let $U_{n-2},U''_{n-2}\in X$ be such that $U_{n-2}\subset U_n$, $U''_{n-2}\subset U_n$ and $U_n\cap U''_n = U_{n-2}\cap U''_{n-2}$. The images $\varphi(U_{n-2})$ and $\varphi(U''_{n-2})$ are then two distinct hyperplanes in $\pi(U_n)$. Hence the intersection $W_{l-1}:=\varphi(U_{n-2})\cap\varphi(U''_{n-2})$ has dimension $l-1$, and we deduce that the two images $\varphi(U_{n-2})$ and $\varphi(U''_{n-2})$ belong to the projective line $\PP^1_{Y,W_{l-1}\subset W_{l+1}}$ on $Y=G(l,V_{l+2})$. Since $\varphi$ is a linear embedding, $U_{n-2}$ and $U''_{n-2}$ belong to a projective line $\PP^1_{X,U_{n-3}\subset U_{n-1}}$ on $X$. In particular $\dim (U_{n-2}\cap U''_{n-2})=n-3$, a contradiction.

This enables us to conclude that $\pi$ is a linear embedding, and we can define the previously announced embedding $\ul{\pi}$ as the composition of $\pi$ with the duality isomorphism
$$
\ul{\pi}: S \stackrel{\pi}{\hw} G(l+1,V_{l+2}) \cong \PP(V^*_{l+2}) \;.
$$
Similarly, let $\ul{\varphi}$ be the composition
$$
\ul{\varphi}: X \stackrel{\varphi}{\hw} G(l,V_{l+2})\cong G(2,V_{l+2}^*) \;.
$$

The explicit form of the embedding $\ul{\pi}$ yields immediately to our desired factorization property of the embedding $\varphi$. Indeed, $\ul\pi:S\to V_{l+2}^*$ induces a linear map $\ul{\hat\pi}: V_S\hookrightarrow V_{l+2}^*$, which in turn induces a standard extension $\sigma_2:G(2,V_S)\hw G(2,V_{l+1}^*)$ through which $\ul{\varphi}$ factors. The resulting embedding $X\hw G(2,V_S)$ is
\begin{gather*}
\ul{\varphi}(U_{n-2})=\ul{\pi}(\PP^1_{S,U_{n-2}})=\PP(\ul{\pi}(U_n)+\ul{\pi}(U))=\PP(\ul{\varphi}(U_{n-2}))\in \sigma_2(G(2,V_S))\;,
\end{gather*}
which is identical with $\delta_n$ by (\ref{For special emb delta n}). Therefore $\ul{\varphi}=\sigma_2\circ\delta_n$, and the initial embedding $\varphi$ equals the composition
$$
\varphi: X\stackrel{\delta_n}{\hw} G(2,V_S) \stackrel{\sigma_2}{\hw} G(2,V_{l+2}^*) \cong G(l,V_{l+2}) \stackrel{\sigma_1}{\hw} G(l,V_r)=Y\;.
$$
This completes the proof of part (i) for $n\geq 5$.

It remains to consider the case $n=4$. Here the embedding $\delta_4$ is related to the tautological embedding $\tau_X$ via the isomorphism $GO(1,V_8)\cong GO(4,V_8)$, and the two factorizations in part (i) are the same. The result is that every linear embedding $GO(2,V_8)\hookrightarrow Y$, where $Y$ is an ordinary or a symplectic grassmannian, factors through the tautological embedding into $G(2,V_8)$. In case $Y\cong GO(l,V_r)$, there are also the options for the embedding to be a standard extension and to factor through a standard quadric. This settles parts (i),(ii),(iii) for $n=4$.

Parts (ii) and (iii) for $n\geq 5$ follow from part (i) by use of the composition $\tau_Y\circ \varphi$ of a given embedding $\varphi$ of $X$ into an isotropic grassmannian $Y$ with the tautological embedding $\tau_Y$ of $Y$.
\end{proof}

\subsection{The grassmannians $GO(n-1,V_{2n+1})$}\label{Sec GOcodim1odd}

In this subsection we classify linear embeddings of the grassmannian $X:=GO(n-1,V_{2n+1})$ into non-spinor grassmannians for $n\geq 3$. This case is similar and closely related to the case of $GO(n-2,V_{2n})$ considered in the previous subsection.

There is a special linear embedding of $X$, obtained as the composition
\begin{gather}\label{For special emb odddelta}
\delta^1_{n+1}: GO(n-1,V_{2n+1})\stackrel{\sigma_1}{\hookrightarrow} GO(n-1,V_{2n+2}) \stackrel{\delta_{n+1}}{\hookrightarrow} G(2,V_{2^n}) = G(2,V_{GO(n,V_{2n+1})}) 
\end{gather}
of a standard extension $\sigma_1$ and the embedding $\delta_{n+1}$ given in (\ref{For special emb delta n}). If the embedding $\delta^1_{n+1}$ factors through a tautological embedding $GO(2,V_{2^n})\hw G(2,V_{2^n})$ or $GS(2,V_{2^n})\hw G(2,V_{2^n})$, we obtain respective embeddings $\delta_{n+1}^{1,O}$ and $\delta^{1,S}_{n+1}$ such that
\begin{gather}\label{For special emb odddelta OS}
\begin{array}{rl}
\delta^1_{n+1}:& G(n-1,V_{2n+1})\stackrel{\delta^{1,O}_{n+1}}{\hw} GO(2,V_{2^n}) \hw G(2,V_{2^n}) \;,\\
\delta^1_{n+1}:& G(n-1,V_{2n+1})\stackrel{\delta^{1,S}_{n+1}}{\hw} GS(2,V_{2^n}) \hw G(2,V_{2^n}) \;.
\end{array}
\end{gather}
The existence of such factorizations for $\delta^1_{n+1}$ is equivalent to the existence of a symmetric, or respectively symplectic, non-degenerate bilinear form on $V_{GO(n,V_{2n+1})}$ for which every projective line on $GO(n,V_{2n+1})$ is the projectivization of an isotropic plane in $V_{GO(n,V_{2n+1})}$.

\begin{theorem}\label{Theo Embed GOcodim1odd}
The following statements hold for $n\geq 3$.
\begin{enumerate}
\item[{\rm (i)}] Every linear embedding $GO(n-1,V_{2n+1})\hookrightarrow G(l,V_r)$ factors as
\begin{gather}\label{For GO tau G s G}
GO(n-1,V_{2n+1}) \stackrel{\tau_X}{\hw} G(n-1,V_{2n+1}) \stackrel{\sigma}{\hw} G(l,V_{r})
\end{gather}
or as
$$
GO(n-1,V_{2n+1}) \stackrel{\delta^1_{n+1}}{\hookrightarrow} G(2,V_{2^n}) \stackrel{\sigma'}{\hw} G(l,V_{r})\;,
$$
where $\sigma$ and $\sigma'$ are standard extensions, or factors through a projective space.
\item[{\rm (ii)}] Every linear embedding $GO(n-1,V_{2n+1})\hookrightarrow GO(l,V_r)$, $l\leq\lfloor\frac{r-1}{2}\rfloor$, is a standard extension, or factors through a standard quadric, or factors as
$$
GO(n-1,V_{2n+1}) \stackrel{\psi}{\hw} G(l,V_{k}) \stackrel{\iota}{\hw} GO(l,V_{r})
$$
where $\psi$ is one of the embeddings in part {\rm(i)} and $\iota$ is an isotropic extension, or factors as
$$
GO(n-1,V_{2n+1}) \stackrel{\delta^{1,O}_{n+1}}{\hookrightarrow} GO(2,V_{2^n}) \stackrel{\sigma'}{\hw} GO(l,V_{r})
$$
where $\sigma,\sigma'$ are standard extensions and $\delta^{1,O}_{n+1}$ is the embedding given in (\ref{For special emb odddelta OS}).
\item[{\rm (iii)}] Every linear embedding $GO(n-1,V_{2n+1})\hookrightarrow GS(l,V_r)$ factors as
$$
GO(n-1,V_{2n+1}) \stackrel{\psi}{\hw} G(l,V_{k}) \stackrel{\iota}{\hw} GS(l,V_{r})
$$
where $\psi$ is one of the embeddings in part {\rm(i)} and $\iota$ is an isotropic extension, or factors as
$$
GO(n-1,V_{2n+1}) \stackrel{\delta^{1,S}_{n+1}}{\hookrightarrow} GS(2,V_{2^n}) \stackrel{\sigma'}{\hw} GS(l,V_{r})
$$
where $\sigma,\sigma'$ are standard extensions and $\delta^{1,S}_{n+1}$ is the embedding given in (\ref{For special emb odddelta OS}). 
\end{enumerate}
\end{theorem}

\begin{proof}
The proof follows the same ideas as the proof of Theorem \ref{Theo Embed GOcodim2}. We outline the main steps for part (i). Let $\varphi:X:=GO(n-1,V_{2n+1}) \hw G(l,V_r)=:Y$ be a linear embedding. A maximal standard quadric $Q$ on $X$ is 3-dimensional and has the form $Q=Q^3_{X,U_{n-2}}$ given in subsection \ref{Sec MaxProj in GOcodim1odd}. Any 3-dimensional quadric is isomorphic to $GS(2,V_4)$. Hence there are the following two options for the restriction of $\varphi$ to $Q$.

Case 1) $\varphi_{\vert Q}$ factors through a projective space.

Case 2) $\varphi_{\vert Q}$ factors through a composition $\sigma\circ \tau$ where $\tau$ is the tautological embedding of $GS(2,V_4)$ and $\sigma: G(2,V_4)\to Y$ is a standard extension.

In case 1) one can show that the embedding $\varphi$ factors through the tautological embedding $\tau_X:GO(n-1,V_{2n+1})\hw G(n-1,V_{2n+1})$. Thus $\varphi=\sigma\circ\tau_X$ where $\sigma:G(n-1,V_{2n+1})\hw G(l,V_r)$ is a linear embedding. By Theorem \ref{Theo Embed Same Type}, $\sigma$ is a standard extension or factors through a projective space, and we obtain a factorization of the type (\ref{For GO tau G s G}).

In case 2) we claim that the embedding $\varphi$ factors through the embedding $\delta^1_{n+1}$. To show this we construct an auxiliary linear embedding $\pi:S\to \PP(V_r)$ of the spinor grassmannian $S:=GO(n,V_{2n+1})$. The embedding $\pi$ is constructed first for $l=2$, using the fact that $S$ is the parameter space for the family of maximal projective spaces on $X$ (see subsection \ref{Sec MaxProj in GOcodim1odd}). Then one shows that $\varphi$ factors through a standard extension $\sigma_1:X\hw GO(n-1,V_{2n+2})$ using the isomorphism of spinor grassmannians $S\cong GO(n+1,V_{2n+2})=:S'$ and the identification of $GO(n-1,V_{2n+2})$ with the variety of projective lines on $S$ given by the embedding $\delta_{n+1}$. Here $V_{2n+2}:=V_{2n+1}\oplus V_1$ is endowed with a non-degenerate symmetric bilinear form extending the given form on $V_{2n+1}$. Thus $\varphi$ factors through $\delta^1_{n+1}$, and hence (i) holds.

To prove parts (ii) and (iii) one can compose any given linear embedding $X\stackrel{\varphi}{\hw} Y$ into a symplectic or orthogonal non-spinor grassmannian $Y$ with the tautological embedding $\tau_Y$. Then part (i) applies to $\tau_Y\circ \varphi$, and the result follows in a straightforward manner.
\end{proof}

\subsection{Maximal linear embeddings of grassmannians}\label{Sec Maximal Embeddings}

A linear embedding of grassmannians $X\stackrel{\varphi}{\hookrightarrow} Y$ is {\bf maximal} if $\varphi$ is a proper embedding, and does not factor through a grassmannian $Z$ properly contained in $Y$ and properly containing $\varphi(X)$.

\begin{coro}\label{Coro Maximal}
Let $X\stackrel{\varphi}{\hookrightarrow} Y$ be a maximal linear embedding of grassmannians. Assume that $Y$ is not spinor grassmannian, a projective space or a quadric. Then $\varphi$ is one of the following.
\begin{enumerate}
\item[{\rm (a)}] For $Y=G(m,V_n)$ with $1<m<n-1$:
\begin{enumerate}
\item[{\rm (a.1)}] a standard extension $G(k,V_{n-1})\hw G(m,V_n)$ with $l\in\{m,m-1\}$; 
\item[{\rm (a.2)}] a tautological embedding $GO(m,V_n)\hw G(m,V_n)$ or $GS(m,V_n)\hw G(m,V_n)$.
\end{enumerate}

\item[{\rm (b)}] For $Y=GO(m,V_n)$ with $1<m<\lfloor\frac{n-1}{2}\rfloor$:
\begin{enumerate}
\item[{\rm (b.1)}] a standard extension $GO(m,V_{n-1})\hw GO(m,V_n)$,\\ or $GO(m-1,V_{n-2})\hw GO(m,V_n)$;
\item[{\rm (b.2)}] an isotropic extension $G(m,V_{\frac{n}{2}})\hw GO(m,V_n)$ where $n$ is even and $V_{\frac{n}{2}}\subset V_n$ is a maximal isotropic subspace;
\item[{\rm (b.3)}] if $m=2$ and $n=2^{k-1}$ for some $k\geq 1$, the embedding $\delta^O_{k}:GO(k-2,V_{2k}) \hookrightarrow GO(2,V_{2^{k-1}})$ defined in subsection \ref{Sec GOnminus22n G2Spin}, whenever the minimal projective embedding of $GO(k,V_{2k})$ factors through a quadric.
\end{enumerate}

\item[{\rm (c)}] For $Y=GS(m,V_{2n})$ with $1<m<n$:
\begin{enumerate}
\item[{\rm (c.1)}] an isotropic extension $G(m,V_n)\hw GS(m,V_{2n})$ where $V_n\subset V_{2n}$ is a maximal isotropic subspace;
\item[{\rm (c.2)}] a standard extension $GS(l,V_{2n-2})\hw GS(m,V_{2n})$ with $l\in\{m-2,m-1,m\}$;
\item[{\rm (c.3)}] if $m=2$ and $n=2^{k-1}$ for some $k\geq 2$, the embedding $\delta^S_{k}:GO(k-2,V_{2k}) \hookrightarrow GS(2,V_{2^{k-1}})$ given in (\ref{For delta S}) provided $V_{2^{k-1}}$ admits a non-degenerate skew-symmetric bilinear form for which every projective line on $GO(k,V_{2k})\subset \PP(V_{2^{k-1}})$ is isotropic.
\end{enumerate}

\item[{\rm (d)}] A standard extension $GS(m-1,V_{2m-2}) \hw GS(m,V_{2m})$.
\end{enumerate}
\end{coro}

\begin{proof}
The list is derived using the classification of linear embeddings from Theorems \ref{Theo Embed Same Type}, \ref{Theo Embed Mix Type}, \ref{Theo Embed GOcodim2}, \ref{Theo Embed GOcodim1odd}, and Proposition \ref{Prop embed spinor through proj}.
\end{proof}

\section{Equivariance of linear embeddings}\label{Sec Equivar}

A (generalized) flag variety is a coset space of the form $G/P$ where $G$ is a connected semisimple complex algebraic group and $P$ is a parabolic subgroup. The Picard group of $G/P$ is isomorphic to $\ZZ$ if and only if $P$ is a maximal parabolic subgroup of $G$. For a grassmannian $X$ we denote by $G_X$ the simply connected cover of the connected component of the identity of the automorphism group of $X$. Concretely, $G_X= SL(V_n)$ for $X=G(m,V_n)$; $G_X=Spin(V_n)$ for $X=GO(m,V_n)$ with $n\ne 2m+1$; $G_X=Sp(V_n)$ for $X=GS(m,V_n)$ with $n$ even and $m\ne 1$. Then, in all cases we have $X=G_X/P_m$ where $P_m$ is the subgroup of $G_X$ fixing a point in $X$.

The classification of linear embeddings of grassmannians given sections \ref{Sec Lin Emb Grass} and \ref{Sec Special Emb} has the remarkable consequence that almost every linear embedding of grassmannians is equivariant for an appropriate classical group. There are exceptions: embeddings which factor through standard quadrics or standard symplectic projective spaces, as well as certain embeddings of $GO(n-2,V_{2n})$ and $GO(n-1,V_{2n+1})$ into orthogonal or symplectic grassmannians. We will show that these are the only exceptions.

\begin{definition}\label{Def Equivar}
An embedding of grassmannians $X\stackrel{\varphi}{\hookrightarrow} Y$ is $G$-{\bf equivariant} (or simply {\bf equivariant}) if $G$ is a subgroup of $G_X$ acting transitively on $X$ and there is a homomorphism $f:G\to G_Y$ with finite kernel such that $\varphi(gx)=f(g)\varphi(x)$ for $g\in G$ and $x\in X$.
\end{definition}

For any grassmannian $X$ the minimal projective embedding $\pi_X:X\to\PP(V_X)$ is equivariant under $G_X$. Furthermore, note that the image $\varphi(X)\subset Y$ of a $G$-equivariant embedding is a closed orbit of the subgroup $f(G)\subset G_Y$.

\begin{theorem}\label{Theo Equivar gen}
Let $X$ be a grassmannian not isomorphic to $GO(n-2,V_{2n-1})$ or $GO(n-2,V_{2n})$ for $n\geq 4$, and let $Y$ be a non-spinor grassmannian. Then every linear embedding $X\stackrel{\varphi}{\hw} Y$, which does not factor through a projective space or a quadric, is $G_X$-equivariant.
\end{theorem}

Instead of proving Theorem \ref{Theo Equivar gen} directly, we will describe the images of embeddings as in Theorem \ref{Theo Equivar gen} as orbits of suitable groups, and this will imply the claimed equivariance. First, we introduce the following short-hand terminology. Let $G$ be a connected semisimple linear algebraic group. We say that a subgroup $L\subset G$ is an $\mc L$-{\bf subgroup} ($\mc L$ for Levi) if it is a maximal semisimple connected subgroup of a parabolic subgroup of $G$. If an $\mc L$-subgroup $L\subset G$ is infinitesimally simple, we say that $L$ is an $\mc{LS}$-{\bf subgroup}. An algorithm for description of the $\mc{L}$-subgroups of a given simple group is provided by Dynkin in \cite{Dynkin-1952}.

\begin{rem}\label{Rem compose LS}
If follows from standard properties of parabolic subgroups that, if $G_1\stackrel{f_1}{\to}G_2\stackrel{f_2}{\to} G_3$ are homomorphisms of semisimple complex algebraic groups such that $f_1(G_1)\subset G_2$ and $f_2(G_2)\subset G_3$ are $\mc L$-subgroups (respectively, $\mc{LS}$-subgroups) then the image $f_2\circ f_1(G_1)\subset G_3$ is also an $\mc L$-subgroup (respectively, $\mc{LS}$-subgroup).
\end{rem}

\begin{prop}\label{Prop Equivar through LS}
$\;$
\begin{enumerate}
\item[{\rm (i)}] Let $X\stackrel{\varphi}{\to} Y$ be a linear embedding of grassmannians as in Theorem \ref{Theo Equivar gen}, in particular $\varphi$ does not factor through a projective space or a quadric. Then the image $\varphi(X)\subset Y$ is a closed orbit of some $\mc{LS}$-subgroup of $G_Y$, or is properly contained in a closed orbit $Z\subsetneq Y$ of some $\mc{LS}$-subgroup $G_Y$. The latter case occurs if:
\begin{enumerate}
\item[{\rm (i.1)}] $\varphi$ is (mixed) combination of standard and isotropic extensions, or an embedding of an isotropic grassmannian into an ordinary grassmannian; here the embedding $\varphi(X)\hw Z$ is a tautological embedding $GO(m,V_n)\hw G(m,V_n)$ or $GS(m,V_n)\hw G(m,V_n)$;
\item[{\rm (i.2)}] $\varphi$ is a standard extension of orthogonal grassmannians $GO(m,V_n)\hw GO(l,V_{s})$ with $s-n\equiv 1\,({\rm mod}\,2)$; here the embedding $\varphi(X)\hw Z$ a standard extension $GO(m,V_n)\hw GO(l,V_{n+1})$ with $l\in\{m,m+1\}$.
\end{enumerate}
\item[{\rm (ii)}] Let $Y$ be a grassmannian. Every non-constant closed orbit $X\subset Y$ of an $\mc{LS}$-subgroup $L\subset G_Y$ is a grassmannian, and $X$ is not contained in a projective space or a standard quadric on $Y$, unless $X$ is itself a projective space or a standard quadric on $Y$. Furthermore, except in the case where $Y\cong GS(m,V_{2m})$ for some $m$ and $L$ is isomorphic to $SL(k)$ for some $k\geq 2$, the embedding $\varphi$ of $X$ into $Y$ is linear and there are three options:
\begin{enumerate}
\item[{\rm (ii.1)}] $\varphi$ is a standard extension between grassmannians of the same type (ordinary, orthogonal or symplectic);
\item[{\rm (ii.2)}] $X$ is an ordinary grassmannian, $Y$ is a symplectic or non-spinor orthogonal grassmannian, and $\varphi$ factors as
$$
X\cong G(m,V_n) \stackrel{\sigma}{\hw} G(l,V_r) \stackrel{\psi}{\hw} Y
$$
where $\sigma$ is a standard extension and $\psi$ is an isotropic extension.
\item[{\rm (ii.3)}] $X$ is an ordinary grassmannian, $Y=GO(r,V_{2r})$, and $\varphi$ factors as
$$
X\cong G(m,V_n) \stackrel{\sigma}{\hw} G(l,V_r) \stackrel{\theta_r^l}{\hw} Y
$$
for some $1\leq l\leq r-1$, where $\sigma$ is a standard extension and $\theta_r^l$ is the embedding defined in (\ref{For Special Embed theta mr even}),(\ref{For Special Embed theta mr odd}).
\end{enumerate}
\end{enumerate}
\end{prop}

\begin{proof}
(i). We begin with the case where $Y= G(l,V_s)$ and hence $G_Y=SL(V_s)$. If $X= G(m,V_n)$ then $\varphi$ is a standard extension, since we have assumed that $\varphi$ does not factor through a projective space. The image of every standard extension is also the image of a strict standard extension, so we may assume that $\varphi$ is a strict standard extension. Thus we have a decomposition $V_s=V_n \oplus V'$ and $\varphi$ is given by (\ref{For sigma strict standard}). This embedding is clearly $SL(V_n)$-equivariant for the homomorphism $f_{n,s}:SL(V_n)\hw SL(V_s)$ whose image consists of the elements preserving the decomposition $V_s=V_n \oplus V'$ and acting by identity on $V'$. Moreover, this image is an $\mc{LS}$-subgroup of $SL(V_s)$, because it is a Levi component of a parabolic subgroup of $SL(V_s)$ fixing a flag of the form $V_{n}\subset V_{n}\oplus U_1\subset ...\subset V_n\oplus U_{s-n+1}\subset V_s$ in $V_s$ where $U_1\subset...\subset U_{s-n-1}\subset V'$ is a maximal flag in $V'$. (The choice of the maximal flag $U_1\subset...\subset U_{s-n-1}\subset V'$ does not matter.) This completes the argument for embeddings between ordinary grassmannians.

If $X=GS(m,V_n)$ then by Theorem \ref{Theo Embed Mix Type} the embedding $\varphi$ factors through the tautological embedding $\tau_X$ as
$$
GS(m,V_n) \stackrel{\tau_X}{\hw} G(m,V_n) \stackrel{\sigma}{\hw} G(l,V_s)\;,
$$
where $\sigma$ is a standard extension. By the above, $\sigma(G(m,V_n))$ is a closed orbit $Z$ of an $\mc{LS}$-subgroup of $SL(V_s)$, and the embedding $\varphi(X)\hw Z$ is as claimed in (i.1). The case where $X= GO(m,V_n)$ with $m<\frac{n}{2}-2$ is analogous.

Next we assume that $Y= GO(l,V_s)$ with $l<\lfloor\frac{s-1}{2}\rfloor$. Then $G_Y=Spin(V_s)$ but the $G_Y$-action on both $Y$ and $\mc O_Y(1)$ actually reduces to $SO(V_s)$. If $X= G(m,V_n)$ then by Theorem \ref{Theo Embed Mix Type} the embedding $\varphi$ factors as a composition $\iota\circ\sigma$ of a standard extension $\sigma$ of ordinary grassmannians and an isotropic extension $\iota$. By Remark \ref{Rem compose LS} and the already established results on ordinary grassmannians, it suffices to consider the case of an isotropic extension, i.e., $\varphi=\iota$. Let $U_n\subset V_{s}$ be an isotropic subspace with a fixed isomorphism $V_n\cong U_n$ defining $\varphi$. Then $\varphi$ is $SL(V_n)$-equivariant for the homomorphism $h^O_{n,s}:SL(V_n)\to SO(V_s)$ whose image consists of the elements preserving a decomposition $V_s=V_{r}\oplus U_n\oplus U_n^*$ and acting by the identity on $V_{r}$ and on $\Lambda^n U_n$, where $U_n^*\subset V_s$ is an isotropic subspace dual to $U_n$ and $V_r$ is an orthogonal complement to $U_n\oplus U_n^*$. The image of $h^O_{n,s}$ is an $\mc{LS}$-subgroup acting transitively on the image of $\varphi$.

If $X= GO(m,V_n)$ then the embedding $\varphi$ is either a standard extension or a combination of standard and isotropic extensions. We consider the two cases. If $\varphi$ is a standard extension then there is an orthogonal decomposition $V_s\cong V_n\oplus V_{s-n}$ such that $\varphi$ is given by (\ref{For sigma strict standard}), and $\varphi$ is $SO(V_n)$-equivariant for the homomorphism 
\begin{gather}\label{For fnsO}
f^O_{n,s}:SO(V_n)\to SO(V_s)
\end{gather}
whose image consists of the elements preserving the decomposition $V_s\cong V_n\oplus V_{s-n}$ and acting trivially on $V_{s-n}$. The image of $f^O_{n,s}$ is an $\mc{LS}$-subgroup of $G_Y$ if and only if $s-n\equiv 0\,({\rm mod}\, 2)$.

If $s-n\not\equiv 1\,({\rm mod}\, 2)$ then we observe that $\varphi$ factors as a composition of two standard extensions of the form
$$
GO(m,V_n) \stackrel{\sigma_1}{\hw} GO(m,V_{n+1}) \stackrel{\sigma_2}{\hw} GO(l,V_s)  \;.
$$
Thus the image $Z:=\sigma_2(GO(m,V_{n+1}))$ is the orbit of an $\mc{LS}$-subgroup of $G_Y$, and the inclusion $\varphi(X)\subset Z$ is a standard extension as claimed in (i.2).

Now suppose that the embedding $\varphi$ is a combination of standard and isotropic extensions 
$$
X=GO(m,V_n) \stackrel{\tau_X}{\hw} G(m,V_{n}) \stackrel{\sigma}{\hw} G(l,V_r) \stackrel{\iota}{\hw} GO(l,V_s)=Y\;.
$$
By the previous cases, the image $Z:=\iota\circ\sigma(G(m,V_{n}))\subset Y$ is the orbit of an $\mc{LS}$-subgroup of $SO(V_s)$ isomorphic to $SL(V_n)$. The inclusion $\varphi(X)\subset Z$ is given by the tautological embedding of $X$ as claimed in (i.1).

To complete the argument for $Y= GO(l,V_s)$ it remains to consider the case $X= GS(m,V_r)$. Here Theorem \ref{Theo Embed Mix Type} implies that the embedding $\varphi$ is a mixed combination of standard and isotropic extensions, and the statement of (i.1) applies for the same reasons as for the above case of a combination of standard and isotropic extensions.

The case $Y= GS(l,V_{2s})$ is analogous to the case of $GO(l,V_s)$ and we leave it to the reader.

For part (ii) we use the following alternative characterization of $\mc L$-subgroups, see \cite[\S~30.2]{Humphreys}. Let $G$ be a connected semisimple linear algebraic group. A subgroup $L\subset G$ is an $\mc L$-subgroup if and only if $L$ is equal to the derived subgroup $(G_\gamma)'$ of the centralizer subgroup $G_\gamma\subset G$ of a one-parameter subgroup $\gamma:\CC^\times\to G$. If $L\subset G$ is an $\mc L$-subgroup then the normalizer $N_G(L)$ contains a maximal torus $T$ of $G$. Now let $Y=G/P$ be a flag variety of $G$ and $L\subset G$ be an ${\mc L}$-subgroup. An $L$-orbit $L\cdot x \subset Y$ is closed if and only if $x$ is a $T$-fixed point for some maximal torus $T\subset G$ contained in the normalizer $N_G(L)$. For any maximal torus $T\subset G$, the set of $T$-fixed points $Y^T$ forms an orbit of the Weyl group $W:=N_G(T)/T$.

We consider first the case $Y=GS(l,V_{2s})$ where $G_X=Sp(V_{2s})$. Every $\mc{LS}$-subgroup of $Sp(V_{2s})$ is isomorphic to $SL(n)$ or  $Sp(2n)$ for some $n\leq s$, and is obtained via the following construction. Fix a decomposition of $V_{2s}$ as a sum of two maximal isotropic subspaces, $V_{2s} = U_s\oplus U_s^*$. Let $V_{2s}=V_{2n}\oplus V'$ be an orthogonal decomposition such that $U_n:=V_{2n}\cap U_s$ and $U_n^*:=V_{2n}\cap U_s^*$ are complementary maximal isotropic spaces of $V_{2n}$. We denote by $L_{n,A}$ the subgroup of $Sp(V_{2s})$ which preserves the decomposition $V_{2s}=U_n\oplus U_n^*\oplus V'$ and acts by the identity on both $V'$ and $\Lambda^nU_n$, and by $L_{n,C}$ the subgroup of $G$ preserving the decomposition $V_{2s}=V_{2n}\oplus V'$ and acting by the identity on $V'$. Then there are isomorphisms
$$
h^S_{n,2s}:SL(U_n)\cong L_{n,{\rm A}} \subset Sp(V_{2s}) \quad{\rm and}\quad f^S_{2n,2s}:Sp(V_{2n})\cong L_{n,{\rm C}}\subset Sp(V_{2s}) \;.
$$

Next we observe that the closed orbits in $Y$ of an $\mc{LS}$-subgroup $L_{n,{\rm A}}\subset Sp(V_{2s})$ are exactly the $L_{n,{\rm A}}$-orbits of points $U\in GS(l,V_{2s})$ such that $U=(U\cap U_n)\oplus (U\cap U_n^*)\oplus (U\cap V')$. Let us fix one such $U$ and denote the closed orbit $L_{n,{\rm A}}\cdot U$ by $X$. Set $m:=\dim U\cap U_n$ and $m^*:=\dim U\cap U_n^*$. Then $X$ is an $L_{n,{\rm A}}$-fixed point if and only if at least one of $m$ and $m^*$ is equal to $0$. Suppose that $X$ is not a point and $m>0$, the case $m^*>0$ being analogous. Then $X$ is isomorphic to the grassmannian $G(m,U_n)$, since $X$ is the image of the $SL(U_n)$-equivariant embedding
$$
G(m,U_n) \stackrel{\varphi}{\hookrightarrow} GS(l,V_{2s}) \;,\; \varphi(g\cdot (U\cap U_n)) := h^S_{n,2s}(g)\cdot U \quad for\quad g\in SL(U_n)\;.
$$
It is straightforward to check that the embedding $\varphi$ is linear if and only if $l\ne s$. Furthermore, if $\varphi$ is linear then it is a composition of a standard extension $G(m,U_n)\hw G(l,U_s)$ and an isotropic extension $G(l,U_s)\hw GS(l,V_{2s})$, as claimed in (ii.2).

For an $\mc{LS}$-subgroup $L_{n,{\rm C}}\subset Sp(V_{2n})$, the closed $L_{n,{\rm C}}$-orbits in $Y$ are exactly the $L_{n,{\rm C}}$-orbits of points $U\in GS(l,V_{2s})$ such that $U=(U\cap V_{2n})\oplus (U\cap V')$. Let $X:=L_{n,{\rm C}}\cdot U$ be a closed orbit and let $m:=\dim(U\cap V_{2n})$. Then $X$ is a fixed point if and only if $m=0$. Otherwise, $X$ is isomorphic to the grassmannian $GS(m,V_{2n})$ and is linearly embedded in $Y$ by the $Sp(V_{2n})$-equivariant embedding
$$
GS(m,V_{2n}) \stackrel{\varphi}{\hookrightarrow} GS(l,V_{2s}) \;,\; \varphi(g\cdot (U\cap V_{2n})) := f^S_{2n,2s}(g)\cdot U \quad for\quad g\in Sp(V_{2n})
$$
which is a standard extension. This completes the argument for a symplectic grassmannian $Y$.

The other cases are similar, and here we present the details only for a spinor grassmannian. Suppose that $Y=GO(l,V_{2l})$ with $l\geq 5$. Here $G_Y=Spin(V_{2l})$ and the $G_Y$-action on $\mc O_Y(1)$ does not reduce to an $SO(V_{2l})$-action. Every proper $\mc{LS}$-subgroup $L\subset G_Y$ is isomorphic either to $SL(n)$ for $2\leq n\leq l$, or to $Spin(2n)$ for $3\leq n< l$. The inclusion of $L$ in $G$ is respectively
$$
h^O_{n,2s}:SL(U_n)\cong L_{n,{\rm A}} \subset G_Y \;\;{\rm or}\;\; f^O_{2n,2l}:Spin(V_{2n})\cong L_{n,{\rm D}}\subset G_Y ,
$$
where the subgroup $L_{n,{\rm D}}$ consists of the elements of $Spin(V_{2l})$ preserving a fixed orthogonal decomposition $V_{2l}=V_{2n}\oplus V_{2l-2n}$ and acting by the identity on $V_{2l-2n}$, and the subgroup $L_{n,{\rm A}}\subset Spin(V_{2l})$ consists of the elements preserving a fixed decomposition $V_{2l}=U_{n}\oplus U_n^*\oplus V_{2l-2n}$ with $U_{n},U_n^*\subset V_{2l}$ isotropic and $V_{2l-2n}=U_n^\perp\oplus (U_n^*)^\perp$, and acting by the identity on both $V_{2l-2n}$ and $\Lambda^n U_n$.

The closed $L_{n,{\rm A}}$-orbits in $Y=GO(l,V_{2l})$ are exactly the $L_{n,{\rm A}}$-orbits of points $U\in Y$ such that $U=(U\cap U_n)\oplus (U\cap U_n^*)\oplus (U\cap V_{2l-2n})$. Let us fix one such $U$ and denote by $X:=L_{n,{\rm A}}\cdot U\subset Y$ the closed $L_{n,{\rm A}}$-orbit. Set $m:=\dim U\cap U_n$ and $m^*:=\dim U\cap U_n^*$. Then $X$ is an $L_{n,{\rm A}}$-fixed point if and only if $(m,m^*)\in\{(0,0),(l,0),(0,l)\}$. Suppose that $X$ is not a point and $m>0$ (the case $m^*>0$ is analogous). Then $X$ is isomorphic to the grassmannian $G(m,U_n)\cong SL(U_n)/P_m$, since $X$ is the image of the $SL(U_n)$-equivariant embedding
$$
G(m,U_n) \stackrel{\varphi}{\hookrightarrow} GO(l,V_{2l}) \;,\; \varphi(g\cdot (U\cap U_n)) := f^O_{n,2l}(g)\cdot U \quad {\rm for}\quad g\in SL(U_n) \;.
$$
The map $\varphi$ is a linear embedding because it factors as a composition
$$
X=G(m,V_n)\stackrel{\sigma}{\hw} G(k,U_l) \stackrel{\theta_l^{l-k}}{\hw} GO(l,V_{2l})=Y
$$
for some maximal isotropic subspace $U_l\subset V_{2l}$ containing $U_n$, where $\sigma$ is a standard extension and $\theta_l^{l-k}$ is one of the embeddings defined in (\ref{For Special Embed theta mr even}) and (\ref{For Special Embed theta mr odd}). Thus statement (ii.2) applies for the closed $L_{n,{\rm A}}$-orbits in $Y$.

The closed $L_{n,{\rm D}}$-orbits in $Y=GO(l,V_{2l})$ are exactly the $L_{n,{\rm D}}$-orbits through points $U\in Y$ such that $U=(U\cap V_{2n})\oplus (U\cap V_{2l-2n})$. For such $U$ the intersection $U\cap V_{2n}$ is a maximal isotropic subspace of $V_{2n}$, and the same holds for $U\cap V_{2l-2n}$ in $V_{2l-2n}$. Therefore every closed $L_{n,{\rm D}}$-orbit $X:=L_{n,{\rm D}}\cdot U\subset Y$ is isomorphic to the spinor grassmannian $GO(n,V_{2n})$, and its $Spin(V_{2n})$-equivariant embedding
$$
GO(n,V_{2n}) \stackrel{\varphi}{\hookrightarrow} GO(l,V_{2l}) \;,\; \varphi(g\cdot (U\cap V_{2n})) := f^O_{n,2l}(g)\cdot U \quad {\rm for}\quad g\in Spin(V_{r})
$$
is a standard extension. Thus statement (ii.1) holds for the closed $L_{n,{\rm D}}$-orbits in $Y$ and hence (ii) holds for spinor grassmannians. Proposition \ref{Prop Equivar through LS} is proved.
\end{proof}

\begin{proof}[Proof of Theorem \ref{Theo Equivar gen}] The statement is deduced in a straightforward manner from part (i) of Proposition \ref{Prop Equivar through LS}, by the obvious remark that the tautological embeddings $GO(m,V_n)\hw G(m,V_n)$ and $GS(m,V_n)\hw G(m,V_n)$ are equivariant for the tautological inclusions of groups
$$
SO(V_n)\hw SL(V_n) \quad{\rm and}\quad Sp(V_n)\hw SL(V_n)\;,
$$
respectively, and any standard extension $GO(m,V_n)\hw GO(m,V_{n+1})$ is equivariant for a homomorphism $Spin_n\to Spin_{n+1}$.
\end{proof}

Next we consider linear embeddings which factor through a projective space or a standard quadric.

\begin{prop}\label{Prop Equivariance and exceptions}
Let $X\stackrel{\varphi}{\hookrightarrow} Y$ be a linear embedding of grassmannians, where $Y$ is not a spinor grassmannian and $X$ is not isomorphic to $GO(m-2,V_{2m})$ for $m\geq 5$ or to $GO(m-1,V_{2m+1})$ for $m\geq 3$. Assume furthermore that $\varphi$ factors through a projective space or a quadric. Then $\varphi$ is equivariant if and only if it is not one of the following embeddings:
\begin{enumerate}
\item[{\rm (a)}] $\varphi:X\stackrel{\kappa_p}{\hookrightarrow} Q \stackrel{\sigma}{\hookrightarrow} GO(l,V_s)=Y$ for a standard extension $\sigma$ and some $p\in I_2(X)$ which is not invariant under any subgroup $G\subset G_X$ acting transitively on $X$.
\item[{\rm (b)}] $\varphi:X\stackrel{\pi_X}{\hookrightarrow} \PP(V_X)\stackrel{\lambda}{\hookrightarrow} GS(1,V_{s-2l}) \stackrel{\sigma}{\hookrightarrow} GS(l,V_s)=Y$, where $2\leq l<s/2$, $\sigma$ is a standard extension, and $\lambda$ is a linear embedding such that the restriction of the symplectic form of $V_{s-2l}$ to $V_X$ is not $G$-invariant for any subgroup $G\subset G_X$ acting transitively on $X$.
\end{enumerate}
\end{prop}

\begin{proof}
Recall that any minimal projective embedding of a grassmannian $X$ is $G_X$-equivariant. Also, every linear embedding $\psi:\PP(V)\hw Y$ of a projective space into a grassmannian $Y$ is equivariant, and it is $SL(V)$-equivariant if its image is not a standard symplectic projective space on a symplectic grassmannian. Hence, if $\varphi$ factors through a projective space which is not a standard symplectic projective space then $\varphi$ is $G_X$-equivariant. If $\varphi$ factors through a standard symplectic projective space $GS(1,V_{s-2l})$ on $Y$ then the condition for equivariance of $\varphi$ is the invariance of respective symplectic form on $V_{s-2l}$, as stated in (b).

Assume that $\varphi$ factors through a quadric $Q\subset Y$ but not through a projective space. Then, by Corollary \ref{Coro Embed Quadrics} $Y$ is an orthogonal grassmannian and $Q\cong GO(1,V_r)$ is a standard quadric on $Y$. By Lemma \ref{Lemma embed in quadric} the embedding of $X$ in $Q$ is of the form $\sigma\circ\kappa_p$ for some $p\in I_2(X)$. The condition for equivariance of $\varphi$ is the invariance of $p$, as stated in (a).
\end{proof}

Let us comment on the occurrence of non-equivariant embeddings. First, recall that in most cases the group $G_X$ does not admit any proper subgroups acting transitively on $X$, and the conditions given in (a) and (b) concern only the group $G_X$. The exceptions are only the spinor  grassmannians, and the odd-dimensional projective spaces. In these cases there are indeed linear embeddings which are equivariant under a specific subgroup $G\subset G_X$, but not under $G_X$. For instance, the minimal projective embedding of $X=GO(4k+3,V_{8k+6})$ factors through non-degenerate quadrics. None of these quadrics is invariant under $G_X=Spin(V_{8k+6})$, but every subgroup $G\subset G_X$ isomorphic to $Spin_{8k+5}$ admits such an invariant quadric.

Note that, if $\varphi:X\hookrightarrow Y$ is a non-equivariant embedding of type (a) or (b) from Proposition \ref{Prop Equivariance and exceptions}, then the composition of $\varphi$ with any embedding $\psi:Y \hookrightarrow Z$ remains non-equivariant if and only if $\psi$ is a standard extension. All other compositions of $\varphi$ are equivariant.

On the other hand, if $\varphi:X\hookrightarrow Y$ and $\psi:Y \hookrightarrow Z$ are equivariant embeddings then the composition $\psi\circ \varphi$ is equivariant, unless $Y$ is a projective space, $Z\cong GS(l,V_s)$ is a symplectic grassmannian, $\psi(Y)$ is a standard symplectic projective space on $Y$, and $\varphi\circ \psi$ is an embedding of type (b) from Proposition \ref{Prop Equivariance and exceptions}.

In the next proposition we address the equivariance properties of the special embeddings of grassmannians excluded by the hypothesis of Theorem \ref{Theo Equivar gen}.

\begin{prop}\label{Prop special emb delta 4k2 S 4k O}
The embeddings $\delta_n:GO(n-2,V_{2n})\hw G(2,V_{2^{n-1}})$ and $\delta^1_n:GO(n-2,V_{2n-1})\hw G(2,V_{2^{n-1}})$ defined in (\ref{For special emb delta n}) and (\ref{For special emb odddelta}) are respectively $Spin(V_{2n})$- and $Spin(V_{2n-1})$-equivariant. Furthermore, the following hold:
\begin{enumerate}
\item[{\rm (i)}] $n$ is odd if and only if $\delta_n$ does not admit factorizations as a composition of two proper equivariant embeddings of grassmannians;
\item[{\rm (ii)}] $n\equiv 2\,({\rm mod}\, 4)$ if and only if $\delta_n$ factors $Spin(V_{2n})$-equivariantly as
$$
\delta_n:GO(n-2,V_{2n})\stackrel{\delta_n^S}{\hw} GS(2,V_{2^{n-1}})\stackrel{\tau}{\hookrightarrow} G(2,V_{2^{n-1}})\;;
$$
\item[{\rm (iii)}] $n\equiv 0\,({\rm mod}\, 4)$ if and only if $\delta_n$ factors $Spin(V_{2n+1})$-equivariantly as
$$
\delta_n:GO(n-2,V_{2n})\stackrel{\delta_n^O}{\hw} GO(2,V_{2^{n-1}})\stackrel{\tau}{\hookrightarrow} G(2,V_{2^{n-1}})\;;
$$
\item[{\rm (iv)}] $n\equiv 1,2\,({\rm mod}\, 4)$ if and only if $\delta^1_n$ factors $Spin(V_{2n-1})$-equivariantly as
$$
\delta^1_n:GO(n-2,V_{2n-1})\stackrel{\delta_n^{1,S}}{\hw} GS(2,V_{2^{n-1}}) \stackrel{\tau}{\hookrightarrow}G(2,V_{2^{n-1}})\;;
$$
\item[{\rm (v)}] $n\equiv 0,3\,({\rm mod}\, 4)$ if and only if $\delta^1_n$ factors $Spin(V_{2n-1})$-equivariantly as 
$$
\delta^1_n:GO(n-2,V_{2n-1})\stackrel{\delta_n^{1,O}}{\hw} GO(2,V_{2^{n-1}}) \stackrel{\tau}{\hookrightarrow}G(2,V_{2^{n-1}})\;.
$$
\end{enumerate}
\end{prop}

\begin{proof}
The existence of an equivariant factorization of $\delta_n$ through $GO(2,V_{2^{n-1}})$ or $GS(2,V_{2^{n-1}})$ is equivalent to the existence of an invariant symmetric or, respectively, symplectic non-degenerate bilinear form on the spin-representation of $Spin_{2n}$. The conditions for existence of such invariant bilinear forms and the respective self-duality of spin-representations are well-known, see e.g. \cite{Dynkin-Maximal}, \cite{Dynkin-1952}.
\end{proof}

\begin{coro}\label{Coro all equivar in Gls}
Let $X$ be a grassmannian.
\begin{enumerate}
\item[{\rm (i)}] Every linear embedding $\varphi:X\hw G(l,V_s)$ is $G_X$-equivariant.
\item[{\rm (ii)}] Every linear embedding $\varphi:X\hw GO(l,V_s)$ with $l<\lfloor\frac{s-1}{2}\rfloor$ is $G_X$-equivariant, unless it factors through a non-equivariant embedding into a standard quadric, or through a non-equivariant embedding of the form $\delta_n^O$ for $X=GO(n-2,V_{2n})$, or through a non-equivariant embedding of the form $\delta_n^{1,O}$ for $X=GO(n-2,V_{2n-1})$.
\item[{\rm (iii)}] Every linear embedding $\varphi:X\hw GS(l,V_s)$ is $G_X$-equivariant, unless it factors through a non-equivariant embedding into a standard symplectic projective space, or through a non-equivariant embedding of the form $\delta_n^S$ for $X=GO(n-2,V_{2n})$, or through a non-equivariant embedding of the form $\delta_n^{1,S}$ for $X=GO(n-2,V_{2n-1})$.
\end{enumerate}
\end{coro}

\begin{proof}
For $X$ not isomorphic to $GO(n-2,V_{2n})$ or $GO(n-1,V_{2n+1})$, the statement follows immediately from Theorem \ref{Theo Equivar gen} and Proposition \ref{Prop Equivariance and exceptions}. In the two remaining cases the result follows from Theorems \ref{Theo Embed GOcodim2} and \ref{Theo Embed GOcodim1odd}, together with Proposition \ref{Prop special emb delta 4k2 S 4k O}.
\end{proof}

\section{Linear embeddings of linear ind-grassmannians}\label{Sec Ind Embed}

Linear ind-grassmannians are defined and classified in \cite{Penkov-Tikhomirov-Lin-ind-Grass}. In this final section we show that the classification of linear embeddings of grassmannians $X\stackrel{\varphi}{\hw} Y$, where $Y$ is not a spinor grassmannian, carries over to linear ind-grassmannians.

\subsection{Linear ind-grassmannians}

We begin by recalling some basic notions and the classification of ind-grassmannians.

\begin{definition}\label{Def standard Lin indGrass}
A {\bf linear ind-grassmannian} is an ind-variety $X$ which can be obtained as a direct limit $X=\lim\limits_{\to} X_k$ of a sequence of linear embeddings $\psi_k:X_k\to X_{k+1}$ of grassmannians. A sequence $(X_k,\psi_k)$ with the additional property that all $\psi_k$ are standard extensions is a {\bf standard exhaustion} of ${\bf X}$.
\end{definition}

The Picard group of a linear ind-grassmannian is isomorphic to $\ZZ$, generated by the class of the line bundle $\mc O_{\bf X}(1):=\lim\limits_{\leftarrow} \mc O_{X_k}(1)$. Note also that every exhaustion of a linear ind-grassmanian by grassmannians is a chain of linear embeddings up to finitely many terms. Indeed, if there were infinitely many non-linear embeddings in an exhaustion, the Picard group of the ind-grassmannian would be trivial and hence not isomorphic to $\ZZ$.

Let us recall the definitions of the following linear ind-grassmannians from \cite{Penkov-Tikhomirov-Lin-ind-Grass}. Fix a chain of proper inclusions of vector spaces $\nu_k:V_{n_k}\subset V_{n_{k+1}}$, $k\in\ZZ_{\geq 1}$, and denote by $\{m_k\}$ a non-decreasing sequence of positive integers such that $\{n_k - m_k\}$ is also a non-decreasing sequence of positive integers.
\begin{enumerate}
\item[{\rm (A)}] ${\bf G}(m):=\lim\limits_{\to} G(m,V_{n_k})$ is defined as the direct limit of the chain
$$
\dots \hookrightarrow G(m,V_{n_k}) \stackrel{\sigma_k}{\hookrightarrow} G(m,V_{n_{k+1}}) \hookrightarrow \dots
$$
of strict standard extensions associated the inclusions $\nu_k$, where we assume $m:= m_k < n_1$ for all $k$.
\item[{\rm (B)}] ${\bf G}(\infty):=\lim\limits_{\to} G(m_k,V_{n_k})$, the {\bf Sato grassmannian} \cite{Sato}, is defined as the direct limit of an arbitrary chain of standard extensions
$$
\dots \hookrightarrow G(m_k,V_{n_k}) \stackrel{\sigma_k}{\hookrightarrow} G(m_{k+1},V_{n_{k+1}}) \hookrightarrow \dots
$$
associated to the inclusions $\nu_k$, under the assumption that $\lim\limits_{k\to \infty} m_k=\infty=\lim\limits_{k\to\infty} (n_k-m_k)$.

\item[{\rm (C)}] ${\bf GO}(m,\infty):=\lim\limits_{\to} GO(m,V_{n_k})$ is defined as the direct limit of the chain
$$
\dots \hookrightarrow GO(m,V_{n_k}) \stackrel{\sigma_k}{\hookrightarrow} GO(m,V_{n_{k+1}}) \hookrightarrow \dots
$$
of standard extensions associated to the inclusions $\nu_k$, which here are assumed to be non-degenerate inclusions of orthogonal vector spaces, and where $m:= m_k< \lfloor \frac{n_1}{2}\rfloor$ for all $k$.

\item[{\rm (D)}] ${\bf GO}(\infty,\infty):=\lim\limits_{\to} GO(m_k,V_{n_k})$ is defined as the direct limit of an arbitrary chain of standard extensions
$$
\dots \hookrightarrow GO(m_k,V_{n_k}) \stackrel{\sigma_k}{\hookrightarrow} GO(m_{k+1},V_{n_{k+1}}) \hookrightarrow \dots
$$
compatible with the inclusions $\nu_k$, which are assumed to be non-degenerate inclusions of orthogonal vector spaces, and where $\lfloor \frac{n_k}{2}\rfloor-m_k$ is a non-decreasing sequence of positive integers, $\lim\limits_{k\to\infty} m_k=\infty$ and $\lim\limits_{k\to\infty} (\lfloor \frac{n_k}{2}\rfloor-m_k)=\infty$.

\item[{\rm (E)}] ${\bf GO}^0(\infty,m):=\lim\limits_{\to} GO(m_k,V_{n_k})$, for $m\geq 2$, is defined as the direct limit of an arbitrary chain of standard extensions
$$
\dots \hookrightarrow GO(m_k,V_{n_k}) \stackrel{\sigma_k}{\hookrightarrow} GO(m_{k+1},V_{n_{k+1}}) \hookrightarrow \dots
$$
compatible with the inclusions $\nu_k$, which are assumed to be non-degenerate inclusions of orthogonal vector spaces, and where $n_k$ is even for all $k$, $\frac{n_k}{2}-m_k$ is a non-decreasing sequence of positive integers, $\lim\limits_{k\to\infty} m_k=\infty$ and $\lim\limits_{k\to\infty} (\frac{n_k}{2}-m_k)=m$.

\item[{\rm (F)}] ${\bf GO}^1(\infty,m):=\lim\limits_{\to} GO(m_k,V_{n_k})$ for $m\geq 0$ is defined as the direct limit of an arbitrary chain of standard extensions
$$
\dots \hookrightarrow GO(m_k,V_{n_k}) \stackrel{\sigma_k}{\hookrightarrow} GO(m_{k+1},V_{n_{k+1}}) \hookrightarrow \dots
$$
compatible with the inclusions $\nu_k$, which are assumed to be non-degenerate inclusions of orthogonal vector spaces, and where $n_k$ is odd for all $k$, $\lfloor \frac{n_k}{2}\rfloor-m_k$ is a non-decreasing sequence of positive integers, $\lim\limits_{k\to\infty} m_k=\infty$ and $\lim\limits_{k\to\infty} (\lfloor \frac{n_k}{2}\rfloor-m_k)=m$.

\item[{\rm (G)}] ${\bf GS}(m,\infty):=\lim\limits_{\to} GS(m,V_{n_k})$ is defined as the direct limit of the chain
$$
\dots \hookrightarrow GS(m,V_{n_k}) \stackrel{\sigma_k}{\hookrightarrow} GS(m,V_{n_{k+1}}) \hookrightarrow \dots
$$
of standard extensions associated to the inclusions $\nu_k$, which are now assumed to be non-degenerate inclusions of symplectic vector spaces, and where $m:= m_k< \lfloor \frac{n_1}{2}\rfloor$ for all $k$.

\item[{\rm (H)}] ${\bf GS}(\infty,\infty):=\lim\limits_{\to} GS(m_k,V_{n_k})$ is defined as an analogue of case (D) for symplectic grassmannians.

\item[{\rm (I)}] ${\bf GS}(\infty,m):=\lim\limits_{\to} GS(m_k,V_{n_k})$ for $m\geq 0$ is defined as the direct limit of an arbitrary chain of standard extensions
$$
\dots \hookrightarrow GS(m_k,V_{n_k}) \stackrel{\sigma_k}{\hookrightarrow} GS(m_{k+1},V_{n_{k+1}}) \hookrightarrow \dots
$$
compatible with the inclusions $\nu_k$, which are assumed to be non-degenerate inclusions of symplectic vector spaces, and where under the assumption that $n_k$ is even for all $k$, $\frac{n_k}{2}-m_k$ is a non-decreasing sequence of positive integers, $\lim\limits_{k\to\infty} m_k=\infty$ and $\lim\limits_{k\to\infty} (\frac{n_k}{2}-m_k)=m$.
\end{enumerate}

\begin{prop}{\rm (\cite[Lemma 4.3]{Penkov-Tikhomirov-Lin-ind-Grass})}
Each of the ind-grassmannians defined above depends up to isomorphism only on the respective limits $\lim\limits_{k\to\infty} m_k$, $\lim\limits_{k\to\infty} (n_k-m_k)$ and $\lim\limits_{k\to\infty} (\lfloor\frac{n_k}{2}\rfloor-m_k)$.
\end{prop}

Note that all ind-grassmannians defined above are defined only up to isomorphism. Furthermore, the ind-grassmannians ${\bf G}(m)$ admit an alternative construction as ${\bf G}(m)\cong \lim\limits_\to G(m_k,V_{n_k})$ where $n_k-m_k=m$ for all $k$.

The ind-grassmannian ${\bf G}(1)$ is the {\bf projective ind-space} $\PP^\infty$, and is isomorphic to ${\bf GS}(1,\infty)$. The ind-grassmannian ${\bf GO}(1,\infty)$ is the {\bf ind-quadric} $Q^\infty$.

\begin{theorem}\label{Theo Standard Ind Grass} {\rm (\cite[Theorem 2]{Penkov-Tikhomirov-Lin-ind-Grass})} 
Every linear ind-grassmannian is isomorphic as an ind-variety to one of the ind-grassmannians ${\bf G}(m)$ for $m\geq 1$, ${\bf G}(\infty)$, ${\bf GO}(m,\infty)$ for $m\geq 1$, ${\bf GO}^0(\infty,m)$ for $m\geq 2$, ${\bf GO}^1(\infty,m)$ for $m\geq 0$, ${\bf GO}(\infty,\infty)$, ${\bf GS}(m,\infty)$ for $m\geq 2$, ${\bf GS}(\infty,m)$ for $m\geq 0$, ${\bf GS}(\infty,\infty)$, and the latter are pairwise non-isomorphic.
\end{theorem}

We should point out that, despite Theorem \ref{Theo Standard Ind Grass}, some linear ind-grassmannians admit exhaustions which are not standard. For instance, the Sato grassmannian ${\bf G}(\infty)$ can be obtained as a direct limit of a chain
\begin{gather}\label{For Sato}
\dots \hookrightarrow GS(m_{2l},V_{n_{2l}}) \stackrel{\varphi_{2l}}{\hookrightarrow} GO(m_{2l+1},V_{n_{2l+1}}) \stackrel{\varphi_{2l+1}}{\hookrightarrow} GS(m_{2l+2},V_{n_{2l+2}}) \hookrightarrow \dots
\end{gather}
where each $\varphi_k$ is a mixed combination of standard and isotropic extensions. Indeed, a chain of embeddings of the form (\ref{For Sato}) is easy to define under the assumption that $\lim\limits_{s\to\infty} m_s = \infty =\lim\limits_{s\to\infty} (\frac{n_s}{2}-m_s)$. Then the direct limit of such a chain will be isomorphic to the direct limit of the subchain consisting of ordinary grassmannians only, and the latter direct limit is necessarily isomorphic to ${\bf G}(\infty)$ since the embeddings of this chain do not factor through projective spaces.

We now give an alternative definition of the linear ind-grassmannians. One can think of the definitions (A)-(I) as local, while the definition below is global and yields a well-defined ind-variety rather than an isomorphism class of ind-varieties.

Let $V$ be a fixed vector space of infinite countable dimension. Let $W\subset V$ be a subspace and $E\subset V$ be a basis of $V$ such that $E\cap W$ is a basis of $W$. Let ${\bf G}(W,E,V)$ denote the set of subspaces $U\subset V$ which are $E$-{\bf commensurable} with $W$, i.e., satisfy:
\begin{enumerate}
\item[$\bullet$] there exists a finite-dimensional subspace $T_U\subset V$ such that $U\subset W+T_U$, $W\subset U+T_U$ and $\dim(U\cap T_U)=\dim(W\cap T_U)$;
\item[$\bullet$] ${\rm codim}_U{\rm span}(E\cap U) <\infty$.
\end{enumerate}

Let now $V$ be endowed with a symmetric or skew-symmetric non-degenerate bilinear form $\omega$. A basis $E$ of $V$ is {\bf isotropic} if it admits an involution $i_E:E\to E$ with at most one fixed point, and such that $\omega(e,e')=0$ for $e,e'\in E$ unless $e'=i_E(e)$. For a fixed isotropic subspace $W\subset V$ and an isotropic basis $E\subset V$ containing a basis of $W$, we denote by ${\bf GO}(W,E,V)$, or respectively ${\bf GS}(W,E,V)$, the set of isotropic subspaces of $V$ which are $E$-commensurable with $W$.

\begin{theorem}\label{Theo Dimitrov Penkov} {\rm (\cite{Dimitrov-Penkov})}
The sets ${\bf G}(W,E,V)$, ${\bf GO}(W,E,V)$, ${\bf GS}(W,E,V)$ admit structures of projective ind-varieties isomorphic to linear ind-grassmannians as follows:
\begin{align*}
& {\bf G}(W,E,V) \cong {\bf G}(\min\{\dim W,{\rm codim}_VW\})\;, \\
& {\bf GO}(W,E,V) \cong
\begin{cases}
{\bf GO}(\dim W,\infty) & \quad for \quad {\rm codim}_U W=\infty \;,\\
{\bf GO}^\eps(\infty,{\rm codim}_U W) & \quad for \quad 0<{\rm codim}_U W<\infty \;, \\
{\bf GO}^1(\infty,{\rm codim}_U W) & \quad for \quad {\rm codim}_U W=0 \;,
\end{cases} \\ 
& {\bf GS}(W,E,V) \cong {\bf GS}(\dim W,{\rm codim}_UW)\;,
\end{align*}
where $U\subset V$ is a maximal isotropic subspace containing $W$ and spanned by $E\cap U$, and $\eps\in\{0,1\}$ is the number of fixed points of the involution $\iota_E$.
\end{theorem}

Let ${\bf X}$ be a linear ind-grassmannian with a fixed standard exhaustion ${\bf X}=\lim\limits_\to X_k$. Then a triple $(W,E,V)$ yielding an ind-grassmannian ${\bf G}(W,E,V)$ isomorphic to ${\bf X}$ is obtained by setting $V:=\lim\limits_\to V_{n_k}$ and $W:=\lim\limits_\to U_k$, where $U_1\in X_1$ is fixed arbitrarily, $U_{k+1}:=\psi_k(U_k)\in X_{k+1}$ for $k\geq 1$, and $E\subset V$ is a basis such that $E\cap V_{n_k}$ is a basis of $V_{n_k}$ and $E\cap U_k$ is a basis of $U_k$. Moreover, as a set and ind-variety, ${\bf G}(W,E,V)$ depends not on the entire basis $E$ but only on the intersection $W\cap E$ up to changing finitely many basis vectors in $W\cap E$. In particular, if $\dim W=m<\infty$ then ${\bf G}(W,E,V)$ does not depend on $E$ and depends only on $m$ and not on $W$. In this case we may write ${\bf G}(m,V)$, and $\PP(V)$ for $m=1$.

Let us remark that {\bf the spinor ind-grassmannian} ${\bf GO}^1(\infty,0)$ admits two alternative realizations as ${\bf GO}(W,E,W\oplus W_*\oplus V_1)\cong \lim\limits_\to GO(k,V_{2k+1})$ and ${\bf GO}(W,E,W\oplus W_*)\cong \lim\limits_\to GO(k,V_{2k})$, where $W_*=\iota_E(E\cap W)$ and $V_1=W^\perp\cap (W_*)^\perp$. In some constructions it will be convenient to use the notation ${\bf GO}^0(\infty,0)$ for the latter realization, noting that ${\bf GO}^1(\infty,0)\cong{\bf GO}^0(\infty,0)$.

\subsection{Linear embeddings}

\begin{definition}\label{Def Lin Embed ind Grass}
Let ${\bf X}$ and ${\bf Y}$ be linear ind-grassmannians. A morphism of ind-varieties ${\bf X}\stackrel{\varphi}{\to}{\bf Y}$ is a {\bf linear embedding} if there exist standard exhaustions ${\bf X}=\lim\limits_{\to} X_k$ and ${\bf Y}=\lim\limits_{\to} Y_k$ such that $\varphi$ is equal to the direct limit $\lim\limits_{\to} \varphi_k$ of a sequence of linear embeddings $X_k\stackrel{\varphi_k}{\hookrightarrow} Y_k$.
\end{definition}

If ${\bf X}\stackrel{\varphi}{\hw}{\bf Y}$ is a linear embedding as above then $\mc O_{\bf X}(1)\cong \varphi^*\mc O_{\bf Y}(1)$ holds.

Note that it is possible to obtain an isomorphism of linear ind-grassmannians as a direct limit of proper linear embeddings $\varphi_k$ of grassmannians. Indeed, consider some linear ind-grassmannian ${\bf X}=\lim\limits_\to X_k$ with defining chain of embeddings $\psi_k:X_k\hw X_{k+1}$. By setting $Y_k:=X_{k+1}$ and $\varphi_k:=\psi_k$, we obtain a sequence of linear embeddings $X_k\stackrel{\varphi_k}{\hw} Y_k$ which induces the identity map on ${\bf X}=\lim\limits_{\to} X_k = \lim\limits_\to Y_k$. On the other hand, there are proper linear embeddings between isomorphic linear ind-grassmannians. In particular, any proper inclusion of countable-dimensional vector spaces $V'\subset V$ indices a linear embedding $\PP(V')\hookrightarrow \PP(V)$.

Next we define several types of embeddings of linear ind-grassmannians, generalizing the constructions from the finite-dimensional case. We use the realizations provided by Theorem \ref{Theo Dimitrov Penkov}.

Note first that every linear ind-grassmannian ${\bf X}$ admits a linear embedding into a projective ind-space. Indeed, there are natural inclusions $V_{X_k}\subset V_{X_{k+1}}$ for a given standard exhaustion ${\bf X}=\lim\limits_{\to} X_k$, and we can set $V_{\bf X}:=\lim\limits_\to V_{X_k}$. Then a linear projective embedding
$$
\pi_{\bf X}:{\bf X}\hw \PP(V_{\bf X})
$$
is obtained as the direct limit $\pi_{\bf X}:=\lim\limits_\to \pi_{X_k}$ of the minimal projective embeddings $\pi_{X_k}:X_k\hw\PP(V_{X_k})$.

An embedding ${\bf G}(m)\stackrel{\varphi}{\hookrightarrow}{\bf G}(l)$, where $m,l\in\NN\cup\{\infty\}$, is a {\bf standard extension} if $\varphi$ can be expressed as
\begin{gather}\label{For ind-standard ext oridnary}
\begin{array}{rcl}
\varphi:{\bf G}(m)\cong {\bf G}(W,E,V) & \hookrightarrow & {\bf G}(\tilde W,\tilde E,\tilde V) \cong {\bf G}(l) \\ 
U & \mapsto  & U\oplus W'
\end{array}
\end{gather}
for some isomorphisms ${\bf G}(m)\cong {\bf G}(W,E,V)$ and ${\bf G}(l)\cong {\bf G}(\tilde W,\tilde E,\tilde V)$, together with an inclusion of infinite-dimensional vector spaces $V\subset \tilde V$ such that $E=\tilde E\cap V$, a decomposition $\tilde V=V\oplus V'$, and a subspace $W'\subset V'$ with basis $\tilde E\cap W'$ such that $\tilde W=W\oplus W'$. Note that $m\leq l$ necessarily holds in the above situation.

A standard extension ${\bf GS}(m,c)\stackrel{\varphi}{\hw}{\bf GS}(l,d)$ is defined by the same formula $U \mapsto U\oplus W'$ as above, under the additional requirements that the inclusion $V\subset \tilde V$ is non-degenerate, the subspaces $W$ and $\tilde W$ are isotropic, the bases $E$, $\tilde E$ are isotropic. The same applies for standard extensions ${\bf GO}(m,\infty)\stackrel{\varphi}{\hw}{\bf GO}(l,\infty)$, as well as for ${\bf GO}^\eps(\infty,m)\stackrel{\varphi}{\hw}{\bf GO}^{\eps'}(\infty,l)$ with $\eps,\eps'\in\{0,1\}$. In the latter case we require $l=0$ whenever $m=0$.

It is straightforward to verify that every standard extension of linear ind-grassmannians is a linear embedding. Furthermore, a standard extension between non-spinor ind-grassmannians is a proper embedding if and only if the inclusion $V\subset \tilde V$ is proper.

A {\bf tautological embedding} of ${\bf GS}(m,\infty)$ (and analogously of ${\bf GO}(m,\infty)$) into ${\bf G}(m)$ is defined for any isomorphism ${\bf GS}(m,\infty)\cong {\bf GS}(W,E,V)$ (where $\dim W=m$) as
\begin{gather*}
\begin{array}{rcl}
\tau_{{\bf GS}(W,E,V)}:{\bf GS}(m,\infty)\cong {\bf GS}(W,E,V) & \hookrightarrow & {\bf G}(W,E,V) \cong {\bf G}(m) \;.\\ 
U & \mapsto  & U
\end{array}
\end{gather*}
This embedding is linear for all $m>0$.

A {\bf tautological embedding} of ${\bf GO}^\eps(\infty,m)$ into ${\bf G}(\infty)$ is defined for any isomorphism ${\bf GO}^\eps(\infty,m)\cong {\bf GO}(W,E,V)$ (where $\dim W=\infty$ and $W$ has codimension $m$ in a maximal isotropic subspace of $V$) as
\begin{gather*}
\begin{array}{rcl}
\tau_{{\bf GO}(W,E,V)}:{\bf GO}^\eps(\infty,m)\cong {\bf GO}(W,E,V) & \hookrightarrow & {\bf G}(W,E,V) \cong {\bf G}(\infty) \;.\\ 
U & \mapsto  & U
\end{array}
\end{gather*}
This embedding is linear for $m>0$.

A {\bf tautological embedding} of ${\bf GS}(\infty,m)$ into ${\bf G}(\infty)$ is defined analogously, and is linear for all $m\geq0$.

An {\bf isotropic extension} of ${\bf G}(m)$ to ${\bf GS}(m,\infty)$, and to ${\bf GS}(\infty,n)$ for $n\geq m$, is defined for a given isotropic subspace $V\subset \tilde V$ and an isotropic basis $F\subset \tilde V$ such that $E:=V\cap F$ is a basis of $V$, as an embedding
\begin{gather*}
\begin{array}{rcl}
\iota:{\bf G}(m)\cong {\bf G}(W,E, V) & \hookrightarrow & {\bf GS}(W,F, \tilde V) \cong {\bf GS}(m,\infty) \\ 
U & \mapsto  & U
\end{array}
\end{gather*}
where $W\subset V$ is a subspace of dimension $m$ and infinite codimension, and respectively as
\begin{gather*}
\begin{array}{rcl}
\iota:{\bf G}(m)\cong {\bf G}(V_\infty,E, V) & \hookrightarrow & {\bf GS}(V_\infty,F, \tilde V) \cong {\bf GS}(\infty,n) \\ 
U & \mapsto  & U
\end{array}
\end{gather*}
where $V_\infty\subset V$ is a subspace of codimension $m$ in $V$ and codimension $n$ in a maximal isotropic subspace of $\tilde V$.

Isotropic extensions of ${\bf G}(m)$ to ${\bf GO}(m,\infty)$, or to ${\bf GO}^\eps(\infty,n)$ with $n\geq m$, are defined analogously. All isotropic extensions are linear. Note that the ind-grassmannians ${\bf GS}(\infty,0)$ and ${\bf GO}^\eps(\infty,0)$ are not targets of isotropic extensions.

Composing tautological embeddings, standard extensions and isotropic extensions, we define ({\bf mixed}) {\bf combinations of standard and isotropic extensions} by analogy with the finite dimensional case.

A {\bf projective ind-space on} (or a {\bf projective space on}) an ind-grassmannian ${\bf X}$ is the image of a linear embedding $\PP^\infty\hookrightarrow{\bf X}$ (respectively, $\PP^k\hw{\bf X}$ for some $k$). An {\bf ind-quadric on} ${\bf X}$ is the image of a linear embedding $Q^\infty\hookrightarrow {\bf X}$. For ${\bf X}\cong {\bf GO}(m,\infty)$ with $m\in\NN\cup\{\infty\}$, a {\bf standard ind-quadric on} ${\bf X}$ is the image of a standard extension $Q^\infty\hookrightarrow {\bf X}$. For ${\bf X}\cong {\bf GS}(m,\infty)$ with $m\in\NN\cup\{\infty\}$, a {\bf standard symplectic projective ind-space on} ${\bf X}$ is the image of a standard extension $\PP^\infty\cong {\bf GS}(1,\infty)\hookrightarrow {\bf X}$.

\begin{prop}
Let ${\bf X}$ and ${\bf Y}$ be linear ind-grassmannians.
\begin{enumerate}
\item[{\rm (i)}] There exists a linear embedding ${\bf X}\stackrel{\varphi}{\hookrightarrow}{\bf Y}$ which factors through a projective ind-space, if and only if ${\bf Y}$ is not isomorphic to ${\bf GS}(\infty,0)$.
\item[{\rm (ii)}] Suppose ${\bf X}$ is not a projective ind-space or an ind-quadric, and fix a standard exhaustion ${\bf X}=\lim\limits_\to X_k$ such that none of the grassmannians $X_k$ is a projective space or a quadric. Let ${\bf X}\stackrel{\varphi}{\hookrightarrow}{\bf Y}$ be a linear embedding obtained as the direct limit of a sequence $X_k\stackrel{\varphi_k}{\hookrightarrow} Y_k$ for a suitable standard exhaustion ${\bf Y}=\lim\limits_\to Y_k$. Then $\varphi$ factors through a projective ind-space or through a standard ind-quadric if and only if the embedding $\varphi_1$ factors through a projective space or a standard ind-quadric.
\end{enumerate}
\end{prop}

\begin{proof}
Part (i) follows from the fact that every linear ind-grassmannian admits a linear embedding into a projective ind-space, together with the observation that there is an infinite-dimensional projective ind-space on every linear ind-grassmannian except ${\bf GS}(\infty,0)$. The maximal projective spaces on ${\bf GS}(\infty,0)$ are projective lines.

For part (ii), suppose first that ${\bf X}$ is not isomorphic to ${\bf GS}(\infty,0)$. We may also assume that ${\bf Y}$ is not a projective ind-space, an ind-quadric, or ${\bf GS}(0,\infty)$. Let us recall that on a grassmannian $Z$ which is not a projective space, a quadric, or $GS(l,V_{2l})$, there are two irreducible families $\mc F_1(Z),\mc F_2(Z)$ of maximal projective spaces, unless $Z$ is a non-spinor orthogonal grassmannians which has one family $\mc F_1(Z)$ of maximal standard quadrics and one family $\mc F_2(Z)$ of maximal projective spaces not contained in a quadric. If $\psi:Z_1\to Z_2$ is a linear embedding of such grassmannians $Z_1,Z_2$, and $\psi$ does not factor through a projective space or through a standard quadric, then $\psi$ separates the families $\mc F_1$ and $\mc F_2$, i.e., after possible relabelling of the families $\mc F_1(Z_2),\mc F_2(Z_2)$ we can assume that $\psi$ maps spaces from $\mc F_j(Z_1)$ to spaces from $\mc F_j(Z_2)$ for $j=1,2$. Note in addition that a projective space of dimension $l\geq 3$ on $Z$ is contained in a unique maximal projective space or in a unique maximal standard quadric on $Z$. 

If $\varphi_1$ factors through a projective space $\PP\subset Y_1$ then $3\leq \dim X_1\leq\dim\PP$ holds, and hence $\PP$ determines a unique maximal projective space or a unique maximal standard quadric $M_k\subset Y_k$ for each $k\geq 1$. Furthermore, either each $M_k$ is a projective space, or each $M_k$ is a quadric. We clearly have $M_k\subset M_{k+1}$, and thus $\PP$ determines a unique maximal projective space or maximal standard quadric $M\subset {\bf Y}$. We claim that $\varphi_k(X_k)\subset M_k$. Indeed, for every $k\geq 1$ there exists $F^k_j\in\mc F_j(X_k)$ for $j=1,2$ such that $F^k_j\cap X_1\in \mc F_j(X_1)$. Then necessarily $\varphi_k(F^k_j)\subset M_k$. This implies that $\varphi_k$ sends both families $\mc F_1(X_k),\mc F_2(X_k)$ to one family, say $\mc F_1(Y_k)$, on $Y_k$. Consequently, we have $\varphi({\bf X})\subset M$. 

Suppose now that $\varphi_1$ does not factor through a projective space or through a standard quadric. Then any two spaces $F_j\in \mc F_j(X_1)$, $j=1,2$, determine unique sequences $\tilde F_j^k\in\mc F_j(Y_k)$ for $k\geq 1$ such that $\tilde F_j^k\cap \varphi({\bf X})\geq 2$ for all $k$ and for $j=1,2$. Thus the image $\varphi({\bf X})$ is not contained in a projective ind-space or an ind-quadric on ${\bf Y}$.

It remains to consider the case $X={\bf GS}(\infty,0)$. Due to Theorems \ref{Theo Embed Same Type} and \ref{Theo Embed Mix Type}, every linear embedding of $GS(m,V_{2m})$ which does not factor through a projective space is a standard extension $GS(m,V_{2m})\hw GS(l,V_{2l})$ or factors through the tautological embedding $GS(m,V_{2m})\hw G(m,V_{2m})$. The case of a standard extension is excluded since ${\bf Y}\not\cong{\bf GS}(\infty,0)$. The remaining options allow us to reduce the proof to the case of ordinary ind-grassmannians.
\end{proof}

For the following proposition, we recall that every linear embedding of a grassmannian $X$ into a quadric factors through one of the embeddings $\kappa_p$ for $p\in I_2(X)$ described in Lemma \ref{Lemma embed in quadric}. Also, every linear embedding $\psi:X\hw Y$ of grassmannians induces a linear embedding of projective spaces $\hat\psi:\PP(V_X)\hw \PP(V_{Y})$ as in (\ref{For hat phi}), and hence a pullback $\hat\psi^*:\CC[V_{Y}]\to \CC[V_{X}]$ on polynomial functions. Furthermore, we have $\hat\psi^*(I_2(Y))\subset I_2(X)$. For a linear ind-grassmannian ${\bf X}$ with a standard exhaustion $X_k\stackrel{\psi_k}{\hw} X_{k+1}$, we obtain the inverse limit $I_2({\bf X})=\lim\limits_{\leftarrow} I_2(X_k)$ consisting of sequences $p_k\in I_2(X_k)$ satisfying $p_k=\hat \psi_k^* p_{k+1}$.

\begin{prop}\label{Prop ind Emb in Q}
Let ${\bf X}\stackrel{\varphi}{\hw} Q^\infty$ be a linear embedding. If $\varphi$ does not factor through a projective ind-space, then $\varphi$ factors as
$$
{\bf X} \stackrel{\kappa_{\bf p}}{\hw} Q^\infty \stackrel{\sigma}{\hw} Q^\infty 
$$
where $\sigma$ is a standard extension, ${\bf p}\in I_2({\bf X})$ is a nonzero element, and $\kappa_{\bf p}:=\lim\limits_\to \kappa_{p_k}$.
\end{prop}

\begin{proof}
Let $\tilde\psi_k: Q^{s_k}\hw Q^{s_{k+1}}$ be standard extensions defining ${\bf Y}:=Q^{\infty}$, such that the embedding $\varphi$ is the direct limit of a sequence of linear embeddings $X_k\stackrel{\varphi_k}{\hw}Q^{s_k}$, none of which factors through a projective space. By Lemma \ref{Lemma embed in quadric}, each $\varphi_k$ factors as $\varphi_k=\sigma_k\circ\kappa_{p_k}$ for a standard extension $\sigma_k:Q^{n_k}\hw Q^{s_k}$ and a nonzero $p_k\in I_2(X_k)$. The relation $p_k=\hat \psi_k^* p_{k+1}$ follows from the observation that $\varphi_{k+1}\circ \psi_k=\tilde\psi_k\circ\varphi_k$, since $p_k$ is equal to the restriction to $V_{X_k}$ of a non-degenerate quadric form defining $Q^{s_k}$. Thus we can form the element ${\bf p}:=\lim\limits_{\leftarrow}p_k\in I_2({\bf X})$. The quadrics $Q^{n_k}$ satisfy 
$$
\tilde\psi_k(\varphi_k(X_k))\subset \tilde\psi_k(\sigma_k(Q^{n_k}))\subset \sigma_{k+1}(Q^{n_{k+1}}) \;,
$$
and we obtain a chain of standard extensions $\eta_k:Q^{n_k}\hw Q^{n_{k+1}}$ fitting into a commutative diagram
\begin{gather*}
\begin{array}{lccccc}
\varphi_{k+1}: & X_{k+1} & \stackrel{\kappa_{p_{k+1}}}{\hw} & Q^{n_{k+1}} & \stackrel{\sigma_{k+1}}{\hw} & Q^{s_{k+1}} \\
 & \psi_k \huw \qquad\quad & & \eta_k \huw \qquad\quad & & \tilde\psi_k \huw \qquad\quad \\
\varphi_k: & X_{k} & \stackrel{\kappa_{p_k}}{\hw} & Q^{n_k} & \stackrel{\sigma_k}{\hw} & Q^{s_{k}}
\end{array}\;.
\end{gather*}
We can now conclude that the embedding $\varphi$ factors through the ind-quadric ${\bf Z}=\lim\limits_{\to} Q^{n_k}\cong Q^\infty$ defined by the embeddings $\eta_k$, as
$$
\varphi: {\bf X} \stackrel{\kappa_{\bf p}}{\hw} {\bf Z}\stackrel{\sigma}{\hw} {\bf Y}
$$
where $\kappa_{\bf p}=\lim\limits_\to \kappa_{p_k}$ and $\sigma=\lim\limits_\to \sigma_k$.
\end{proof}

The next proposition requires some preparation. Let $V$ be an orthogonal space, $W\subset V$ be a maximal isotropic subspace, and $E\subset V$ be an isotropic basis containing a basis of $W$. We assume that the involution $\iota_E$ has no fixed points, so that $V=W\oplus W_*$ where $W_*={\rm span}(\iota_E(E\cap W))$. Then ${\bf S}:={\bf GO}(W,E,V)$ is isomorphic to the spinor ind-grassmannian ${\bf GO}^0(\infty,0)$. Let ${\bf S}\stackrel{\pi_{\bf S}}{\hw} \PP(V_{\bf S})$ be the minimal projective embedding of ${\bf S}$. 

Let $W^{(-2)}\subset W$ be a fixed subspace of codimension $2$, spanned by $E\cap W^{(-2)}$. Then $\PP^1_{{\bf S},U}:=\{U'\in {\bf S}:U\subset U'\}$ for $U\in {\bf GO}^0(W^{(-2)},E,V)$ is a projective line on ${\bf S}$, and every projective line on ${\bf S}$ has this form for some (unique) $U$. Thus the set of projective lines on ${\bf S}$ admits a structure of a linear ind-grassmannian isomorphic to ${\bf GO}^0(\infty,2)$.

\begin{prop}\label{Prop Ind delta}
The map
\begin{gather*}
\begin{array}{rcl}
\delta^0 :{\bf GO}^0(\infty,2)\cong {\bf GO}(W^{(-2)},E,V) & \hookrightarrow & {\bf G}(2,V_{\bf S}) \cong {\bf G}(2) \;.\\
U & \mapsto & \pi_{\bf S}(\PP^1_{{\bf S},U})
\end{array}
\end{gather*}
is a proper linear embedding of linear ind-grassmannians. The embedding $\delta^0$ admits two factorizations:
$$
{\bf GO}^0(\infty,2) \stackrel{\delta^0_S}{\hw} {\bf GS}(2,\infty) \stackrel{\tau_S}{\hw} {\bf G}(2) \quad,\qquad  {\bf GO}^0(\infty,2) \stackrel{\delta^0_O}{\hw} {\bf GO}(2,\infty) \stackrel{\tau_O}{\hw} {\bf G}(2) \;,
$$
where $\tau_S$, $\tau_O$ are the respective tautological embeddings.
\end{prop}

\begin{proof}
The linearity of the map $\delta^0$ follows from a local expression as a direct limit $\delta^0=\lim\limits_{\to} \delta_n$, where $\delta_n:GO(n-2,V_{2n})\hookrightarrow G(2,V_{2^{n-1}})$ are the embeddings defined in (\ref{For special emb delta n}).

The embeddings $\delta^0_S$ and $\delta^0_O$ are obtained as direct limits $\delta^0_S=\lim\limits_{\to} \delta^S_{4k+2}$ and $\delta^0_O=\lim\limits_{\to} \delta^O_{4k}$ of the embeddings given in Proposition \ref{Prop special emb delta 4k2 S 4k O}.
\end{proof}

The compositions of a standard extension ${\bf GO}^1(\infty,1) \stackrel{\sigma}{\hw} {\bf GO}^0(\infty,2)$ with the embeddings $\delta^0, \delta^0_S$, $\delta^0_O$ yield linear embeddings
\begin{gather}\label{For ind delta1}
\begin{array}{rl}
\delta^1 =& \delta^0\circ\sigma :{\bf GO}^1(\infty,1) \hookrightarrow {\bf G}(2) \;,\\
\delta^1_S =& \delta^0_S\circ\sigma :{\bf GO}^1(\infty,1) \hookrightarrow {\bf GS}(2,\infty) \;,\\
\delta^1_O =& \delta^0_O\circ\sigma :{\bf GO}^1(\infty,1) \hookrightarrow {\bf GO}(2,\infty) \;.
\end{array}
\end{gather}

Next we describe the inductive limit of the embeddings $\theta_m^k$ defined in subsection \ref{Sec SpinorGrass}. Let $V$ be an orthogonal vector space of infinite countable dimension, $W\subset V$ be a maximal isotropic subspace, $W^{(-k)}\subset W$ be a subspace of codimension $k$, $E$ be a basis of $W$ containing a basis of $W^{(-k)}$, $\tilde E$ be an isotropic basis of $V$ containing $E$, and $E_*:=i_{\tilde E}(E)\subset \tilde E$.

\begin{prop}\label{Prop Ind theta}
For $k\in\ZZ_{>0}\cup\{\infty\}$ the map 
\begin{gather*}
\begin{array}{rcl}
\theta^k :{\bf G}(k)\cong {\bf G}(W^{(-k)},E,W) & \hookrightarrow & {\bf GO}(W,\tilde E,V)\cong {\bf GO}^1(\infty,0)\\
 U & \mapsto & U\oplus (U^\perp\cap {\rm span}(E_*))
\end{array}
\end{gather*}
is a linear embedding of linear ind-grassmannians.
\end{prop}

\begin{proof}
For any finite $k$, the embedding $\theta^k$ is obtained as the direct limit $\theta^k:=\lim\limits_{\to} \theta^k_m$ of the embeddings $\theta^k_m$ defined in subsection \ref{Sec SpinorGrass}. For $k=\infty$, the embedding $\theta^\infty:{\bf G}(\infty)\hw {\bf GO}^1(\infty,0)$ is defined as the direct limit $\theta^\infty:=\lim\limits_{\to} \theta^{2l}_{4l}$.
\end{proof}

\subsection{Classification of linear embeddings}

\begin{theorem}\label{Theo ind Embed Same Type} {\rm (Linear embeddings between linear ind-grassmannians of the same type)}

Let $\varphi:{\bf X}\to{\bf Y}$ be a linear embedding between linear ind-grassmannains, which does not factor through a projective ind-space.

\begin{enumerate}
\item[\rm (i)] If $\varphi$ has the form $\varphi:{\bf G}(m) \hookrightarrow {\bf G}(l)$ for $l,m\in\NN\cup\{\infty\}$, then $\varphi$ is a standard extension.

\item[\rm (ii.1)] If $\varphi$ has the form $\varphi:{\bf GO}(m,\infty) \hookrightarrow {\bf GO}(l,\infty)$ for $m,l\in\NN\cup\{\infty\}$, then there are the following options:

-- $\varphi$ is a standard extension and $m\leq l$ holds;

-- $\varphi$ is a combination of standard and isotropic extensions;

-- $\varphi$ factors through a standard ind-quadric.

\item[\rm (ii.2)] If $\varphi$ has the form $\varphi:{\bf GO}^\varepsilon(\infty,c) \hookrightarrow {\bf GO}^{\varepsilon'}(\infty,d)$ for $c,d\in\NN$, $d\ne0$, $\eps,\eps'\in\{0,1\}$, then $c\ne 0$ holds and there are the following options:

-- $\varphi$ is a standard extension and $d\geq c$ holds;

-- $\varphi$ is a combination of standard and isotropic extensions;

-- $c=1$ and $\varphi$ factors as $\iota\circ\sigma\circ\delta^1$, or as $\sigma_O\circ \delta^1_O$;

-- $c=2$ and $\varphi$ factors as $\iota\circ\sigma\circ\delta^0$, or as $\sigma_O\circ \delta^0_O$.

\item[\rm (ii.3)] If $\varphi$ has the form $\varphi:{\bf GO}^\varepsilon(\infty,c) \hookrightarrow {\bf GO}(l,\infty)$ for $c\in\NN$, $l\in\NN\cup\{\infty\}$, then there are the following options:

-- $\varphi$ factors through a standard ind-quadric;

-- $c\ne 0$ and $\varphi$ is a combination of standard and isotropic extensions;

-- $c=1$ and $\varphi$ factors as $\iota\circ\sigma\circ\delta^1$, or as $\sigma_O\circ \delta^1_O$;

-- $c=2$ and $\varphi$ factors as $\iota\circ\sigma\circ\delta^0$, or as $\sigma_O\circ \delta^0_O$.

\item[\rm (ii.4)] If $\varphi$ has the form $\varphi:{\bf GO}(m,\infty)\hw {\bf GO}^\varepsilon(\infty,d)$ for $m\in\ZZ_{\geq1}\cup\{\infty\}$, $d\in\ZZ_{\geq1}$ and $\varepsilon\in\{0,1\}$, then $\varphi$ is a combination of standard and isotropic extensions.

\item[\rm (iii)] If $\varphi$ has the form $\varphi:{\bf GS}(m,c) \hookrightarrow {\bf GS}(l,d)$ for $m,l,c,d\in\NN\cup\{\infty\}$, $\{\infty\}\in\{m,c\}\cap\{l,d\}$, then there are the following options: 

-- $\varphi$ is a standard extension;

-- $\varphi$ is a combination of standard and isotropic extensions.
\end{enumerate}
\end{theorem}

\begin{theorem}\label{Theo ind Embed Mix Type} {\rm (Linear embeddings between linear ind-grassmannians of different types)}

Let $\varphi:{\bf X}\to{\bf Y}$ be a linear embedding between linear ind-grassmannains, which does not factor through a projective ind-space.

\begin{enumerate}
\item[\rm (i.1)] If $\varphi$ has the form $\varphi:{\bf G}(m)\hw {\bf GO}(l,\infty)$, then $\varphi$ is a composition
$$
{\bf G}(m) \stackrel{\sigma}{\hookrightarrow} {\bf G}(l) \stackrel{\iota}{\hookrightarrow} {\bf GO}(l,\infty)
$$
where $\sigma$ is a standard extension and $\iota$ is isotropic, or $\varphi$ factors through a standard ind-quadric.
\item[\rm (i.2)] If $\varphi$ has the form $\varphi:{\bf G}(m)\hw {\bf GO}^\varepsilon(\infty,l)$ with $l>0$, then $\varphi$ is a composition
$$
{\bf G}(m) \stackrel{\sigma}{\hookrightarrow} {\bf G}(l) \stackrel{\iota}{\hookrightarrow} {\bf GO}^\varepsilon(\infty,l)
$$
where $\sigma$ is a standard extension and $\iota$ is isotropic.
\item[\rm (ii)] If $\varphi$ has the form $\varphi:{\bf G}(m)\hw {\bf GS}(l)$, then $\varphi$ is a composition
$$
{\bf G}(m) \stackrel{\sigma}{\hookrightarrow} {\bf G}(l) \stackrel{\iota}{\hookrightarrow} {\bf GS}(l,\infty)
$$
where $\sigma$ is a standard extension and $\iota$ is isotropic.
\item[\rm (iii.1)] If $\varphi$ has the form $\varphi:{\bf GO}(m,\infty)\hookrightarrow {\bf G}(l)$, then $\varphi$ is a composition
$$
{\bf GO}(m,\infty) \stackrel{\tau}{\hookrightarrow} {\bf G}(m) \stackrel{\sigma}{\hookrightarrow} {\bf G}(l)
$$
where $\tau$ is a tautological embedding and $\sigma$ is a standard extension.
\item[\rm (iii.2)] If $\varphi$ has the form $\varphi:{\bf GO}^\varepsilon(\infty,m)\hookrightarrow {\bf G}(l)$, then $\varphi$ is a composition
$$
{\bf GO}^\varepsilon(\infty,m) \stackrel{\tau}{\hookrightarrow} {\bf G}(\infty) \stackrel{\sigma}{\hookrightarrow} {\bf G}(\infty)
$$
where $\tau$ is a tautological embedding and $\sigma$ is a standard extension, or $m=2-\eps$ holds and the embedding $\varphi$ is a composition
$$
{\bf GO}^\eps(\infty,2-\eps) \stackrel{\delta^\eps}{\hookrightarrow} {\bf G}(2) \stackrel{\sigma}{\hookrightarrow} {\bf G}(l)
$$
where $\delta^\eps$ is given in Proposition \ref{Prop Ind delta} and $\sigma$ is a standard extension.

\item[\rm (iv)] If $\varphi$ has the form $\varphi:{\bf GS}(m,c)\hookrightarrow {\bf G}(l)$, then $\varphi$ factors as a composition
$$
{\bf GS}(m,\infty) \stackrel{\tau}{\hookrightarrow} {\bf G}(m) \stackrel{\sigma}{\hookrightarrow} {\bf G}(l) \;,
$$
where $\tau$ is a tautological embedding and $\sigma$ is a standard extension.

\item[\rm (v.1)] If $\varphi$ has the form $\varphi:{\bf GO}(m,\infty)\hookrightarrow {\bf GS}(l,d)$, then $\varphi$ is a mixed combination of standard and isotropic extensions.

\item[\rm (v.2)] If $\varphi$ has the form $\varphi:{\bf GO}^\eps(\infty,m)\hookrightarrow {\bf GS}(l,d)$, then $l=d=\infty$ and $\varphi$ is a mixed combination of standard and isotropic extensions, or $m=2-\eps$ holds and the embedding $\varphi$ is a composition
$$
{\bf GO}^\eps(\infty,2-\eps) \stackrel{\delta_S^\eps}{\hookrightarrow} {\bf GS}(2,\infty) \stackrel{\psi}{\hookrightarrow} {\bf GS}(l,d)\;,
$$
where $\delta_S^\eps$ is given in Proposition \ref{Prop Ind delta} and $\psi$ is either a standard extension or a combination of standard and isotropic extensions.

\item[\rm (vi)] If $\varphi$ has the form $\varphi:{\bf GS}(m,c)\hookrightarrow {\bf GO}(l,\infty)$, then $\varphi$ is a mixed combination of standard and isotropic extensions, or factors through a standard ind-quadric. If $\varphi$ has the form $\varphi:{\bf GS}(m,c)\hookrightarrow {\bf GO}^\eps(\infty,d)$, then $\varphi$ is a mixed combination of standard and isotropic extensions.
\end{enumerate}
\end{theorem}

\begin{proof}[Proof of Theorems \ref{Theo ind Embed Same Type} and \ref{Theo ind Embed Mix Type}.]
Let $\varphi:{\bf X}\hw {\bf Y}$ be obtained as the direct limit of a sequence of linear embeddings $\varphi_k:X_k\hw Y_k$ for appropriate standard exhaustions ${\bf X}=\lim\limits_\to X_k$ and ${\bf Y}=\lim\limits_\to Y_k$. Thus we have a commutative diagram
\begin{gather}\label{For ind seq phi psi}
\begin{array}{ccccccc}
\dots & \stackrel{\psi_{k-1}}{\hw} & X_k & \stackrel{\psi_k}{\hw} & X_{k+1} & \stackrel{\psi_{k+1}}{\hw} & \dots \\
 &  & \varphi_{k}\hdw\qquad\qquad &  & \varphi_{k+1}\hdw\qquad\qquad\quad &  & \\
\dots & \stackrel{\tilde\psi_{k-1}}{\hw} & Y_k & \stackrel{\tilde\psi_k}{\hw} & Y_{k+1} & \stackrel{\tilde\psi_{k+1}}{\hw} & \dots
\end{array}
\end{gather}
where $\psi_k$ and $\tilde\psi_k$ are standard extensions.

We start by addressing part (i) of Theorem \ref{Theo ind Embed Same Type}. Here ${\bf X}:={\bf G}(m)=\lim\limits_{\to} G(m_k,V_{n_k})$ and ${\bf Y}:={\bf G}(l)=\lim\limits_{\to} G(l_k,\tilde V_{s_k})$. Since $\varphi$ does not factor through a projective ind-space, Theorem \ref{Theo Embed Same Type} implies that the embeddings $\varphi_k$ exhausting $\varphi$ are standard extensions. Furthermore, we reduce to the case where all embeddings in the diagram (\ref{For ind seq phi psi}) are strict standard extensions, by omitting if necessary some embeddings $\varphi_k$ and/or replacing each $V_{n_k}$ by $V_{n_k}^*$ and/or replacing each $V_{n_k}$ by $V_{n_k}^*$.

To show that $\varphi$ is a standard extension as in (\ref{For ind-standard ext oridnary}) we need to construct a global expression for $\varphi$ from the local data of the above exhaustions. Recall first that every strict standard extension of ordinary grassmannians $G(m,V_n)\stackrel{\sigma}{\hw}G(l,V_s)$ corresponds to an inclusion of vector spaces $\nu:V_n\hw V_s$. In fact, the inclusion $\nu$ is determined by $\sigma$. Indeed, the pullback $\sigma^*\mc S^*_{G(l,V_s)}$ splits as a direct sum of $\mc S^*_{G(m,V_n)}$ and a trivial bundle, and the inclusion $\nu$ is dual to the surjection 
$$
V_s^*\cong H^0(G(l,V_s),\mc S^*_{G(l,V_s)})\stackrel{\sim}{\to} H^0(G(m,V_n),\sigma^*\mc S^*_{G(l,V_s)}) \to H^0(G(m,V_n),\mc S^*_{G(m,V_n)})\cong V_n^*\;.
$$
By applying this to the strict standard extensions in the diagram (\ref{For ind seq phi psi}), we obtain a commutative diagram of inclusions $\nu_k,\mu_k,\tilde\mu_k$ of vector spaces 
\begin{gather}\label{For ind seq v n v s}
\begin{array}{ccccccc}
\dots & \stackrel{\mu_{k-1}}{\hw} & V_{n_k} & \stackrel{\mu_k}{\hw} & V_{n_{k+1}} & \stackrel{\mu_{k+1}}{\hw} & \dots \\
 &  & \nu_{k}\hdw\qquad\qquad &  & \nu_{k+1}\hdw\qquad\qquad\quad &  & \\
\dots & \stackrel{\tilde\mu_{k-1}}{\hw} & \tilde V_{s_k} & \stackrel{\tilde\mu_k}{\hw} & \tilde V_{s_{k+1}} & \stackrel{\tilde\mu_{k+1}}{\hw} & \dots 
\end{array}\;.
\end{gather}
Now the direct limit $\nu:=\lim\limits_{\to}\nu_k$ of the inclusions $\nu_k:V_{n_k}\hw \tilde V_{s_k}$ is an inclusion of countable-dimensional vector spaces
$$
\nu:V:=\lim\limits_\to V_{n_k}\hw \tilde V:=\lim\limits_\to \tilde V_{s_k} \;.
$$

To complete the global expression of $\varphi$ we need to find suitable subspaces and bases of $V$ and $\tilde V$. For each $k$, the embedding $\varphi_k$ is written as
$$
\varphi_k(U)=\nu_k(U)\oplus W^{(k)}\;, \quad{\rm where}\quad W^{(k)}:= \bigcap\limits_{U\in G(m_k,V_{n_k})} \varphi_k(U)\;.
$$
Since each $W^{(k)}$ is naturally a subspace of $\tilde V$, we can define the intersection
$$
\ul{W}^{(k)}:=\bigcap\limits_{j\geq k} W^{(j)} \;.
$$
Then $\ul{W}^{(k-1)}\subset \ul{W}^{(k)}$ and the direct limit $W':=\lim\limits_{\to}\ul{W}^{(k)}$ is identified with a subspace of $\tilde V$ satisfying
$$
\nu(V)\cap W'=0 \quad,\quad W'=\bigcap\limits_{U\in{\bf X}}\varphi(U) \;.
$$

Next, we fix a sequence $U_{m_k}\in G(m_k,V_{n_k})$ such that $\psi_k(U_{m_k}) = U_{m_{k+1}}$ for all $k$, and define subspaces $W\subset V$ and $\tilde W\subset \tilde V$ by setting $W:=\lim\limits_{\to} U_{m_k}$ and $\tilde W:=\nu(W)\oplus W'$. It remains to construct relevant bases of $V$ and $\tilde V$. We start with a basis of $E_W$ of $W$ such that $E_W\cap U_{m_k}$ is a basis of $U_{m_k}$ for every $k$. Then we extend $E_W$ to a basis $E$ of $V$ such that $E\cap V_{n_k}$ is a basis of $V_{n_k}$. We fix a basis $E'$ of $W'$ such that $E'\cap \ul{W}^{(k)}$ is a basis of $\ul{W}^{(k)}$ for all $k$. Then $\nu(E)\sqcup E'$ is a basis of $\nu(V)\oplus W'$, and $\nu(E_W)\sqcup E'$ is a basis of $\tilde W$. We extend $\nu(E)\sqcup E'$ to a basis $\tilde E$ of $\tilde V$.

We have now constructed identifications ${\bf X}={\bf G}(W,E,V)$ and ${\bf Y}= {\bf G}(\tilde W,\tilde E,\tilde V)$ such that the embedding $\varphi$ is written as
\begin{gather}\label{For ind standard ext in proof}
\begin{array}{rcl}
\varphi: {\bf G}(m)= {\bf G}(W,E,V) & \hw & {\bf G}(\tilde W,\tilde E,\tilde V)= {\bf G}(l)\;.\\
 U & \mapsto & \nu(U)\oplus W'
\end{array}
\end{gather}
Hence $\varphi$ is a standard extension. This completes the proof of part (i) of Theorem \ref{Theo ind Embed Same Type}.

To address the linear embeddings between orthogonal ind-grassmannians (parts (ii.1)-(ii.4) of Theorem \ref{Theo ind Embed Same Type}), we first need to classify linear embeddings between ordinary ind-grassmannians and orthogonal ind-grassmannians (parts (i.1),(i.2),(iii.1),(iii.2) of Theorem \ref{Theo ind Embed Mix Type}). 

For part (i.1) of Theorem \ref{Theo ind Embed Mix Type}, the embedding $\varphi$ has the form
$$
\varphi:{\bf G}(m)\cong {\bf X}\hw {\bf Y}\cong {\bf GO}(l,\infty) \;.
$$
By Theorem \ref{Theo Embed Mix Type} each of the embeddings $\varphi_k$ factors as a composition of a standard extension $\sigma_k$ and an isotropic extension $\iota_k$:
$$
\varphi_k: X_k=G(m_k,V_{n_k})\stackrel{\sigma_k}{\hw} G(l_k,\hat V_{r_k}) \stackrel{\iota_k}{\hw} GO(l_k,\tilde V_{s_k})=Y_k
$$
for appropriate spaces $\hat V_{r_k}$. We take the canonical choice for $\hat V_{r_k}$ given as the sum of the subspaces $\varphi_k(U)$ for $U\in X_k$, which is an isotropic subspace of $\tilde V_{s_k}$.

Similarly to a standard extension of ordinary grassmannians, the standard extensions $\tilde\psi_k:GO(l_k,\tilde V_{s_k})\hw GO(l_{k+1},\tilde V_{s_{k+1}})$ induce inclusions $\tilde\mu_k:\tilde V_{s_k}\hw \tilde V_{s_{k+1}}$. We have $\tilde\mu_k(\hat V_{r_k})\subset \hat V_{r_{k+1}}$ due to the commutativity relation $\varphi_{k+1}\circ\psi_k=\tilde\psi_k\circ\varphi_k$. Therefore the restriction of $\hat\psi_k$ to $G(l_k,\hat V_{r_k})$ defines a strict standard extension $\eta_k: G(l_k,\hat V_{r_k})\hw G(l_k,\hat V_{r_{k+1}})$. It follows that $\varphi$ factors through the ind-grassmannian ${\bf Z}:=\lim\limits_{\to} G(l_k,\hat V_{r_k})\cong {\bf G}(l)$ as
$$
{\bf X}\stackrel{\sigma}{\hw} {\bf Z} \stackrel{\iota}{\hw} {\bf Y}
$$
where $\sigma:=\lim\limits_\to \sigma_k$, and $\iota$ is an isotropic extension. To complete the argument, we can apply part (i) of Theorem \ref{Theo ind Embed Same Type} to deduce that $\sigma$ is a standard extension as required.

The proofs of the remaining statements of Theorems \ref{Theo ind Embed Same Type} and \ref{Theo ind Embed Mix Type} follow the same general lines, and we just give the details in one least straightforward case: statement (iii.2) for ${\bf X}\cong{\bf GO}^0(\infty,2)$ and ${\bf Y}\cong{\bf G}(l)$ with $l\in\ZZ_{\geq 2}\cup\{\infty\}$. 

Here $X_k\cong GO(n_k-2, V_{2n_k})$ and $Y_k\cong G(l_k,\tilde V_{r_k})$. Since none of the embeddings $\varphi_k$ factors through a projective space, by Theorem \ref{Theo Embed GOcodim2} there are two options for each $\varphi_k$: it factors as a composition
$$
GO(n_k-2,V_{2n_k}) \stackrel{\tau_{k}}{\hw} G(n_k-2,V_{2n_k}) \stackrel{\sigma^1_k}{\hw} G(l_k,\tilde V_{r_k})
$$
where $\tau_{k}$ is the tautological embedding of $X_k$ and $\sigma^1_k$ is a standard extension, or $\varphi_k$ factors as a composition
$$
GO(n_k-2,V_{2n_k}) \stackrel{\delta_{n_k}}{\hw} G(2,V_{S_{2n_k}}) \stackrel{\sigma^2_k}{\hw} G(l_k,\tilde V_{r_k})
$$
where $S_{2n_k}:=GO(n_k,V_{2n_k})$, $\delta_{n_k}$ is the linear embedding defined in (\ref{For special emb delta n}), and $\sigma^2_k$ is a standard extension. Clearly we can assume that one of the these two options for $\varphi_k$ holds simultaneously for all $k$. We consider the two cases separately.

Suppose first that $\varphi_k$ factors through $\tau_k$ for all $k$. In this case we necessarily have $l=\lim\limits_{k\to\infty} l_k=\infty$. For each $k$ there is an isomorphism $V_{2n_k}\cong V_{2n_k}^*$ provided by the respective symmetric bilinear form. This allows us to assume that the standard extension $\sigma^1_k: G(n_k-2,V_{2n_k})\hw G(l_k,\tilde V_{r_k})$ is strict for all $k$, and so are the standard extensions $\tilde\psi_k$. Let $\mu_k:V_{2n_k}\hw V_{2n_{k+1}}$ be the inclusions associated to the standard extensions $\psi_k$. Then there are strict standard extensions $\xi_k$ such that the following diagram is commutative for all $k$: 
\begin{gather*}
\begin{array}{lccccc}
\varphi_{k+1}: & GO(n_{k+1}-2,V_{2n_{k+1}}) & \stackrel{\tau_{k+1}}{\hw} & G(n_{k+1}-2,V_{2n_{k+1}}) & \stackrel{\sigma^1_{k+1}}{\hw} & G(l_{k+1},\tilde V_{r_{k+1}}) \\
 & \psi_k \huw \qquad & & \xi_k \huw \qquad & & \tilde\psi_k \huw \qquad \\
\varphi_k: & GO(n_{k}-2,V_{2n_{k}}) & \stackrel{\tau_{k}}{\hw} & G(n_{k}-2,V_{2n_k}) & \stackrel{\sigma^1_{k}}{\hw} & G(l_{k},\tilde V_{r_{k}}).
\end{array}
\end{gather*}

Set $V:=\lim\limits_\to V_{2n_k}$, and let $W^{(-2)}:=\lim\limits_\to U_{n_k-2}$ be the subspace of $V$ defined by a chain $U_{n_k-2}\in X_k$ such that $\psi_k(U_{n_k-2})=U_{n_{k+1}-2}$. Let $E\subset V$ be a isotropic basis containing a basis of each of the subspaces $V_{2n_k}$, $U_{n_k-2}$ for all $k$. Then the direct limit of the tautological embeddings $\tau_k$ admits the following global expression:
\begin{gather*}
\begin{array}{rcl}
\tau_{\bf X}=\lim\limits_\to\tau_k:{\bf X}= {\bf GO}(W^{-2},E,V) & \hw & {\bf G}(W,E,V) \cong \lim\limits_\to G(n_k-2,V_{2n_k}) = {\bf G}(\infty) \;. \\
 U & \mapsto & U
\end{array}
\end{gather*}

Since each $\sigma^1_k$ is a strict standard extension, the corresponding inclusion $\nu_k:V_{2n_k}\hw \tilde V_{r_k}$ is well defined, and we have an inclusion of direct limits $\nu:=\lim\limits_{\to} \nu_k:V\hw \tilde V:=\lim\limits_\to \tilde V_{r_k}$. The argument used for part (i) of Theorem \ref{Theo ind Embed Same Type} implies that the direct limit $\sigma^1:=\lim\limits_\to \sigma^1_k:{\bf G}(\infty)\to{\bf G}(\infty)$ is a standard extension of the form
\begin{gather*}
\begin{array}{rcl}
\sigma^1=\lim\limits_\to\sigma^1_k:{\bf G}(W,E,V) & \hw & {\bf G}(\tilde W,\tilde E,\tilde V) = {\bf Y} \\
 U & \mapsto & U\oplus W'
\end{array}
\end{gather*}
for some suitable subspace $W'\subset \tilde W$ and basis $E$. We conclude that the initial embedding $\varphi$ factors as $\varphi=\sigma^1\circ\tau_{\bf X}$ as asserted.

Next, suppose that $\varphi_k$ factors through $\delta_{n_k}$ for all $k$. We claim that there exist standard extensions $\eta_k$ making the following diagram commutative for every $k$:
\begin{gather}\label{For eta}
\begin{array}{lccccc}
\varphi_{k+1}: & GO(n_{k+1}-2,V_{2n_{k+1}}) & \stackrel{\delta_{n_{k+1}}}{\hw} & G(2,V_{2^{n_{k+1}-1}}) & \stackrel{\sigma^2_{k+1}}{\hw} & G(l_{k+1},\tilde V_{r_{k+1}}) \\
 & \psi_k \huw \qquad & & \eta_k \huw \qquad & & \tilde\psi_k \huw \qquad \\
\varphi_k: & GO(n_{k}-2,V_{2n_{k}}) & \stackrel{\delta_{n_k}}{\hw} & G(2,V_{2^{n_{k}-1}}) & \stackrel{\sigma^2_{k}}{\hw} & G(l_{k},\tilde V_{r_{k}}).
\end{array}
\end{gather}
Note first that, as in the previous case, we can modify the exhaustions of ${\bf X}$ and ${\bf Y}$ so that the standard extensions $\sigma^2_k$ and $\tilde\psi_k$ are strict for all $k$. Each standard extension $\psi_k: GO(n_{k}-2,V_{2n_{k}}) \hw GO(n_{k+1}-2,V_{2n_{k+1}})$ has the form $\psi_k(U)= \mu_k(U)\oplus U'$ for a unique $U'\in GO(n_{k+1}-n_{k},V_{2n_k})$. Thus $\psi_k$ determines a standard extension 
$$
\zeta_k: S_{2n_k}:= GO(n_k,V_{2n_k}) \hw GO(n_{k+1},V_{2n_{k+1}})=:S_{2n_{k+1}} \;,\quad \zeta_k(U)= \mu_k(U)\oplus U'\;.
$$

Let $\hat\zeta_k:V_{S_{2n_k}}\hw V_{S_{2n_{k+1}}}$ be the inclusion corresponding to the embedding $\zeta_k$. Then $\hat\zeta_k$ determines a strict standard extension 
$$
\eta_k:G(2,V_{S_{2n_{k}}})\hw G(2,V_{S_{2n_{k}}}) \;,\quad \eta_k(U):=\hat\zeta_k(U)\;.
$$
Furthermore, we have $\delta_{n_{k+1}}\circ \psi_k (U)=\hat\zeta_k( \delta_k(U))$, which implies that the commutation relation $\eta_k\circ\delta_{n_k}=\delta_{n_{k+1}}\circ \psi_k$ is satisfied for every $k$. Next, the construction from Case 2 of the proof of Theorem \ref{Theo Embed GOcodim2},(i), provides a linear embedding $\pi_k:S_{n_k}\hw\PP(V_{r_k})$ such that the standard extension $\sigma^2_{k}$ is induced by $\pi_k$. Let $\hat \pi_k:V_{S_{n_k}}\hw V_{r_k}$ be the inclusion induced by $\pi_k$. Then the relation $\varphi_{k+1}\circ\psi_k=\tilde\psi_k\circ\varphi_k$ implies $\hat\pi_{k+1}\circ\hat\zeta_k=\tilde\mu_k\circ\hat\pi_k$, and hence $\tilde\psi_k\circ\sigma^2_{k}=\sigma^2_{k+1}\circ\eta_k$. This shows that the diagram (\ref{For eta}) is commutative.

The existence of the standard extensions $\eta_k$ implies that the embedding $\varphi$ factors through the linear ind-grassmannian ${\bf G}(2):=\lim\limits_\to G(2,V_{S_{2n_k}})$ as $\varphi=\delta \circ \sigma^2$ where $\delta:=\lim\limits_\to \delta_{n_k}$ and $\sigma^2:=\lim\limits_\to \sigma^2_k$. Global expressions for $\delta$ and $\sigma^1$ are constructed respectively in Proposition \ref{Prop Ind delta} and in the proof of part (i) of Theorem \ref{Theo ind Embed Same Type}. The classification of linear embeddings ${\bf GO}^0(\infty,2)\hw {\bf G}(l)$ (with $2\leq l\leq \infty$) is complete.

All remaining cases in the proof of Theorems \ref{Theo ind Embed Same Type} and \ref{Theo ind Embed Mix Type} are left to the reader.
\end{proof}

\begin{prop}
Let $Y$ be a linear ind-grassmannian not isomorphic to ${\bf GO}^1(\infty,0)$. Then every linear embedding ${\bf GO}^1(\infty,0)\hw Y$ factors through a projective space, or possibly through a standard ind-quadric in case $Y$ isomorphic to ${\bf GO}(l)$ or ${\bf GO}(\infty,\infty)$.
\end{prop}

\begin{proof}
The statement follows from Proposition \ref{Prop embed spinor through proj}.
\end{proof}

Theorems \ref{Theo ind Embed Same Type} and \ref{Theo ind Embed Mix Type} yield the following.

\begin{coro}
Let ${\bf X}\stackrel{\varphi}{\hw}{\bf Y}$ be a proper linear embedding of linear ind-grassmannians which does not factor through a projective ind-space or an ind-quadric. Assume furthermore that ${\bf Y}$ is not a spinor ind-grassmannian. Then the following is a complete list of the possible choices for ${\bf X}$ and ${\bf Y}$:
\begin{enumerate}
\item[\rm (a)] $\varphi:{\bf G}(m)\hw {\bf G}(l)$ for $2\leq m\leq l\leq \infty$;\\

\item[\rm (b.1)] $\varphi:{\bf GO}(m,\infty)\hw {\bf GO}(l,\infty)$  $2\leq m\leq l\leq \infty$;\\

\item[\rm (b.2)] $\varphi:{\bf GO}(m,\infty)\hw {\bf GO}^\eps(\infty,l)$ for $2\leq m\leq l< \infty$;\\

\item[\rm (b.3)] $\varphi:{\bf GO}^\eps(\infty,m)\hw {\bf GO}(\infty,\infty)$ for $0<m<\infty$;\\

\item[\rm (b.2)] $\varphi:{\bf GO}^\eps(\infty,m)\hw {\bf GO}^{\eps'}(\infty,l)$ for $0< m\leq l<\infty$;\\

\item[\rm (c.1)] $\varphi:{\bf GS}(m,\infty)\hw {\bf GS}(l,d)$ for $2\leq m\leq \min\{l,d\}$;\\

\item[\rm (c.1)] $\varphi:{\bf GS}(\infty,m)\hw {\bf GS}(\infty,l)$ for $0\leq m\leq l\leq \infty$;\\

\item[\rm (i.1)] $\varphi:{\bf G}(m)\hw {\bf GO}(l,\infty)$ for $2\leq m\leq l$;\\

\item[\rm (i.2)] $\varphi:{\bf G}(m)\hw {\bf GO}^\varepsilon(\infty,l)$ for $2\leq m \leq l<\infty$;\\

\item[\rm (iii.1)] $\varphi:{\bf GO}(m,\infty)\hookrightarrow {\bf G}(l)$ for $2\leq m\leq l$;\\

\item[\rm (iii.2)] $\varphi:{\bf GO}^\varepsilon(\infty,m)\hookrightarrow {\bf G}(l)$ for $0<m\leq \infty$ and $l=\infty$, or for $m=2-\eps \leq l-\eps$;\\

\item[\rm (iv.1)] $\varphi:{\bf GS}(m,\infty)\hookrightarrow {\bf G}(l)$ for $2\leq m\leq l\leq\infty$;\\

\item[\rm (iv.2)] $\varphi:{\bf GS}(\infty,m)\hookrightarrow {\bf G}(\infty)$ for $0\leq m\leq l=\infty$;\\

\item[\rm (v.1)] $\varphi:{\bf GO}(m,\infty)\hookrightarrow {\bf GS}(l,d)$ for $2\leq m\leq \min\{l,d\}\leq \infty$;\\

\item[\rm (v.2)] $\varphi:{\bf GO}^\eps(\infty,m)\hookrightarrow {\bf GS}(l,d)$ for $l=d=\infty$, or for $m=2-\eps\leq \min\{l,d\}-\eps$;\\

\item[\rm (vi.1)] $\varphi:{\bf GS}(m,c)\hookrightarrow {\bf GO}(l,\infty)$ for $2\leq m\leq l\leq \infty$ and $0\leq c\leq\infty$;\\

\item[\rm (vi.1)] $\varphi:{\bf GS}(m,c)\hookrightarrow {\bf GO}^\eps(\infty,l)$ for $2\leq m\leq l<\infty$ and $0\leq c\leq\infty$.
\end{enumerate}
\end{coro}

It is well known, \cite{Dimitrov-Penkov}, that every linear ind-grassmannian ${\bf X}$ admits a transitive action of an ind-group ${\bf G}$ obtained as a direct limit $\lim\limits_{\to} G_{X_k}$ where $X_k\stackrel{\sigma_k}{\hw} X_{k+1}$ is an exhaustion of ${\bf X}$ by standard extensions. Furthermore, the image of the corresponding homomorphism $G_{X_k}\stackrel{f_k}{\to} G_{X_{k+1}}$ is an $\mc{LS}$-subgroup for each $k$.

\begin{coro}\label{Theo ind-Equivar}
Let ${\bf Y}$ be a non-spinor linear ind-grassmannian. Every linear embedding of linear ind-grassmannians ${\bf X}\stackrel{\varphi}{\hw}{\bf Y}$, which does not factor through a projective ind-space or a standard ind-quadric, is equivariant with respect to a homomorphism of ind-groups ${\bf G}\stackrel{f}{\to} {\bf H}$ where ${\bf G}$ acts transitively on ${\bf X}$ and ${\bf H}$ acts transitively on ${\bf Y}$.
\end{coro}

\begin{proof}
The equivariance properties of linear embeddings of grassmannians show in Theorem \ref{Theo Equivar gen}, Proposition \ref{Prop special emb delta 4k2 S 4k O} and Corollary \ref{Coro all equivar in Gls}, imply that for linear ind-grassmannians the claimed equivariance holds for tautological embeddings, standard extensions, isotropic extensions, as well as for the embeddings $\delta^\eps,\delta^\eps_O,\delta^\eps_S$ with $\eps\in\{0,1\}$. This implies the result for all linear embeddings of ind-grassmannians, due to the classification in Theorems \ref{Theo ind Embed Same Type} and \ref{Theo ind Embed Mix Type}.
\end{proof}

\vspace{0.2cm}

\noindent Ivan Penkov\\
Constructor University, 28759 Bremen, Germany\\
E-mail address: ipenkov@constructor.university

\vspace{0.2cm}

\noindent Valdemar Tsanov\\
Institute of Mathematics and Informatics, Bulgarian Academy of Sciences,\\
Bulgaria, Sofia 1113, Acad. G. Bonchev Str., Bl. 8\\
and Constructor University, 28759 Bremen, Germany. E-mail address: valdemar.tsanov@math.bas.bg\\

\end{document}